\newcommand\Modified{1.0; October 17th, 2013}
\documentclass[11pt]{amsart}
\usepackage{amsfonts, amsmath,amscd, amssymb, latexsym, mathrsfs, mathtools, slashed, stmaryrd, verbatim, wasysym }
\usepackage[all]{xy}
\usepackage{hyperref}
\usepackage{graphicx}
\newcommand\datver[1]{\def\datverp%
{\par\boxed{\boxed{\text{Version: #1; Run: \today}}}}}

\newtheorem{theorem}{Theorem}

\newtheorem{definition}[theorem]{Definition}
\newtheorem{corollary}[theorem]{Corollary}

\newtheorem{lemma}[theorem]{Lemma}

\newtheorem{proposition}[theorem]{Proposition}
\newtheorem{remark}[theorem]{Remark}

\numberwithin{equation}{section}
\numberwithin{theorem}{section}

\newtheorem{assume}{Assumption}

\usepackage{color}
\definecolor{darkgreen}{cmyk}{1,0,1,.2}
\definecolor{m}{rgb}{1,0.1,1}

\renewcommand{\bar}{\overline}

\renewcommand{\hat}[1]{\widehat{#1}}
\newcommand{\rest}[1]{\big\rvert_{#1}} 

\newcommand{\wt}[1]{\widetilde{#1}}

\newcommand{\Coll}{\mathrm{Coll}} 
\newcommand\eps\varepsilon

\newcommand\lra{\longrightarrow}
\newcommand\xlra[1]{\xrightarrow{\phantom{x} #1 \phantom{x}}}

\newcommand\pa{\partial}

\newcommand\spec{\mathrm{spec}}
\newcommand\Spec{\mathrm{Spec}}

\newcommand\ie{\operatorname{ie}}
\newcommand\iie{\operatorname{iie}}

\newcommand\Iie{{}^{\iie}}

\newcommand \R {\mathbb{R}}

\newcommand\dCI{\dot{\mathcal{C}}^{\infty}}

\newcommand\CI{{\mathcal{C}}^{\infty}}
\newcommand\CIc{{\mathcal{C}}^{\infty}_c}
\newcommand\CmI{{\mathcal{C}}^{-\infty}}

\newcommand{\lrpar}[1]{\left( #1 \right)}
\newcommand{\lrspar}[1]{\left[ #1 \right]}
\newcommand\ang[1]{\left\langle #1 \right\rangle}
\newcommand{\lrbrac}[1]{\left\lbrace #1 \right\rbrace}
\newcommand{\norm}[1]{\lVert #1 \rVert}
\newcommand{\abs}[1]{\left\lvert #1 \right\rvert}

\DeclareMathOperator*{\btimes}{\times} 

\newcommand{\coker}{\operatorname{ coker }}

\newcommand\diag{\operatorname{diag}}
\newcommand\Diff{\operatorname{Diff}}

\newcommand{\dR}{\operatorname{dR}}
\newcommand\dvol{\operatorname{dvol}}

\newcommand{\tend}{\operatorname{End}}

\newcommand{\Hom}{\operatorname{Hom}}
\newcommand{\Ht}{\operatorname{H}}

\newcommand\Id{\operatorname{Id}}
\renewcommand{\Im}{\operatorname{Im}}
\newcommand{\Image}{\operatorname{Image}}

\newcommand{\loc}{\operatorname{loc}}
\renewcommand{\mid}{\operatorname{mid}}

\newcommand{\phg}{ \mathrm{phg} }

\renewcommand{\Re}{\operatorname{Re}}

\newcommand{\reg}{ \mathrm{reg} }

\newcommand{\sign}{\operatorname{sign}}

\newcommand{\supp}{\operatorname{supp}}

\newcommand{\ff}{\mathrm{ff}}

\newcommand\Mand{\text{ and }}

\newcommand\Mfor{\text{ for }}
\newcommand\Mforall{\text{ for all }}

\newcommand\Mhence{\text{ hence }}

\newcommand\Mif{\text{ if }}

\newcommand\Min{\text{ in }}

\newcommand\Mon{\text{ on }}
\newcommand\Mor{\text{ or }}

\newcommand\Mst{\text{ s.t. }}

\newcommand\Mwhenever{\text{ whenever }}
\newcommand\Mwhere{\text{ where }}
\newcommand\Mwith{\text{ with }}

\DeclareMathAlphabet{\mathpzc}{OT1}{pzc}{m}{it}
\newcommand{\cl}[1]{\mathpzc{cl}\left( #1 \right)}

\newcounter{mnotecount}[section]

\newcommand\paperintro%
        {%
         }
\newcommand\paperbody%
        {%
         }

\newcommand\bbB{\mathbb{B}}
\newcommand\bbC{\mathbb{C}}

\newcommand\bbN{\mathbb{N}}

\def\RR{\bbR}
\newcommand\bbR{\mathbb{R}}
\newcommand\bbS{\mathbb{S}}

\newcommand\cA{\mathcal{A}}
\newcommand\cB{\mathcal{B}}
\newcommand\cC{\mathcal{C}}
\newcommand\calC{\mathcal{C}}
\newcommand\cD{\mathcal{D}}
\newcommand\calD{\mathcal{D}}
\newcommand\cE{\mathcal{E}}
\newcommand\cF{\mathcal{F}}
\newcommand\cG{\mathcal{G}}
\newcommand\cH{\mathcal{H}}
\newcommand\cI{\mathcal{I}}
\newcommand\cJ{\mathcal{J}}
\newcommand\cK{\mathcal{K}}
\newcommand\cL{\mathcal{L}}
\newcommand\cM{\mathcal{M}}

\newcommand\cO{\mathcal{O}}
\newcommand\calO{\mathcal{O}}
\newcommand\cP{\mathcal{P}}

\newcommand\cS{\mathcal{S}}
\newcommand\cT{\mathcal{T}}
\newcommand\cU{\mathcal{U}}
\newcommand\cV{\mathcal{V}}
\newcommand\cW{\mathcal{W}}

\newcommand\sB{\mathscr{B}}
\newcommand\sC{\mathscr{C}}
\newcommand\sD{\mathscr{D}}

\newcommand\sK{\mathscr{K}}

\newcommand\sP{\mathscr{P}}

\newcommand\sT{\mathscr{T}}

\newcommand\tH{\operatorname{H}}
\newcommand\tI{\operatorname{I}}

\newcommand\bB{\mathbf{B}}
\newcommand\bN{\mathbf{N}}

\newcommand\cTR{\mathcal{TR}}

\datver{\Modified}

\begin{document}
\title{Hodge theory on Cheeger spaces}

\author{Pierre Albin}
\address{University of Illinois at Urbana-Champaign}
\email{palbin@illinois.edu}
\author{Eric Leichtnam}
\address{CNRS Institut de Math\'ematiques de Jussieu-PRG, b\^atiment Sophie Germain.}
\author{Rafe Mazzeo}
\address{Department of Mathematics, Stanford University}
\email{mazzeo@math.stanford.edu}
\author{Paolo Piazza}
\address{Dipartimento di Matematica, Sapienza Universit\`a di Roma}
\email{piazza@mat.uniroma1.it}

\begin{abstract}
We extend the study of the de Rham operator with ideal boundary conditions from the case of isolated conic singularities,
as analyzed by Cheeger, to the case of arbitrary stratified pseudomanifolds. We introduce a class of ideal boundary 
conditions and the notion of mezzoperversity, which intermediates between the standard lower and upper middle perversities 
in intersection theory, as interpreted in this de Rham setting, and show that the de Rham operator with these boundary
conditions is Fredholm and has compact resolvent. We also prove an isomorphism between the resulting Hodge 
and $L^2$ de Rham cohomology groups, and that these are independent of the choice of iterated edge metric. 
On spaces which admit ideal boundary conditions of this type which are also self-dual, which we call `Cheeger spaces', 
we show that these Hodge/de Rham cohomology groups satisfy Poincar\'e Duality.
\end{abstract}

\maketitle

\tableofcontents

\paperintro
\section*{Introduction}

This paper is part of a larger project which develops various aspects of de Rham and Hodge cohomology theory for the natural 
class of `iterated edge' metrics on smoothly stratified spaces. In the course of this, we develop new analytic tools 
adapted to this singular geometry.\\

Stratified spaces arise naturally even when one is primarily interested in smooth objects. For example, Whitney showed that 
if $\hat X$ is the set of common zeros of a finite family of polynomials, then the singular part of $\hat X,$ 
$\mathrm{sing}(\hat X),$ is again algebraic and of dimension strictly less than that of $\hat X.$ This gives a natural filtration
\begin{equation*}
	\hat X \supseteq \mathrm{sing}(\hat X) \supseteq \mathrm{sing}(\mathrm{sing}(\hat X)) \supseteq \ldots
\end{equation*}
by algebraic subvarieties, or alternately a decomposition of $\hat X$ into a union of analytic manifolds: 
\begin{equation*}
	\hat X^{\mathrm{reg}} = \hat X \setminus \mathrm{sing}(\hat X), \quad
	\mathrm{sing}(\hat X) \setminus \mathrm{sing}(\mathrm{sing}(\hat X)), \quad \ldots,
\end{equation*}
where these different `strata' fit together in a precise way. There has been extensive investigation of this notion of 
stratified spaces. This class includes every analytic variety (Whitney), subanalytic set (Verdier), as well as the generalization 
of this last class to the class of definable sets in an arbitrary order-minimal structure (Loi). Stratified spaces also include 
polyhedra, orbit spaces of many group actions on manifolds, and mapping cylinders of maps between manifolds.
These spaces, the definition of which is recalled below, thus arise as ubiquitously as smooth manifolds. Nonetheless,
many basic aspects of geometric analysis on these spaces have not yet been developed in the literature, and this paper
should be regarded as one step toward this goal. 

A fundamental problem is to understand the de Rham and Hodge theory of these spaces. The seminal work on this was
carried out by Cheeger in the early 1980's \cite{Cheeger:Conic, Cheeger:Hodge, Cheeger:Spec}.  If $\hat X$ is a compact
stratified pseudomanifold and $X$ its top-dimensional regular part, and if $g$ is an `iterated incomplete 
edge' (which we abbreviate as $\iie$) metric on $X$, then there  are two natural de Rham complexes of $L^2$ forms:  
the first uses the minimal extension of the exterior derivative, with domain
\begin{equation*}
	\mathcal D_{\min}(d)
	= \{ \omega \in L^2(X; \Lambda^*T^*X): 
	\exists \omega_j \in \cC^{\infty}_c(X;\Lambda^*T^*X) \text{ s.t. } \omega_j \to \omega, \Mand (d\omega_j) 
\text{ is $L^2$-Cauchy} \},
\end{equation*}
while the second uses the maximal extension of $d$, with domain
\begin{equation*}
	\mathcal D_{\max}(d)
	= \{ \omega \in L^2(X; \Lambda^*T^*X): d\omega \in L^2(X; \Lambda^*T^*X)\}.
\end{equation*}
Both complexes have finite dimensional cohomology and Cheeger \cite[Theorem 6.1]{Cheeger:Hodge} showed 
that the latter is dual to the topological lower middle perversity intersection homology group $IH_*^{\overline m}(\hat X)$ 
of Goresky-MacPherson. It follows that the cohomology of the minimal domain $L^2$ complex is the dual of the upper 
middle perversity groups $IH_*^{\overline n}(\hat X).$ There is a special class of stratified spaces, satisfying the Witt condition, 
on which these two cohomologies coincide. Even more strongly, for suitably chosen metrics, the domains 
$\cD_{\min}(d)$ and $\cD_{\max}(d)$ coincide.

In a previous paper \cite{ALMP} we have given a new treatment of the de Rham/Hodge theory of an $\iie$ metric on 
Witt spaces; this was then applied to the Novikov conjecture on these spaces. In the present paper, we extend that theory
to handle general non-Witt spaces. For any metric on a non-Witt space, $\cD_{\min}(d) \neq \cD_{\max}(d)$, and one key
part of our task is to define geometrically appropriate domains for the exterior derivative that lie between these two extremes.
This is done following ideas of Cheeger, published by him in the case of isolated conic singularities \cite{Cheeger:Conic}, 
and here extended to this general stratified setting. A space $\hat X$ is non-Witt if and only if there is a stratum $Y$ with 
link at $q \in Y,$ $Z_q,$ of even dimension such that
\begin{equation*}
	IH^{\overline m}_{\tfrac12\dim Z_q}(Z_q) \neq \{ 0 \}.
\end{equation*}
If this is the case, these spaces form a vector bundle over $Y$ with a natural flat connection. At every stratum where
the Witt condition fails, we  choose a smooth subbundle, with 
\begin{equation*}
	W(Z_q) \subseteq (IH^{\overline m}_{\tfrac12\dim Z_q}(Z_q))^*,
\end{equation*}
which is parallel with respect to the flat connection, and so that whenever two strata have intersecting closures,
these choices are compatible. Such a choice will be called a {\em mezzoperversity.}

We work with `suitably scaled, rigid' $\iie$ metrics. The rigidity condition states that the metric has an `exact conic'
structure on the conic fibers of the normal neighborhoods of each stratum.  As shown in \cite[Propositions 3.1-3.2 ]{ALMP},
there is no loss of generality in assuming this. The other condition of being suitably scaled is also easy to
arrange, and requires simply that the restriction of the metric to the link of each of these conic fibrations be multiplied
by a small enough factor so as to exclude spurious small eigenvalues of the induced operators on these links, 
and hence also excludes unnecessary indicial roots. We say more about this in \S\ref{sec:IndicialDeRham}, see in
particular Lemma~\ref{lem:SuitableScaling}.
\begin{theorem}
Let $(\hat X, g)$ be a stratified pseudomanifold with a (suitably scaled) $iie$-metric. Each mezzoperversity $\cW$ 
on $\hat X$ determines a domain for the exterior derivative and for the de Rham operator
\begin{equation*}
	\mathcal D_{\cW}(d), \quad
	\mathcal D_{\cW}(d + \delta). 
\end{equation*}
We refer to this specification by saying that we are imposing {\em Cheeger ideal boundary conditions.}
The latter makes $d+ \delta$ a self-adjoint operator with compact resolvent and induces a strong Kodaira decomposition on differential forms, the former induces a Fredholm complex whose cohomology
\begin{equation*}
	\text{H}_{\cW}^*(\hat X) = \text{H}^*(d, \mathcal D_{\cW}(d))
\end{equation*}
satisfies a Hodge theorem and is independent of the choice of $\iie$-metric.
\end{theorem}
In a companion \cite{ALMP13.2}  to this paper, we show that this cohomology is invariant under stratified homotopy equivalences.

In proving this theorem, there are two main analytic difficulties to overcome. The first is the definition of the domains. If $Y$ 
is a stratum of depth one, then its link $Z$ is a closed manifold, and we establish below that any element 
$u \in \cD_{\max}(d+\delta)$ has a distributional asymptotic expansion in terms of $x,$ a `boundary defining function' for $Y,$
\begin{equation}\label{DistExpansion}
\begin{gathered}
	u = x^{-f/2}(\alpha(u) + dx\wedge \beta(u)) + v, \\
	\Mwith 
	\alpha(u), \beta(u) \in H^{-1/2}(Y,\Lambda^*T^*Y \otimes \Ht^{\mid}_{L^2}(Z) ), \quad
	v \in x^{1-\eps}H^{-1}(X; \Lambda^* \Iie T^*X)
\end{gathered}
\end{equation}
(for arbitrarily small $\eps$).
The coefficients $\alpha(u),$ $\beta(u)$ vanish if $\dim Z$ is not even or, more generally, if the Witt condition holds.
Otherwise these forms constitute the Cauchy data, on which the ideal boundary conditions are (algebraic) restrictions.
We show that once we have imposed suitable boundary conditions on the first $k$ singular strata (in terms of
increasing depth), then there is an expansion of the form \eqref{DistExpansion} at the contiguous stratum of depth $k+1$. 

The second analytic difficulty, to the resolution of which the bulk of the paper is devoted, involves the construction 
of a parametrix for $d+\delta$ with Cheeger ideal boundary conditions. On a space with a simple edge, i.e.\ only one
singular stratum, $x(d+\delta)$ is an `elliptic edge operator' in the sense of \cite{Mazzeo:Edge}, and there is a complete 
pseudodifferential theory to study it \cite{Mazzeo:Edge, Mazzeo-Vertman}, at least when $\dim Z$ is even so that there 
is no indicial root exactly at the $L^2$ cutoff.   The strategy here is to construct a parametrix for $x(d+\delta)$ acting 
on the weighted space $x^{\eps}L^2,$ and then modify it to obtain a parametrix for $(d+\delta,\cD_{\cW}(d+\delta))$ 
with an error term taking values in $x^{\eps}L^2.$  Using this parametrix, we show that the domain $\cD_{\cW}(d+\delta)$ 
is compactly included in $L^2.$

There is now a substantial amount of literature concerning related problems on spaces with simple edge singularities.
We mention in particular, beyond \cite{Mazzeo:Edge} and \cite{Mazzeo-Vertman}, the ongoing work of
Krainer and Mendoza \cite{Krainer-Mendoza:Kernel, Krainer-Mendoza2, Krainer-Mendoza3} and the recent work
of Cheeger and Dai \cite{Cheeger-Dai}. The oeuvre of Schulze \cite{Schulze}
handles certain problems, though less general than considered here, for stratified spaces. \\

Every mezzoperversity $\cW$ has a dual mezzoperversity $D\cW$ such that the intersection pairing of differential forms 
restricts to a non-degenerate pairing between the de Rham cohomology of $\cW$ and that of $D\cW.$
A mezzoperversity which coincides with its dual is called a {\em self-dual mezzoperversity}. It is not hard to 
see that there are topological obstructions, for example involving the signatures of the links, to the existence 
of a self-dual mezzoperversity on a stratified pseudomanifold. We shall call a compact stratified space which carries a  
self-dual mezzoperversity a {\em Cheeger space}. 

\begin{theorem}
Let $(\hat X, g)$ be a stratified pseudomanifold with a (suitably scaled) $iie$-metric. 
If $\mathcal L$ is a self-dual mezzoperversity on $\hat X,$ then the domains
\begin{equation*}
	\mathcal D_{\cW}(d), \quad
	\mathcal D_{\cW}(d + \delta)
\end{equation*}
are invariant under the action of the Hodge star and the cohomology
\begin{equation*}
	\text{H}_{\cW}^*(\hat X) = \text{H}^*(d, \cD_{\cW}(d))
\end{equation*}
satisfies Poincar\'e duality. Moreover, its signature is the index of the analytic signature operator, i.e.\ 
the de Rham operator $d+\delta$ with the involution induced by the Hodge star, with boundary
conditions and the domain induced by $\cW.$
\label{thm2}
\end{theorem}

In a companion paper \cite{ALMP13.2}, we apply the constructions of this paper to define higher signatures on a 
Cheeger space and then prove the Novikov conjecture for Cheeger spaces which have fundamental groups satisfying 
the strong Novikov conjecture. We also prove there that the signature of Theorem~\ref{thm2} above is independent 
of the choice of self-dual mezzoperversity 
and is also a bordism invariant. These properties of the signature can also be deduced by comparing the analytic approach 
developed in the present paper with a purely topological definition of self-dual mezzoperversities and Poincar\'e duality 
developed earlier by Banagl. The fact that these two approaches lead to equivalent signatures is proved in \cite{ABLMP}, 
joint with Banagl. \\

Notice that at each stratum there are two canonical choices of mezzoperversity, the subbundles of $\cH^{\mid}_{\cL(Z)}(Z)$ of full rank or of rank zero.
It follows from the work of Cheeger, and is shown directly in \cite{ABLMP}, that these choices correspond to the upper-middle and lower-middle perversities of Goresky-MacPherson at this stratum. This companion paper also shows that any choice of self-dual sheaf in the sense of \cite{BanaglShort} yields a mezzoperversity, satisfying the self duality studied in \S\ref{sec:SignLag}. In particular, from \cite{Banagl-Kulkarni} we see that the reductive Borel-Serre compactification of a Hilbert modular surface is an example of a space carrying interesting non-trivial self-dual mezzoperversities, i.e., these are Cheeger spaces.
Although showing that a particular space is a Cheeger space is in general a difficult task, Cheeger spaces (with boundary) are sufficiently abundant to provide cobordisms between any two smooth oriented manifolds with the same dimension and signature (see \cite[\S 6]{ALMP13.2}).

\noindent {\bf Acknowledgements.} 
P.A. was partly supported by NSF Grant DMS-1104533 and an IHES visiting position and thanks Sapienza 
Universit\`a di Roma, Stanford, and Institut de Math\'ematiques de Jussieu for their hospitality and support.
R.M. acknowledges support by NSF Grant DMS-1105050. 
P.P. thanks the {\it Projet Alg\`ebres d'Op\'erateurs} of {\it Institut
de Math\'ematiques de Jussieu}
for hospitality during several short visits and a two months long visit
 in the Spring of 2013; financial support was
provided by Universite' Paris 7, Istituto Nazionale di Alta Matematica (INDAM) and CNRS
(through the bilateral project ``Noncommutative geometry") 
and Ministero dell'Universit\`a e della Ricerca Scientifica 
(through the project ``Spazi di Moduli e Teoria di Lie"). 
E.L. thanks Sapienza Universita' di Roma 
for hospitality during several week-long visits; 
financial support was  provided again by INDAM and CNRS
through the bilateral project ``Noncommutative Geometry".

\smallskip

The authors are happy to thank {Jesus Alvarez-Lopez, Markus Banagl, Francesco Bei, Jeff Cheeger, Michel Hilsum, Thomas Krainer, 
Richard Melrose and Gerardo Mendoza} for many useful and interesting discussions.

\paperbody
\section{High codimension boundary conditions} \label{sec:BdyConditions}
In this section we give an overview of the types of operators and boundary conditions to which our methods apply.
To help motivate and better explain all of this, we describe it in the now well-known simple edge setting,
i.e.\ where $X$ has only one singular stratum, then explain the analogous constructions in the next most complicated
case, the depth two setting, and finally turn to spaces with singular strata of arbitrary depth. As we will see, much
of the new complexity already occurs in the depth two case; the extension to higher depth singularities involves
an induction argument, but otherwise differs mostly in the notation.  We present all of this for differential operators 
of first order satisfying a few structural hypotheses, but our main interest is in the
de Rham and signature operators associated to an incomplete iterated edge metric, and so we use these operators 
as the key examples throughout, and describe the types of boundary problems in full detail only for these operators. 

Let $\hat X$ be a compact stratified space, with singular strata 
\begin{equation*}
	Y^1, Y^2, \ldots, Y^{k+1},
\end{equation*}
ordered by increasing depth. We also denote by $\hat X^{\mathrm{reg}}$ the
open dense stratum of top dimension, the so-called regular part of $\hat X$; the other strata $Y^j$ are thus regarded
as higher codimension boundaries of $\hat X$. We refer the reader to \cite[\S 2]{ALMP}
for a careful discussion of this class of spaces, and for the accompanying description of a resolution process
which associates to any such stratified space $\hat X$ a compact manifold with corners $\widetilde X$ with iterated
fibration structure on its boundary faces, obtained by blowing up the strata $Y^j$ in order of decreasing depth. 

\begin{remark}
For notational convenience, we assume for the rest of this paper that each stratum is connected, and even
more strongly that there is a single maximal chain of strata $Y^1, \ldots, Y^{k+1}$ where
$\overline{Y}^j \supset \overline{Y}^{j+1}$ for all $j$. Both of these assumptions are simple to remove
and are adopted here to make the already complicated notation as simple as possible. 
\end{remark}

Without yet specifying the degeneracies of the operators we shall consider later (i.e.\ the behavior of their 
coefficients near the singular strata), we first recall some generalities about closed extensions of
unbounded operators on Hilbert spaces. 

Let $L$ be an elliptic operator of order one (for convenience, one can proceed similarly for general elliptic differential operators) 
on $\hat X$ acting between sections of two bundles $E$ and $F$, and fix a smooth measure 
on the regular part of $\hat X$ and Hermitian metrics on the fibres of $E$ and $F$. We may then regard $L$ as an 
unbounded operator acting on $\calC^\infty_0(\hat X^{\mathrm{reg}})$ sections of $E$.  Its graph, as a subspace
of $L^2(\hat X^{\reg}; E) \oplus L^2(\hat X^{\reg}; F)$, is not closed. There are two canonical closed extensions of this subspace,
both of which are graphs over dense domains $L^2 \supset \calD_{\max} \supset \calD_{\min} \supset \calC^\infty_0(\hat X)$. 
These are the minimal and maximal extensions: $\calD_{\min}$ is the domain corresponding to the graph closure of $L$ on the
core domain $\calC^\infty_c(\hat X^{\mathrm{reg}})$, while $\calD_{\max}$ consists of all sections $u \in L^2$
with $L u \in L^2$. The imposition of boundary condition consists of a choice of an intermediate domain $\calD$ 
lying between $\calD_{\min}$ and $\calD_{\max}$. The goal is to choose such a domain so that the corresponding opertor 
$(L, \calD)$ is not only closed but also Fredholm. 

It is very difficult to analyze $\calD_{\max}$ for operators on a stratified space directly.  Instead, we proceed
inductively. If $\hat X$ has depth $k$, assume that `good' boundary conditions have already been chosen 
at the strata $Y^1, \ldots, Y^{k-1}$ of depth less than $k$.  We study the `partially minimal' and `partially maximal' domains 
obtained by imposing these boundary conditions at the lower depth strata and applying the general definitions above
at the stratum $Y^k$ of depth $k$.  Functions (or sections) in this partially maximal domain which satisfy
$Lu = 0$ (or $Lu = f$ for $f \in L^2$) have partial asymptotic expansions at $Y^k$, the coefficients of which 
can be regarded as the Cauchy data of $u$.  A choice of good boundary conditions at $Y^k$ consists of imposing 
appropriate conditions on these Cauchy data; in this paper we shall only impose local algebraic conditions on these
Cauchy data, but more generally one could impose nonlocal (e.g.\ pseudodifferential) conditions.  A refined analysis 
of more general boundary value problem of this type on spaces with simple edge singularities (i.e.\ on depth $1$ spaces)
is carried out in \cite{Mazzeo-Vertman}, and in a series of articles by Krainer and Mendoza, 
\cite{Krainer-Mendoza:Kernel, Krainer-Mendoza2, Krainer-Mendoza3, Gil-Krainer-Mendoza:Wedge}.

As noted earlier, we shall make a set of structural assumptions on the operator $L$, and for the sake of
exposition, we introduce these assumptions gradually, once we have provided motivation and suitable notation.
For the first of these, let us say that a domain is {\bf localizable} if it is closed under multiplication by functions in
\begin{equation*}
	\CI_{\Phi}(\wt X) = \{ f \in \CI(\wt X) : i^*_Hf \in \phi_H^*\CI(Y_H) \text{ for all boundary hypersurfaces } H \}.
\end{equation*}
(recall that $i_H:H \lra \wt X$ is the inclusion of a boundary hypersurface and $\phi_H:H\lra Y_H$ is the boundary fibration of $H$).

\begin{assume}\label{Ass:Localizable} $\cD_{\max}(L)$ is localizable. 
\end{assume}
\noindent  
Note that $f \in \CI_{\Phi}(\wt X)$ is equivalent to asking that $df$ have bounded pointwise norm with respect to an $\iie$-metric.
Any first order differential operator $L$ satisfies
\begin{equation*}
	[L,f] = \tfrac1i\sigma(L)(df)
\end{equation*}
and hence $\cD_{\max}(L)$ is localizable as long as the principal symbol of $L$ induces a bounded operator on $L^2.$
We showed in \cite{ALMP} that Assumption~\ref{Ass:Localizable} 
holds for the de Rham operator (the same argument extends easily to show that it also holds for any 
Dirac-type operator associated to an $\iie$ metric).

\subsection{Depth one}\label{sec:SimpleEdgeBC}
As promised, we now embark on a much more detailed discussion, initially on stratified spaces $\hat X$ with
only one singular stratum $Y.$  We call these simple edge spaces.  One of the basic features of a stratified space
is its Thom-Mather data (see \cite[\S2.1]{ALMP}), which in this simple edge setting consists of a tubular neighborhood 
$\sT_Y$ of $Y$ in $\hat X$, equipped with a fibration over $Y$, the typical fiber of which is a cone over a smooth 
closed manifold $Z$. The `resolution' $\wt X$ of $\hat X$ is obtained by radially blowing up each of these
conic fibres at its vertex; thus $\wt X$ is a smooth manifold with boundary, and the Thom-Mather data 
of $\hat X$ is converted to a fibration of $\pa \wt X$ over $Y$ with fiber $Z,$
\begin{equation*}
	Z - \pa \wt X \xlra{\phi} Y.
\end{equation*}
There is a canonical identification of the interior of $\wt X$ and $\hat X^{\mathrm{reg}}$, and we denote these common smooth
spaces by $X$: thus
\begin{equation*}
	X = \wt X^\circ = \hat X^{\mathrm{reg}}.
\end{equation*}

Now choose any smooth boundary defining function $x$ for $\pa \wt X$. A complete edge metric on $X$ is, by
definition, a metric which near $Y$ takes the form
\begin{equation*}
	\frac{dx^2}{x^2} + \frac{\phi^*g_Y}{x^2} + g_Z,
\end{equation*}
where $g_Y$ is metric on $Y$ and $g_Z$ a family of two-tensors that restricts to a metric on each fiber of $\phi.$
We always assume that all of these terms are smooth for $x \geq 0$. We have buried some technicalities about the 
precise form of the asymptotics we allow, and shall return to these details when they become relevant below. 
Complete edge metrics may be regarded as the basic structural regulators of the geometry of simple edge spaces,
and many of the auxiliary quantities defined below are most naturally phrased in terms of these metrics.
However, the focus of this paper is on {\it incomplete} edge metrics; by definition $g$ is an incomplete
edge metric if $x^{-2}g$ is a complete edge metric, see \cite{ALMP}.  

The main reference for all of the definitions and constructions in this simple edge case
is the paper \cite{Mazzeo:Edge}. 
A vector field $V$ on $\wt X$ has pointwise length which is uniformly bounded with respect to a complete edge metric 
if the restriction of $V$ to $\pa \wt X$ is tangent to the fibers of $\phi.$  The space of all such $V$ is called the
space of `edge vector fields'; these are sections of the naturally defined `edge tangent bundle', ${}^e T\wt X$.
This bundle naturally includes in the ordinary tangent bundle, ${}^e T\wt X \lra T\wt X$, and this inclusion is an 
isomorphism over $X$. Over $\pa \wt X$, however, 
\begin{equation*}
	{}^e N\pa X = \ker ( {}^e T\wt X \lra T\wt X ) \lra \pa \wt X
\end{equation*}
is a bundle of rank $\dim Y + 1.$ 

A first-order differential operator $L$ on $X$ is an `{\em edge differential operator}' if it can be written as a locally finite sum of
products of edge vector fields. This means, in local coordinates near the boundary, that it has the form
\begin{equation*}
	L = \sum_{j+ |\alpha|+|\beta| \leq 1} a_{j,\alpha,\beta}(x,y,z) (x\pa_x)^j (x\pa_y)^{\alpha} \pa_z^{\beta}.
\end{equation*}
Here, as before, $x$ is a boundary defining function, while $y_1, \ldots, y_h$ are coordinates 
along $Y,$ and $z_1, \ldots, z_f$ are coordinates along $Z$. The coefficients $a_{j,\alpha,\beta}$ are assumed to be 
smooth functions in these variables; if $L$ acts between sections of smooth bundles $E$ and $F$ over $\wt X$, then
the $a_{j,\alpha,\beta}$ are smooth sections of $\mathrm{End}(E,F)$. The space of these operators of degree at most one
is denoted $\Diff^1_e(X; E, F)$. 

There is an invariantly defined {\em edge symbol} of $L$,
\begin{equation*}
	{}^e\sigma_m(L): = \sum_{j+ |\alpha|+|\beta| = 1} a_{j,\alpha,\beta}(x,y,z) (\xi)^j (\eta)^{\alpha} (\theta)^{\beta},
\end{equation*}
which is naturally defined as a section of the pullback bundle $p^* \operatorname{Hom}(E,F)$ over the edge cotangent bundle 
${}^e T^*X \xlra{p} X$, which is the dual of ${}^e TX$.  This edge symbol is part of a short exact sequence
\begin{equation*}
	0 \lra
	\Diff_e^{0}(X;E,F) \lra
	\Diff_e^1(X;E,F) \xlra{{}^e\sigma_1}
	\CI({}^e T^*X; p^*\Hom(E, F))
	\lra 0
\end{equation*}
in which $\Diff_e^0(X;E,F) = \CI(X;\Hom(E,F)).$
An edge differential operator $L$ is said to be (edge-)elliptic if ${}^e\sigma_1(L)$ is invertible away from
the zero section of ${}^eT^*X$.   For brevity, we always say ellipticity rather than edge ellipticity; the latter
is always implied. 

Next, the wedge operators (also known as incomplete edge operators) of order one are defined as the class
\begin{equation*}
	\Diff_{\ie}^1(X;E,F) = x^{-1} \Diff_e^1(X;E,F).
\end{equation*}
Thus $L$ is in $\Diff_{\ie}^1(X;E,F)$ if, in the same type of local coordinates, 
\begin{equation*}
	L = \sum_{j+ |\alpha|+|\beta| =1} a_{j,\alpha,\beta}(x,y,z) (\pa_x)^j (\pa_y)^{\alpha} (\tfrac 1x\pa_z)^{\beta} + \tfrac 1x a_{0,0,0}(x,y,z),
\end{equation*}
where the coefficients $a_{\cdot}$ are sections of $\Hom(E,F).$
We say that any such $L$ is elliptic if the associated edge operator $P = x L$ is edge elliptic. 

As we described in general earlier, if $L$ is a wedge operator, then 
\begin{equation*}
	L: \CIc(X; E) \lra \CIc(X;F),
\end{equation*}
and if $X$ is endowed with an (incomplete) edge metric and $E,$ $F$ with Hermitian metrics, 
there are two canonical closed extensions: the {\em minimal extension}, with domain 
\begin{multline*}
	\cD_{\min}(L)  = \{ u \in L^2(X;E) : \mbox{there exists} \ u_j \in \CIc(X;E) \\
	\Mst 	u_j \to u \Min L^2(X;E) \Mand Lu_j \text{ is Cauchy in }L^2(X;F) \},
\end{multline*}
where by definition, if $u \in \cD_{\min}(L)$ then $Lu = \lim Lu_j$, and the {\em maximal extension}, with domain
\begin{equation*}
	\cD_{\max}(L) = \{ u \in L^2(X;E) : Lu \in L^2(X;F) \};
\end{equation*}
here $Lu$ is defined distributionally.

We can perform the same constructions for the associated edge operator $P = xL$, 
and if it is elliptic then it is a consequence of the results in \cite{Mazzeo:Edge} that 
\begin{equation*}
	\cD_{\min}(P) = \cD_{\max}(P) =H^1_e(X;E) = \{ u \in L^2(X;E) : V u \in L^2(X;E) \Mforall V \in \CI(X;{}^e TX) \}.
\end{equation*}
This is the edge Sobolev space of order one.  There is some relationship between these spaces and
the domains of the wedge operator $L$, namely
\begin{equation*}
	xH^{1}_e(X;E) \subseteq \cD_{\min}(L) \subseteq \cD_{\max}(L) \subseteq H^{1}_e(X;E),
\end{equation*}
but in general, these subspaces may all be different.  To prove this chain of inclusions, note that clearly $u \in \cD_{\max}(L)$ 
implies $u \in \cD_{\max}(xL) = H^{1}_e(X;E)$; on the other hand, since $H^{1}_e(X;E) = \cD_{\min}(Lx),$ then
given any $u \in xH^{1}_e(X;E)$, we can find $w_n \in \CIc(X;E)$ which converges to $x^{-1}u$ in $L^2(X;E)$ 
and such that $Lx(w_n) \to Lx(x^{-1}u) = Lu,$ which shows that $u \in \cD_{\min}(L).$

A key difference from the closed (nonsingular) case is that an elliptic edge operator is not automatically Fredholm,
and one of the main points of the theory is to determine the extra conditions which need to be imposed to obtain
a Fredholm problem.  These conditions are phrased in terms of a secondary model for $L$ beyond its
principal symbol; this is a family of operators known collectively as the `normal operator,'
\begin{equation*}
	N: \Diff_e^1(X;E,F) \lra
	\Diff_e^1( (TY^+ \times_Y \pa X) /Y; \pi^*E,\pi^*F);
\end{equation*}
here $TY^+$ is the inward-pointing half of the bundle ${}^eN\pa X$ and $\pi:TY^+ \times_Y \pa X \lra Y$ is the natural projection.
The restriction to the fiber over a point $q \in Y$ is denoted 
\begin{equation*}
	N_q: \Diff_e^1(X;E,F) \lra
	\Diff_e^1( T_qY^+ \times Z_q; \pi_q^*E,\pi_q^*F)
\end{equation*}
where, in an abuse of notation, $\pi_q^*E$ denotes the restriction of the bundle $\pi^*E\lra TY^+ \times_Y \pa X$ to 
$T_qY^+ \times Z_q,$ and similarly $\pi_q^*F.$ In local coordinates, after identifying $T_qY^+$ with $\bbR^+ \times T_qY,$ 
\begin{equation*}
	N_q\lrpar{ 
	\sum_{j+ |\alpha|+|\beta| \leq 1} a_{j,\alpha,\beta}(x,y,z) (x\pa_x)^j (x\pa_y)^{\alpha} \pa_z^{\beta} }
	= 
	\sum_{j+ |\alpha|+|\beta| \leq 1} a_{j,\alpha,\beta}(0,q,z) (s\pa_s)^j (s\pa_u)^{\alpha} \pa_z^{\beta} 
\end{equation*}
where $s$ and $u$ are coordinates on $\bbR^+$ and $T_qY$ respectively. In \cite{Mazzeo:Edge} it is shown that an elliptic 
edge differential operator $L$ induces a Fredholm operator on $x^aL^2(X;E) \lra x^aL^2(X;F)$ precisely when
\begin{equation*}
	N_q(L): 
	s^a L^2(T_qY \times \bbR^+ \times Z_q; \pi_q^*E)
	\lra s^aL^2(T_qY \times \bbR^+ \times Z_q; \pi_q^*F)
\end{equation*}
is invertible for every $q \in Y$.  

The method for proving this result relies on a detailed analysis of the structure of the Schwartz kernel of a parametrix for $L$.
This parametrix is constructed using both ${}^e\sigma_m(L)^{-1}$, via a modification of the standard symbol calculus
parametrix method, and $N_q(L)^{-1}$. The inverse of this family of model problems is integrated into
the parametrix using a fundamentally geometric construction. Specifically, we make the ansatz that the
Schwartz kernel of the parametrix lifts to a polyhomogeneous (or at least conormal) distribution on the `edge double space',
\begin{equation*}
	X^2_e = [X^2; \pa X \btimes_Y \pa X].
\end{equation*}
This space is obtained from $X^2$ by radially blowing up the fiber diagonal of $\pa X,$ i.e., replacing the fiber diagonal with 
its spherical inward-pointing normal bundle, see \cite{Mazzeo:Edge}. The boundary hypersurface produced by this blow-up is 
called the {\em edge front face} and denoted $\ff_e.$  A pseudodifferential edge operator $A$, by definition, is a pseudodifferential
operator on $X$, the Schwartz kernel $K_A$ of which is a distribution on $X^2$, but which lifts to $X^2_e$ as a polyhomogeneous
distribution.  We typically place restrictions on, or at least keep track of, the exponents in the expansions of this lift of $K_A$ 
at the various boundary faces. In particular, this lift depends smoothly on the normal variable to $\ff_e$, and its restriction
to this face is, by definition, the normal operator of $N(A)$.  The face $\ff_e$ is the total space of a fibration over $Y$
with each fiber naturally identified with a projective compactification of $T_q Y \times \bbR^+ \times Z_q$, and $N(A)$ should
be regarded as a family of operators $N_q(A)$ on this space. There are analogues of the various symbolic rules, including
that 
\[
N_q(A \circ B) = N_q(A) \circ N_q(B).
\]
Specializing to the case where $A$ is an elliptic differential edge operator with invertible normal operator, one should
therefore expect to construct a parametrix $B$ for $A$ by choosing $B$ so that ${}^e \sigma_{-m}(B) = {}^e\sigma_m(A)^{-1}$
and $N_q(B) = N_q(A)^{-1}$.  This is indeed the case, and is the basis for the parametrix construction in \cite{Mazzeo:Edge}. 
Thus pseudodifferential operators include both differential edge operators and parametrices of elliptic edge operators with invertible normal operator.
We carry out a somewhat stripped down version of this construction here on stratified spaces, see \S \ref{sec:InvNormal}. 

In studying the normal operator of $P \in \Diff^1_e(X;E,F)$, one uses properties of a simpler model operator. 
This is the {\em indicial family}, 
\begin{equation*}
	(Y, \bbR) \ni (q,\zeta) \mapsto I_q(P;\zeta) \in \Diff^1(Z_q;E,F),
\end{equation*}
given in local coordinates by 
\begin{equation*}
	I_q\lrpar{ 
	\sum_{j+ |\alpha|+|\beta| \leq 1} a_{j,\alpha,\beta}(x,y,z) (x\pa_x)^j (x\pa_y)^{\alpha} \pa_z^{\beta} ; \zeta}
	= 
	\sum_{j+|\beta| \leq 1} a_{j,0,\beta}(0,q,z) (\zeta)^j  \pa_z^{\beta}.
\end{equation*}
It is also defined by the equation
\begin{equation*}
	P( x^{\zeta} u(x, y,z) ) = x^{\zeta} I_y(P;\zeta) u(0,y,z) + \cO(x^{\zeta+1}).
\end{equation*}
For any fixed $q \in Y,$ the indicial family is invertible away from a discrete set $\spec_b(P;q)$, the elements of
which are called the {\em indicial roots} of $P$ at $q.$  There is a slightly more refined object, 
\begin{equation*}
	\Spec_b(P;q) = \{ (\zeta, p) \in \bbC \times \bbN_0: I_q(P;\eta)^{-1} \text{ has a pole of order at least $p+1$ at } \zeta \}. 
\end{equation*}
We say that a indicial root $\zeta$ is {\em simple} if $I_q(P;\eta)^{-1}$ has only a simple pole at $\zeta.$
We also set
\begin{equation*}
	\spec_b(P) = \bigcup_{q \in Y} \spec_b(P;q), \quad
	\Spec_b(P) = \bigcup_{q \in Y} \Spec_b(P;q).
\end{equation*}

Returning now to the incomplete edge operator $L \in \Diff^1_{\ie}(X;E,F)$,  observe that 
\begin{equation}
	u \in \cD_{\max}(L) \implies Lu \in L^2(X;F) \implies Pu = x Lu \in x L^2(X;F);
\label{extravan}
\end{equation}
in other words, $Pu$ decays faster than expected, compared to a generic element of $\cD(P)$.  The paper \cite{Mazzeo:Edge}
restricts attention only to those elliptic operators $P$ which satisfy the strong hypothesis that 
the set $\Spec_b(P)$ is independent of $y \in Y$.
Using this, a basic result in \cite{Mazzeo:Edge}, (cf. Lemma \ref{lem:ExpansionExistence} below) is that
\eqref{extravan} implies that $u$ has a partial expansion: 
\begin{equation}\label{Eq:ModelExpansion1}
	u \in \cD_{\max}(L) \implies 
	u \sim \sum_{(\zeta_j,p) \in \Spec_b(P)} u_{\zeta_j,p}(y,z) x^{\zeta_j}(\log x)^p + \wt u, \qquad \wt u \in x^{1-}H^{-1}_e(X;E)
\end{equation}
this sum is over the subset $\cS\cD(L)$ of pairs $(\zeta_j,p)$ in $\Spec_b(P)$ for which $x^{\zeta_j} \in L^2_{\loc} \setminus xL^2_{\loc},$ and we use the notation
\begin{equation*}
	x^{1-}H^{-1}_e(X;E) := \bigcap_{\eps>0} x^{1-\eps}H^{-1}_e(X;E).
\end{equation*}
Moreover, each map 
\begin{equation*}
	u \mapsto u_{\zeta_j, 0}, \ (\zeta_j,0) \in \cS\cD(L), 
\end{equation*}
is continuous from $\cD_{\max}(L)$, endowed with the graph topology, to $L^2(dz, H^{-\Re \zeta_j}(Y))$
(and similarly for $u_{\zeta_j,p},$ see the works of Krainer-Mendoza and Mazzeo-Vertman cited above). 
Whenever $p_j$ is the largest logarithmic power accompanying a given $\zeta_j,$ $u_{\zeta_j,p_j}$ is in the null space of $I(P;\zeta_j).$

The paper \cite{Mazzeo-Vertman} studies elliptic edge operators $L$ for which only the indicial roots $\zeta \in \Spec_b(P)$ 
lying in a particular strip $\underline{\delta} < \mathrm{Re}\, \zeta < \overline{\delta}$ (where $\underline{\delta}$ 
and $\overline{\delta}$ are in a range determined by mapping properties of the normal operators $N_q(P)$) are required to be constant.
The orders of the indicial roots are allowed to change, however.  In the present paper we work with the intermediate
\begin{assume}[Constant, simple indicial roots in a strip]\label{Ass:CstIndicialRoots}
	The set $\cS\cD(L)$ is independent of $y \in Y$ and consists entirely of simple indicial roots.
\end{assume}
This strip turns out to be the same as the one considered in \cite{Mazzeo-Vertman} (although there $\underline{\delta}$ and $\overline{\delta}$ are not allowed to be indicial roots, and we will necessary deal with this case). 
Quite a lot can still be done even without Assumption \ref{Ass:CstIndicialRoots}.  For example, 
in a series of papers \cite{Krainer-Mendoza:Kernel, Krainer-Mendoza2, Krainer-Mendoza3} and further
ongoing work, Krainer and Mendoza study elliptic edge operators assuming only that $\spec_b(P)$ 
does not intersect the boundaries of this same strip, but allowing the indicial data in this range to vary. 

Under Assumption \ref{Ass:CstIndicialRoots}, it is straightforward to regard the asymptotic coefficients 
$u_{\zeta_j}$ as distributional sections of a finite dimensional smooth 
vector bundle over $Y,$ called the trace bundle $\cTR_Y(L).$ One of the key results by Krainer and Mendoza is that one can still make sense of 
this trace bundle even when the indicial roots vary.
In any case, we henceforth restrict attention to operators for which the elements of $\Spec_b(P)$ in this strip are constant. 
We show in Lemma \ref{lem:SuitableScaling} that the de Rham operator for a (suitably scaled) metric satisfies this assumption, and do not
discuss the more general situation further. 

The starting point for the formulation of boundary conditions is the Cauchy data map 
\begin{eqnarray*}
\cD_{\max}(L) & \longrightarrow & \cTR_Y(L) \\
 u & \longmapsto & \cC_Y(L)(u) := ( u_{\zeta_j}) .
\end{eqnarray*}
For a single singular stratum, we show in Proposition \ref{prop:DminB} that for a first order, elliptic, incomplete edge operator $L$ we have
\begin{equation*}
	\cD_{\min}(L) = \{ u \in \cD_{\max}(L) : \cC_Y(L)(u) = 0 \}. 
\end{equation*}

This fact now shows that the projection 
\begin{equation*}
	\cD_{\max}(L) \lra \cD_{\max}(L) / \cD_{\min}(L) 
\end{equation*}
can be identified with the Cauchy data map, and hence every closed domain for $L$ can be realized as 
\begin{equation*}
	\cD_B(L) = \{ u \in \cD_{\max}(L) : B \circ \cC_Y(u) = 0 \},
\end{equation*}
where $B$ is an arbitrary closed linear operator acting on the coefficients $u_{\zeta_j, k}.$  Thus $L$ has a 
unique closed extension, or equivalently $\cD_{\max}(L) = \cD_{\min}(L)$, if and only if $\cS\cD(L)$ is empty.

A {\bf local ideal boundary condition} for $L$ is a bundle homomorphism 
\begin{equation*}
	B \in \CI\lrpar{Y;  \Hom(\cTR_Y(L) , \cG)},
\end{equation*}
where $\cG \lra Y$ is some auxiliary bundle, with associated domain
\begin{equation*}
	\cD_B(L) = \{ u \in \cD_{\max}(L) : B\circ \cC_Y(u)=0 \}.
\end{equation*}

For the de Rham operator $\eth_{\dR}$ of an incomplete edge metric (suitably scaled, see Lemma \ref{lem:SuitableScaling} below), 
$\cS\cD(\eth_{\dR})$ is nonempty if and only if $\hat X$ does {\it not} satisfy the {\em Witt} condition. In fact, from Corollary 
\ref{cor:DeRhamExp}, $\cS\cD(\eth_{\dR})$ contains a single element $\zeta_0$, and the corresponding coefficient 
$u_{\zeta_0}$ is a distributional section of the bundle 
\begin{equation*}
	\cH^{\mid}(\pa X/Y) \oplus \cH^{\mid}(\pa X/Y) \lra Y;
\end{equation*}
here $\cH^*(\pa X/Y)$ is the vertical Hodge bundle with respect to the induced metric and 
\begin{equation*}
	\mid = \tfrac12 \dim Z.
\end{equation*}
Thus when $\hat X$ is Witt, there is no need to impose boundary conditions, while if $\hat X$ is not Witt,
then a local ideal boundary condition is a bundle homomorphism
\begin{equation*}
	B \in \CI\lrpar{ Y; \Hom( \cH^{\mid}(\pa X/Y)\oplus \cH^{\mid}(\pa X/Y), \cG ) }.
\end{equation*}

We shall be primarily interested in very special local ideal boundary conditions
associated to arbitrary sub-bundles 
\begin{equation}\label{eq:SubBundle}
	W \lra Y \text{ of } \cH^{\mid}(\pa X/Y) \lra Y.
\end{equation}
Indeed, for any such $W$, if $W^{\perp}$ is the orthogonal complement with respect to the induced bundle 
metric, and 
\begin{equation*}
	\cP_W \in \CI(Y; \tend(\cH^{\mid}(\pa X/Y) ) )
\end{equation*}
is orthogonal projection onto $W,$ then we define the {\bf Cheeger ideal boundary condition} associated
to $W$ by 
\begin{equation*}
	B_W = ( \Id - \cP_W, \cP_W) \in \CI(Y; \tend( \cH^{\mid}(\pa X/Y) \oplus \cH^{\mid}(\pa X/Y) ) ).
\end{equation*}

The bundle 
\begin{equation*}
	\tH^*(\pa X/Y) \lra Y,
\end{equation*}
of de Rham cohomology groups of the fibers of the fibration $\pa X \lra Y,$
has a natural flat structure, \cite[Proposition 10.1]{Berline-Getzler-Vergne}, 
\cite[Proposition 3.4]{Bismut-Lott}, \cite[Proposition 14]{Hausel-Hunsicker-Mazzeo}. This is a general fact 
about the vertical cohomology of a fibration: if $M \xlra{\phi} B$ is a smooth fibration endowed with a 
connection $TM = T^HB \oplus TM/B,$ then there is a splitting of the differential forms on $M$ as 
\begin{equation*}
	\Omega^1(M) = \bigoplus_{p+q=1}\Omega^{p,q}(M),
\end{equation*}
where $p$ and $q$ are the horizontal and vertical degrees, respectively. The exterior differential $d_M$ decomposes as 
\begin{equation*}
	d_M = d^{0,1}_M + d^{1,0}_M + d^{2,-1}_M, \qquad 
	d^{j,k}: \Omega^{p,q}(M) \lra \Omega^{p+j, q+k}(M).
\end{equation*}
The first term, $d^{0,1}_M = d_{M/B}$, is the vertical de Rham operator on the fibers, the second, $d^{1,0}_M$, is the sum
of a lift of $d_B$ and a tensorial term built from the second fundamental form of the fibration, and the third,
$d^{2,-1}_M=R$, is purely tensorial and involves the curvature of the fibration.  The identity $d_M^2=0$ implies 
various identities among these terms, one of which implies that $d^{1,0}_M$ induces a flat connection on the vertical 
de Rham cohomology bundle
\begin{equation*}
	\tH^*(M/B) \lra B
\end{equation*}
and, in the presence of a Riemannian submersion metric, $d^{1,0}_M$ projects to a connection on the vertical Hodge cohomology bundle
\begin{equation*}
	\cH^*(M/B) \lra B.
\end{equation*}

For a simple edge space $\hat X$, the boundary fibration $\pa X \xlra{\phi} Y$ extends to small values of $x,$
\begin{equation*}
	\xymatrix{
	 Z \ar@{-}[r] & X \cap \{ x = c \} = \pa_c X \ar[r]^-{\phi_c} & Y. }
\end{equation*}
It is not hard to calculate how each of the individual terms in the decomposition of $d$, acting on $\ie$ forms, scale with $c$:
\begin{equation*}
	d^{0,1}_{\pa_c X} = \tfrac1c d_Z, \quad
	d^{2,-1}_{\pa_c X} = cR, \quad
	d^{1,0}_{\pa_c X} \text{ is independent of $c.$}
\end{equation*}
As before, $d^{1,0}_{\pa_c X}$ induces flat connections
\begin{equation*}
	\nabla^{\tH} \Mon \tH^*(\pa X/Y) \lra Y 
	\quad \Mand \quad
	\nabla^{\cH} \Mon \cH^*(\pa X/Y) \lra Y.
\end{equation*}

We shall call a choice of subbundle $W$ in \eqref{eq:SubBundle} a (Hodge) {\bf mezzoperversity} or {\bf flat structure} if it 
is parallel with respect to $\nabla^{\cH}$. We show in Theorem \ref{Thm:MainHodgeThm} below that given any 
mezzoperversity $W,$ the closed realization of the de Rham operator with Cheeger ideal boundary conditions,
\begin{equation*}
	\lrpar{ \eth_{\dR}, \cD_{B_W}(\eth_{\dR}) },
\end{equation*}
is a self-adjoint Fredholm operator on $L^2$ with compact resolvent.

\subsection{Depth two} \label{sec:TwoEdgeBC}
Now suppose that the stratification of the pseudomanifold $\hat X$ has two strata, one
contained in the closure of the other: 
\begin{equation*}
	Y^2 \subseteq \bar Y^1 \subseteq \hat X.
\end{equation*}
Thus the strata of $\hat X$ are 
\begin{equation*}
	Y^2, \quad Y^1 = \bar Y^1 \setminus Y^2, \quad \mbox{and} \quad X = \hat X \setminus \bar Y^1.
\end{equation*}
We denote by $Z^1$ the link of $Y^1$ and $Z^2$ the link of $Y^2$, both in $\hat X$.  Notice that $\bar Y^1$ is
a stratified space in its own right; in fact, it is a simple edge space with singular stratum $Y^2$, and we denote 
by $M$ its link in $\bar Y^1$.  Moreover, the link $Z^2$ is also a simple edge space with singular stratum
$M$ and such that the link of $M$ in $Z^2$ is $Z^1$.  Note finally that $Y^2$, $Z^1$ and $M$ are all smooth
closed manifolds. 

\begin{center}  \includegraphics[scale=0.4]{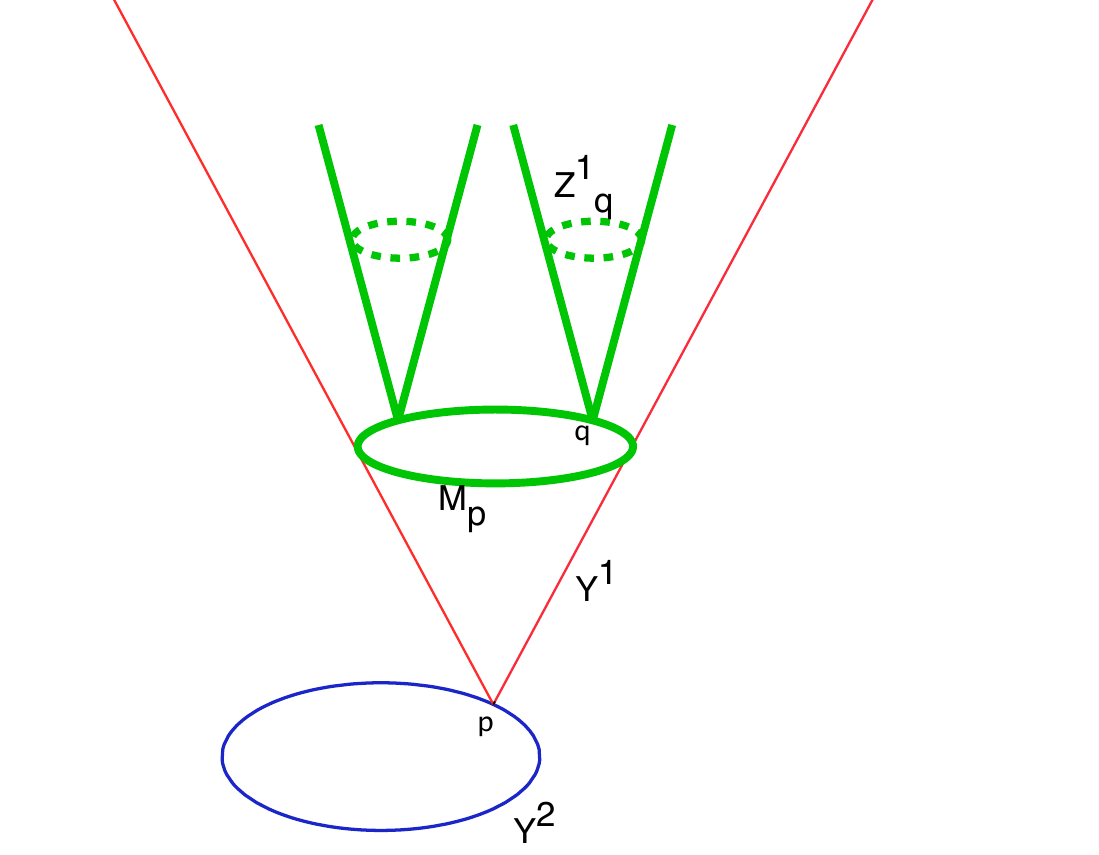}\end{center}

This figure is a schematic representation of a stratified space with two strata:
\begin{itemize}
\item the stratum $Y^2$ is the base circle; 
\item  the stratum $Y^1$ is the union of the large 
punctured cones 
$C_p\setminus\{p\}$, as $p$ varies within $Y^2$; the link of a point
$q$ in $Y^1$ is the closed manifold $Z^1_q,$ denoted by the small 
dotted circles;
\item 
the closure $\bar Y^1$ is the union of the large 
cones $C_p$ as $p$
varies within $Y^2$; the link of $\bar Y^1$ at $p,$ $M_p,$ is denoted by the large bold 
circle;
\item
the link of $\hat X$ at $p,$ $Z^2_p,$ is denoted in bold 
in the picture and consists of the cones $C(Z^1_q),$ as $q$ varies in $M_p;$
it is a depth one stratified space with singular stratum $M_p.$
\end{itemize}
Notice that
given  $q$ in $Y^1$  as in the picture, its link $Z^1_q$ is also equal to the link of $q$ viewed as
a point in the singular stratum $M_p$  of $Z^2_p.$ This exhibits 
the well known slogan ``the link of a link is a link."

\medskip
As explained in \cite{ALMP}, it is convenient to resolve $\hat X$ to a smooth manifold with corners, $\wt X$
with iterated fibration structure on its boundary. This resolution process is explained in great detail
in \S 2 of that paper. In this depth two case, $\wt X$ is a manifold with corners of codimension two 
obtained by radially blowing up first $Y^2$ and then the lift of $Y^1$: 
\begin{equation*}
	\wt X = [\hat X; Y^2; Y^1].
\end{equation*}
This space has two boundary hypersurfaces, one corresponding to $Y^1$ and the other to $Y^2$,
and each of these boundary faces fibers over the resolution of the (closure of the) corresponding stratum, 
\begin{equation*}
	H^i \xlra{\phi_i} \wt Y^i
\end{equation*}
with fiber $Z^1$ over $\wt Y^1$ and $\wt Z^2,$ the resolution of $Z^2,$ over $Y^2 = \wt Y^2.$ These fibrations are 
compatible at the corner $H^1 \cap H^2$ in the sense that they fit into a commutative diagram
\begin{equation*}
	\xymatrix{
	\pa \wt Z^2 \ar@{-}[rd] & & & & M \ar@{-}[ld] \\
	Z^1 \ar@{-}[r] & H^1 \cap H^2  \ar[rr]^-{\phi_1} \ar[rd]^-{\phi_2} & & \pa \wt Y^1 \ar[ld] & \\
	& & Y^2 & &}
	\xymatrix{ \\
	\text{ which induces } Z^1 \ar@{-}[r] & \pa \wt Z^2 \ar[r]^-{\phi_1} & M }
\end{equation*}

A Riemannian metric $g$ on $X$ is an {\em iterated incomplete edge (or $\iie$) metric} if near $Y^1$ it is an incomplete 
edge metric as defined above in \S\ref{sec:SimpleEdgeBC}, while near points of $Y^2$ it has the form
\begin{equation*}
	dx^2 + x^2g_{Z^2} + \phi_2^*g_{Y^2}
\end{equation*}
where $x$ is a boundary defining function for $H^2,$ $g_{Y^2}$ is a Riemannian metric on $Y^2$ and $g_{Z^2}$ restricts to an 
incomplete edge metric on each fiber of $\phi_2.$  As in the simple edge case, this can be interpreted as a metric on the 
iterated incomplete edge cotangent bundle $\Iie T^*X$ over all of $\wt X,$ see \cite[\S 3]{ALMP}. If $E$ and $F$ are vector 
bundles over $\hat X,$ then, following \cite{ALMP},  a differential operator $L$ of order one acting between sections of 
$E$ and $F$ is an {\em iie operator} or {\em iterated wedge operator} if 
\begin{itemize}
\item [i)] In local coordinates near $q \in Y^1,$ 
\begin{equation*}
	L = \frac1r \sum_{j+ |\alpha|+|\beta| \leq 1} a_{j,\alpha,\beta}(r,y,z) (r\pa_r)^j (r\pa_y)^{\alpha} (\pa_z)^{\beta}.
\end{equation*}
where $r$ is a bdf for $Y^1,$ $y_1, \ldots y_{h^1}$ are local coordinates along $Y^1$ and $z_1,\ldots, z_{v^1}$ are local 
coordinates along $Z_q^1.$ The coefficient functions $a_{j,\alpha,\beta}(r,y,z)$ are smooth on the resolved space $\wt X.$
Thus, restricted to a subset of $\hat X \setminus Y^2,$ $L$ is an $\iie$ operator as in 
\S \ref{sec:SimpleEdgeBC}.
\item [ii)]
In local coordinates near $q \in Y^2,$ 
\begin{equation}\label{Eq:LocCoord2}
	L = \frac1x \sum_{j+ |\alpha|+|\beta| \leq 1} a_{j,\alpha,\beta}(x,y,z) (x\pa_x)^j (x\pa_y)^{\alpha} (V_z)^{\beta}.
\end{equation}
where here $x$ is a bdf for $Y^2,$ $y_1, \ldots, y_{h^2}$ are local coordinates along $Y^2,$ and $V_z$ are $\phi_2$-vertical iie-operators on the fibers of $\phi_2$ of order at most one. The coefficient functions are again smooth on the resolved space.
\end{itemize}

\medskip

We start by analyzing elements of $\cD_{\max}(L)$ at points in $Y^1.$ 
Each point $q\in Y^1$ has a distinguished 
neighborhood $\cU_q$ in which $E$ and $F$ are trivial, and which is identified with 
$\bbB^{h_1} \times (0,1)_r \times Z^1_q$. Here $r$ is a bdf for $Y^1$ and $\bbB^{h_1}$ is the unit ball in $T_qY^1.$
This neighborhood is an (open) simple edge space and the restriction of $L$ to $\cU_q$ is an element of $\Diff^{1}_{\ie}(\cU_q).$
This restricted operator may then be analyzed as described above; in particular, we impose {\em Assumption \ref{Ass:CstIndicialRoots}}
there
so that elements of $\cD_{\max}(L|_{\cU_q})$ have distributional partial asymptotic expansions as $r\to0.$ The leading terms of any such element fit together to determine a section of a trace bundle 
\begin{equation*}
	\cTR_{Y^1}(L) \lra \wt Y^1.
\end{equation*}

We choose 
\begin{equation*}
	B^1 \in \CI\lrpar{\wt Y^1; \Hom(\cTR_{Y^1}(L); \cG^1)}
\end{equation*}
for some auxiliary bundle $\cG^1$ over $\wt Y^1,$
and then we 
impose the local ideal boundary condition $B^1$ for $L$
by defining
\begin{multline*}
	\cD_{\max, B^1}(L) = \Big\lbrace u \in L^2(X;E) : Lu \in L^2(X;F) \\ \left.
	\text{ and } \chi u \in \cD_{B^1}(L\rest{\hat X \setminus Y^2}) \ \mbox{for every} \ \chi \in \CIc( \hat X \setminus Y^2) \right\rbrace.
\end{multline*}
The key point is to analyze elements of $\cD_{\max, B^1}(L)$ near $Y^2.$  Thus, near any $q \in Y^2$ choose a 
distinguished neighborhood $\cU_q$ as before. Thus, $E$ and $F$ are trivial on this neighbourhood and 
\begin{equation*}
	\cU_q \cong \bbB^{h_2} \times (0,1)_x \times Z^2_q,
\end{equation*}
where $x$ is a bdf for $Y^2$ and $\bbB^{h_2}$ is the unit ball in $T_qY^2.$ In each such neighbourhood, 
\begin{equation*}
	\bar{\cU_q} \cap Y^2 = \bbB^{h_2}, \quad
	\cU_q \cap Y^1 = \bbB^{h_2} \times (0,1)_x \times M_q.
\end{equation*}
We prove below that elements of $\cD_{\max, B^1}(L)$ with support in $\cU_q$ have partial 
asymptotic expansions at $Y^1\cap \cU_q$ with coefficients in the null space of $B^1.$ 

If $L$ is as in \eqref{Eq:LocCoord2} in this distinguished neighborhood, then the `partially completed operator' has the form
\begin{equation*}
	P := x L = \sum_{j+ |\alpha|+|\beta| \leq 1} a_{j,\alpha,\beta}(x,y,z) (x\pa_x)^j (x\pa_y)^{\alpha} (V_z)^{\beta}.
\end{equation*}
It is natural to think of this operator $P$ near $Y^2$ as an edge differential operator `with values in the space of 
wedge differential operators on $Z^2$'. Thus we can proceed essentially just as in the simple edge case, via
model operators, with the proviso that we must keep track of the boundary conditions along $Y^1$. 

First define the family of normal operators of $P$ along $Y^2$ by
\begin{equation*}
	N_q(P) = \sum_{j+ |\alpha|+|\beta| \leq 1} a_{j,\alpha,\beta}(0,q,z) (s\pa_s)^j (s\pa_u)^{\alpha} (V_z\rest{x=0,y=q})^{\beta}.
\end{equation*}
This an element of $x \Diff^1_{\iie}(\bbR^+_s\times \bbR^{h_2}_u \times Z^2_q; \pi_q^*E, \pi_q^*F)$ for each $q \in Y^2$. 

There is a stratification of $\bbR^+_s \times T_q Y^2 \times Z^2_q$; the face $\{ s= 0\}$, which is just $T_qY^2 \times Z^2_q$ 
resolves the depth two stratum, and has link (or fiber) $Z^2_q,$ while the resolution of the depth one stratum is 
$\{ r= 0\}$, or equivalently, $\bbR^+_s \times T_qY^2 \times M$, which has link $Z^1.$ 
An $\iie$ metric on $\hat X$ and Hermitian metrics on $E$ and $F$ induce corresponding objects on 
$\bbR^+_s \times T_qY^{2}\times Z^2_q$. 
At an interior point $q \in \{ r=0, s \neq 0 \}$, $N_q(P)$ is an incomplete edge operator. By the same analysis as 
described earlier, see also Lemma \ref{lem:ExpansionExistence} below, elements of $\cD_{\max}(N_q(P))$
have distributional expansions as $r \searrow 0$, 
\begin{equation*}
	v \in \cD_{\max}(N_q(P)) \implies 
	v \sim \sum_{\zeta_j \in \cS\cD(N_q(P))} v_{\zeta_j}(u,z) s^{\zeta_j} + \wt v,
\end{equation*}
with error term $\wt v \in s^{1-}H^{-1}_e(\bbR^+_s \times T_qY^{2}\times Z^2_q;\pi_q^*E).$  Note that by Assumption 2, 
all indicial roots are simple so this expansion and the ones below do not contain logarithmic powers.
The indicial family of $N_q(P)$ is 
$I_q(P;\zeta)$, so 
\begin{equation*}
	\Spec_b(N_q(P)) = \Spec_b(P;q) \quad \Mand \quad \cS\cD(N_q(P)) = \cS\cD(P,q).
\end{equation*}
The coefficients are distributional 
sections of the trace bundle,
\begin{equation*}
	\cTR_{\bbR^+_s \times T_qY^{2}\times M}  
	\lra \bbR^+_s \times T_qY^{2} \times M.
\end{equation*}
Unwinding the definitions, this is identified with the pull-back and restriction of $\cTR_{Y^1}(L) \lra Y^1.$ The Cauchy data 
map assigns to $u \in \cD_{\max}(N_q(P))$ the coefficients in its expansion, 
\begin{equation*}
	\cC_{\bbR^+_s \times T_qY^{2} \times M}: \cD_{\max}(N_q(P)) \lra \CmI\lrpar{ \bbR^+_s \times T_qY^{2} \times M; \pi_q^*\cTR_{Y^1}(L) }.
\end{equation*}

Now consider the bundle homomorphism $B^1$ on $Y^1$ which determines the boundary conditions there.  Recalling
that $\bar Y^1$ is a simple edge space with singular stratum $Y^2$ and link $M$, and that $B^1$ is
a zero order (edge) operator, we analyze its normal operator at points of $Y^2$, 
\begin{equation*}
	Y^2 \ni q \mapsto N_q(B^1) \in \CI(\bbR^+_s \times T_qY^{2} \times M; \Hom(\pi_q^*\cTR_{Y^1}(L), \pi_q^*\cG^1) ).
\end{equation*}
This defines a domain at the level of normal operators consistent with $\cD_{\max, B^1}(L)$:
\begin{multline*}
	\cD_{\max, N_q(B^1)}(N_q(P)) = \{ u \in L^2(\bbR^+_s \times T_qY^{2} \times Z^2_q;\pi_q^*E) : \\
	N_q(P)u \in L^2(\bbR^+_s \times T_qY^{2} \times Z^2_q;\pi_q^*F) \Mand N_q(B^1) \circ \cC_{\bbR^+_s \times T_qY^{2} \times M}(u) =0 \}.
\end{multline*}
The partially completed operator $P$ has indicial family at point $q \in Y^2$,  given in local coordinates as
\begin{equation*}
	\Diff^1_{\ie}(Z^2_q; \pi_q^*E, \pi_q^*F) \ni I_q(P;\zeta) = \sum_{j+ |\beta| \leq 1} a_{j,0,\beta}(0,q,z) (\zeta)^j (V_z\rest{x=0,y=q})^{\beta}.
\end{equation*}
This has domain 
\begin{multline*}
	\cD_{\max, I_q(B^1)}(I_q(P;\zeta)) = \\ \{ u \in L^2(Z^2_q;\pi_q^*E) : 
	I_q(P;\zeta)u \in L^2(Z^2_q;\pi_q^*F) \Mand I_q(B^1)\circ \cC_{M}(u) =0 \}.
\end{multline*}
Here $\cC_M$ is the Cauchy data map at $M \subseteq Z^2$, which is defined as before, recalling that $Z^2$ is 
a simple edge space with $\iie$ metric. 

We now impose the 
\begin{assume} [Fredholm indicial family] \label{AssIndFamily}
For each $q\in Y^2,$ the family of indicial operators 
\begin{equation*}
	(I_q(P;\zeta), \cD_{\max, I_q(B^1)}(I_q(P;\zeta)) )
\end{equation*}
is Fredholm on $L^2(Z^2_q)$ and there is some $\zeta$ for which this operator is invertible.
\end{assume}
We prove in Theorem \ref{Thm:MainHodgeThm} that this holds for the de Rham operator. 

Since this is a holomorphic family of Fredholm operators, this assumption provides the extra piece of information needed
to apply the analytic Fredholm theorem, which implies that the family of inverses $(I_q(P;\zeta), 
\cD_{\max, I_q(B^1)}(I_q(P;\zeta)) )^{-1}$ is meromorphic.  Hence for each $q$, its poles are discrete and we will continue to assume that they are simple. This meromorphic
family determines the sets
\begin{equation*}
\begin{gathered}
	\spec_b(P;B^1,q) = \{ \zeta \in \bbC : 
	(I_q(P;\zeta), \cD_{\max, I_q(B^1)}(I_q(P;\zeta)) ) \text{ does not have an $L^2$-bounded inverse} \} \\
	\cS\cD(P, B^1, q) = \{ \zeta_j \in \spec_b(P;B^1,q) \text{ such that } x^{\zeta_j}\in L^2_{\loc} \setminus  xL^2_{\loc} \}.
\end{gathered}
\end{equation*}

If $\cS\cD(P,B^1,q)$ remains disjoint from all other indicial roots of $L$ as $q$ varies, then 
by
\cite[Theorems 3.2 and 6.3]{Krainer-Mendoza:Kernel} there is a smooth vector bundle 
\begin{equation*}
	\cTR_{Y^2}(L;B^1) \lra Y^2,
\end{equation*}
the smooth sections of which are
\begin{equation*}
	\Big\lbrace \sum_{ \zeta_j \in \cS\cD(P, B^1, q)} 
	u_{\zeta_j}(y,z) x^{\zeta_j} : u_{\zeta_j} \in \CI(Z^2_q; \pi_q^*E) \Big\rbrace.
\end{equation*}

We impose again the {\em Assumption \ref{Ass:CstIndicialRoots} } that the set $\cS\cD(P,B^1,q)$ is independent of $q \in Y^2$ and that the indicial roots in this strip are simple. Lemma \ref{lem:ExpansionExistence} below asserts that 
\begin{equation}\label{eq:ExpY2}
	u \in \cD_{\max,B^1}(L) \implies u \sim \sum_{\zeta_j \in \cS\cD(P; B^1)} u_{\zeta_j}(y,z) x^{\zeta_j} + \wt u,
\end{equation}
where this sum is a distributional section of $\cTR_{Y^2}(L;B^1)$ and $\wt u \in xH^{-1}_e(X;E).$ We again 
identify sections of $\cTR_{Y^2}(L;B^1)$ with the (finite) set of coefficients $u_{\zeta_j}$, and then define
the Cauchy data map of $L$ at $Y^2$ by 
\[
\mathcal C_{Y^2} = \mathcal C_{Y^2}(L;B^1): \cD_{\max,B^1}(L)  \ni u \longrightarrow 
( u_{\zeta_j} )  \in \mathcal C^{-\infty}( Y^2; \cTR_{Y^2}(L;B^1) ).
\]

We can now define local ideal boundary conditions at $Y^2$ by considering a bundle homomorphism,
\begin{equation*}
	B^2 \in \CI\lrpar{ Y^2; \Hom(\cTR_{Y^2}(L;B^1), \cG^2 ) }
\end{equation*}
where $\cG^2$ is an auxiliary bundle. The corresponding domain for $L$ is 
\begin{equation*}
	\cD_{(B^1,B^2)}(L) = \lrbrac{ u \in \cD_{\max, B^1}(L):  B^2 \circ \cC_{Y^2} u = 0 }.
\end{equation*}
We refer to the pair $(B^1, B^2)$ as {\bf local ideal boundary conditions} for $L$.
We show later that the operator 
\begin{equation*}
	(L, \cD_{(B_1, B_2)}(L))
\end{equation*}
is closed for any choice of local ideal boundary conditions. 

\medskip

We now specialize this discussion to where $L$ is the de Rham operator of an $\iie$ metric (scaled as in 
Lemma \ref{lem:SuitableScaling}). If $\hat X$ is Witt, there is no need to impose boundary conditions. 
If either $Y^1$ or $Y^2$, but not both, does not satisfy the Witt condition, then we can proceed just as in 
\S\ref{sec:SimpleEdgeBC}.  Hence we turn immediately to the case where the Witt condition fails at
both $Y^1$ and $Y^2$.  The leading coefficient of an element of the maximal domain is a distributional section of the bundle
\begin{equation*}
	\cH^{\mid}(H^1/\wt Y^1) \oplus \cH^{\mid}(H^1/\wt Y^1) \lra \wt Y^1,
\end{equation*}
over $\wt Y^1,$ the resolution of $Y^1.$
Each summand here has a flat connection. Suppose that $W^1$ is a flat sub-bundle of $\cH^{\mid}(H^1/\wt Y^1) \lra \wt Y^1$
with orthogonal projection $\cP_{W^1}$. Then we can impose the Cheeger ideal boundary condition
\begin{equation*}
	B_{W^1} = (\Id - \cP_{W^1}, \cP_{W^1}) \in \CI( \wt Y^1;
	\tend(\cH^{\mid}(H^1/\wt Y^1) \oplus \cH^{\mid}(H^1/\wt Y^1) ) ).
\end{equation*}

At the next stratum $Y^2,$ the leading coefficient of an element in
\begin{equation*}
	\cD_{\max, B_{W^1}}(\eth_{\dR})
\end{equation*}
is a distributional section of the bundle
\begin{equation}\label{eq:TraceBdleY2}
	\cH^{\mid}_{W^1}(H^2/Y^2) \oplus \cH^{\mid}_{W^1}(H^2/Y^2)  \lra Y^2.
\end{equation}
The typical fiber in either of these summands is
\begin{equation*}
	\cH_{W^1}^{(\dim Z^2)/2}(Z_q^2) = 
	\ker \lrpar{ \eth_{\dR}^{Z^2_q}, \cD_{B_{W^1}}(Z^2_q) }  \bigcap L^2(Z^2_q; \Lambda^{(\dim Z^2)/2}({}^{\iie}T^*Z_q^2)).
\end{equation*}
As these are the null spaces of a smooth family of Fredholm operators, the fact that they form a bundle is equivalent to the fact that the dimension is constant.
This follows from Theorem \ref{thm:SelfDual}, which identifies these Hodge cohomology spaces with de Rham cohomology spaces and Theorem \ref{thm:StratDiffInv} which shows that the latter are invariant under pull-back by stratified diffeomorphism.
This also follows by combining
Theorem \ref{Thm:MainHodgeThm}, which shows that {\em Assumption \ref{AssIndFamily}} holds, with \cite{Krainer-Mendoza:Kernel}.

The bundle $\cH_{W^1}^{\mid}(H^2/Y^2)$ has a natural flat connection, obtained as in \S\ref{sec:SimpleEdgeBC} from the 
exterior derivative on $H^2$. Since the vertical cohomology is now an $L^2$-cohomology, we discuss this more closely.
Choose a small collar neighborhood $\Coll(H^2) \cong H^2 \times [0,\eps]_x$ of $H^2$ in $\wt X$ and extend the boundary fibration
\begin{equation*}
	\lrpar{Z^2 \times [0,\eps]_x} - \Coll(H^2) \xlra{\Coll(\phi)} Y^2.
\end{equation*}
Using the metric on $\hat X$, we obtain an orthogonal splitting of the tangent bundle
\begin{equation}\label{eq:SplittingCollH^2}
	\Iie T\Coll(H^2) \cong \Iie TH^2 \oplus T[0,\eps]_x \cong  \Iie TY^2 \oplus \Iie TZ^2 \oplus T[0,\eps]_x.
\end{equation}
This induces a decomposition of $d$ on differential forms supported in $\Coll(H^2)$:
\begin{equation}\label{eq:dNearBdy}
	d\rest{\Coll(H^2)} = 
	\begin{pmatrix}
	\tfrac1xd^{Z^2} + d^{1,0}_{H^2} + xR & 0 \\
	\pa_x + \tfrac1x \bN & -(\tfrac1xd^{Z^2} + d^{1,0}_{H^2} +xR),
	\end{pmatrix}
\end{equation}
where $d^{1,0}_{H^2}$ is independent of $x.$

The graph closure of $d$ with core domain $\cD_{\max, B_{W^1}}(\eth_{\dR})$ has domain
\begin{equation*}
	\cD_{\max, B_{W^1}}(d),
\end{equation*}
see also \S\ref{sec:L2Coho} for an alternate description.  We also consider 
\begin{equation*}
	\cE_c^t(d\rest{\Coll(H^2)}) \subset \cD_{\max,B_{W^1}}(d); 
\end{equation*}
this is the space of forms with support in $\Coll(H^2)$ projecting down to a compact subset 
of $Y^2,$ and which annihilate $\pa_x$. 
As in \cite{Bismut-Lott} there is a natural surjective quotient map
\begin{equation*}
	\cE_c^t \cap \ker d^{Z^2}
	\xlra{\psi} 
	\CIc( Y^2; \Lambda^*(\Iie T^*Y^2) \otimes \tH_{W^1}^*( H^2/Y^2 ))
\end{equation*}
which defines a differential operator (for any $m \in \bbN_0$) by
\begin{equation*}
\xymatrix @R=1pt{
	\nabla^{\tH} : 
	\CIc( Y^2; \Lambda^p(\Iie T^*Y^2) \otimes \tH_{W^1}^m( H^2/Y^2 ) )
	\ar[r] &
	 \CIc( Y^2; \Lambda^{p+1}(\Iie T^*Y^2) \otimes \tH_{W^1}^m( H^2/Y^2 ) ) \\
	 \omega = \psi(e) \ar@{|->}[r] & \psi(d^{1,0}_{\Coll(H^2)} e) }
\end{equation*}
It follows from $(d\rest{\Coll(H^2)})^2=0$ that this is well-defined independently of the choice of $e$ lifting $\omega,$ and that
$(\nabla^{\tH})^2=0.$ The Leibniz rule satisfied by $d\rest{\Coll(H^2)}$ induces a Leibniz rule for $\nabla^{\tH},$ so we see that $\nabla^{\tH}$ is a flat connection on the bundle $ \tH_{W^1}^m( H^2/Y^2 )  \lra Y^2.$

\begin{remark}
We have defined the action of $\nabla^{\tH}$ on compactly supported sections, but one can now extend it to arbitrary smooth sections. Alternately, it is easy to see that the map $\psi$ extends to a surjective quotient map
\begin{equation*}
	\cD_{\max, B_{W^1}}(d\rest{\Coll(H^2)}) \cap \mathrm{Ann}(\pa_x) \cap \ker d^{Z^2}
	\lra
	\cD_{\max}(\nabla^{\tH})
\end{equation*}
(where $\mathrm{Ann}(\pa_x)$ refers to the annihilator of $\pa_x$) and so the definition of $\nabla^{\tH}$ above 
extends naturally to this domain.
\end{remark}

There is a similar decomposition of $\delta$ on forms supported in $\Coll(H^2),$
\begin{equation}\label{eq:deltaNearBdy}
	\delta\rest{\Coll(H^2)} = 
	\begin{pmatrix}
	\delta^{Z^2} + (d_{\Coll(H^2)}^{1,0})^* + (xR)^* & \pa_x + \tfrac1x(\bN -f) \\
	0 & -(\delta^{Z^2} + (d_{\Coll(H^2)}^{1,0})^* + (xR)^* )
	\end{pmatrix},
\end{equation}
and finally, a natural identification 
\begin{equation*}
	\cE_c^t(d\rest{\Coll(H^2)}) \cap \ker d^{Z^2} \cap \ker \delta^{Z^2}
	\lra
	\CIc(Y^2, \Lambda^*(\Iie T^*Y^2) \otimes \cH_{W^1}^*(H^2/Y^2)).
\end{equation*}
This allows us to define a flat connection $\nabla^{\cH}$ on the bundle of vertical harmonic forms
\begin{equation*}
	 \cH_{W^1}(H^2/Y^2) \lra Y^2,
\end{equation*}
namely, $\nabla^{\cH} = \Pi_0 d_{H^2}^{1,0} \Pi_0$ where $\Pi_0$ is the orthogonal projection onto $ \ker d^{Z^2} \cap \ker \delta^{Z^2}.$
The Hodge map induces a bundle isomorphism
\begin{equation*}
	 \cH_{W^1}^*(H^2/Y^2)
	 \lra  \tH_{W^1}^*(H^2/Y^2),
\end{equation*}
which also identifies $\nabla^{\cH}$ with $\nabla^{\tH}$ (\cite[Proposition 3.14]{Bismut-Lott}).\\

A local ideal boundary condition at $Y^2$ is a bundle homomorphism acting on sections of the bundle \eqref{eq:TraceBdleY2}.
We associate to each flat sub-bundle 
\begin{equation*}
	W^2 \subseteq \cH^{\mid}_{W^1}(H^2/Y^2)
\end{equation*}
the Cheeger ideal boundary condition
\begin{equation*}
	B_{W^2} = ( \Id - \cP_{W^2}, \cP_{W^2}) \in \CI(Y^2; \tend(\cH^{\mid}_{W^1}(H^2/Y^2) \oplus \cH^{\mid}_{W^1}(H^2/Y^2) ) )
\end{equation*}

The pair 
\begin{equation*}
\begin{gathered}
	W^1, \text{a flat subbundle of }\cH^{\mid}(H^1/\wt Y^1),\\
	W^2, \text{a flat subbundle of }\cH^{\mid}_{W^1}(H^2/Y^2)
\end{gathered}
\end{equation*}
constitute a (Hodge) {\bf mezzoperversity}. We show in Theorem \ref{Thm:MainHodgeThm} that for any such 
mezzoperversity the de Rham operator with Cheeger ideal boundary conditions 
\begin{equation*}
	\lrpar{ \eth_{\dR}, \cD_{B^1, B^2}(\eth_{\dR}) }
\end{equation*}
(with $B^j = B_{W^j}$)
is a self-adjoint Fredholm operator on $L^2$ with compact resolvent.
 
\subsection{Arbitrary depth}\label{sec:ArbitraryDepthBVP}
Now consider a space $\hat X$ which has singular strata up to depth $k+1$. One reason for our elaboration of the depth
two case is that this discussion extends to the higher depth almost unchanged. Order the strata $Y^1, Y^2, \ldots, Y^{k+1}$
so that $\mathrm{depth}\, Y^j < \mathrm{depth}\, Y^{j+1}$, $j < k$.  Thus, at one end, $Y^1$ is a stratum of least depth and
its link is a smooth closed manifold, while $Y^k$ is a smooth closed manifold which has link a stratified pseudomanifold 
with depth $k-1$.  We assume inductively that we have defined local ideal boundary conditions for all spaces with depth 
less than or equal to $k$.  We now describe how to extend this definition to $Y^{k+1}$.  As before, we let $x$ denote
a boundary defining function for the hypersurface corresponding to this stratum in the resolution $\wt X$, and 
work with the partially completed operator $P = xL$. 

We use now the notation of \cite{ALMP} for iterated edge and iterated incomplete edge operators.
Let $L \in \Diff_{\iie}^1(X;E)$ be an elliptic $\iie$ operator of order one acting on sections of a vector bundle $E.$
The local ideal boundary conditions at $Y^1, \ldots, Y^k$ are bundle homomorphisms
\begin{equation}\label{Eq:LocalBC1k}
\begin{split}
	B^1 & \in \CI(\wt Y^1; \Hom(\cTR_{Y^1}(L), \cG^1 ) )\\
	B^2 & \in \CI\lrpar{ \wt Y^2; \Hom ( \cTR_{Y^1}(L;B^1), \cG^2 )}\\
	& \vdots \\
	B^k & \in \CI\lrpar{ \wt Y^k; \Hom( \cTR_{Y^k}(L;B^1, \ldots, B^{k-1}), \cG^k )}
\end{split}
\end{equation}
where each $\cG^j$ is an auxiliary bundle over $\wt Y^j$.  We always impose Assumptions \ref{Ass:Localizable}, \ref{Ass:CstIndicialRoots}
and \ref{AssIndFamily} at each stratum. Recall also that $\wt Y^j,$ $j>2,$ are themselves manifolds with corners with iterated 
fibration structures.

The indicial and normal operators for the elements of $\bB = (B^1, \cdots, B^k)$ define local ideal boundary conditions 
$I_q(\bB)$ for $I_q(P;\zeta)$ and $N_q(\bB)$ for $N_q(P)$, for every $q \in Y^{k+1}$.  Defining $\spec_b(L;\bB)$ at $Y^{k+1}$
as before, we write
\begin{equation}\label{eq:cScDDef}
	\cS\cD_{k+1}(L,\bB) = \{  \zeta_j \in \spec_b(L; \bB): x^{\zeta_j} \in L^2_{\loc}\setminus xL^2_{\loc} \} \ \ \subset\ \ 
\{\zeta: 0 < \Re \zeta + \tfrac{f+1}2 \leq 1\}. 
\end{equation}
Using Assumption \ref{AssIndFamily}, there are smooth bundles 
\[
\ker (I(P; \zeta_j);  I(\bB)) \longrightarrow \wt Y^{k+1}
\]
for every $\zeta_j \in \cS\cD_{k+1}(L, \bB)$. For simplicity, if the meaning is clear, the subscript $k+1$ is omitted 
from $\cS\cD$. The direct sum of these over all $\zeta_j \in \cS\cD(L, \bB)$ gives a smooth vector bundle 
\begin{equation*}
	\cTR_{Y^{k+1}}(L;\bB) \lra \wt Y^{k+1};
\end{equation*}
this is a simple instance of \cite[Theorems 3.2 and 6.3]{Krainer-Mendoza:Kernel}. 

 Choose an adapted neighborhood $\cU_q \cong \bbB^{h_{k+1}} \times (0,1)_{x} \times Z_q^{k+1} \subseteq \hat X$. 
Lemma \ref{lem:ExpansionExistence} below shows that elements of $\cD_{\max, \bB}(L)$ have partial distributional 
asymptotic expansions at $Y^{k+1},$
\begin{equation}\label{eq:ExpYk1}
u \in \cD_{\max,\bB}(L) \implies u \sim \sum_{ \zeta_j \in \cS\cD(L, \bB) } u_{\zeta_j}(y,z) x^{\zeta_j} + \wt u  
\end{equation}
where $u_{\zeta_j} \in H^{-\Re \zeta_j}(\bbB^{h_{k+1}}; \ker (I(P; \zeta_j, I(\bB))$ and $\wt u \in x^{1-}H^{-1}_e(X;E)$.
This allows us to define the Cauchy data map of $L$ at $Y^{k+1}$, 
\begin{equation*}
	\xymatrix @R=1pt @C=80pt{
	\cD_{\max,\bB}(L) \ar[r]^-{\cC_{Y^{k+1}} = \cC_{Y^{k+1}}(L;\bB)} &  \CmI\lrpar{ \wt Y^{k+1}; \cTR_{Y^{k+1}}(L;\bB) } \\
	 u \ar@{|->}[r] &  ( u_{\zeta_j} )  }
\end{equation*}

Local ideal boundary conditions at $Y^{k+1}$ are defined by a bundle homomorphism 
\begin{equation*}
	B^{k+1} \in \CI\lrpar{ \wt Y^{k+1}; \Hom( \cTR_{Y^{k+1}}(L;\bB), \cG^{k+1} )},
\end{equation*}
with $\cG^{k+1} \lra \wt Y^{k+1}$ an auxiliary bundle. We complete the inductive step by setting
\begin{equation*}
	\cD_{\max, (B^1, \ldots, B^{k+1})} = \{ u \in \cD_{\max, \bB}(L) : B^{k+1} \circ \cC_{Y^{k+1}}(u)=0 \}.
\end{equation*}

The collection of bundle homomorphisms $\bB = (B^1, \cdots, B^k)$ constitutes a {\bf local ideal boundary condition} for $L.$
Theorem \ref{Thm:ClosedDomains} shows that for any such choice of $\bB$, 
\begin{equation*}
	(L, \cD_{\bB}(L))
\end{equation*}
is a closed operator. \\

Specialize again to the de Rham operator for a scaled $\iie$ metric. Generalizing the discussion in \S\ref{sec:TwoEdgeBC} we 
make the 
\begin{definition}\label{Def:FlatSystem}
Let $(\hat X,g)$ be a stratified pseudomanifold and $Y^1, \ldots, Y^{k+1}$ an ordering of the strata of $\hat X$ with increasing depth.
A (Hodge) {\bf mezzoperversity} is a collection of bundles 
\begin{equation*}
	\cW = \{ W^1 \lra \wt Y^1, \ldots, W^{k+1} \lra \wt Y^{k+1} \}, \quad \Mwith
\end{equation*}
\begin{equation*}
\begin{split}
	&W^1, \text{ a flat subbundle of }\cH^{\mid}(H^1/\wt Y^1),\\
	&W^2, \text{ a flat subbundle of }\cH^{\mid}_{W^1}(H^2/\wt Y^2), \\
	& \phantom{x}\vdots \\
	& W^{k+1}, \text{ a flat subbundle of }\cH^{\mid}_{W^1, \ldots, W^{k}}(H^{k+1}/\wt Y^{k+1}). 
\end{split}
\end{equation*}
\end{definition}
Here $\cH^{\mid}_{ \lrpar{W^1, \ldots, W^{j-1}} }(H^j/\wt Y^j) \lra \wt Y^j $ is the bundle with fiber at $q\in \wt Y^j$
\begin{equation*}
	\cH^{\tfrac12 \dim Z^j}_{ \lrpar{W^1, \ldots, W^{j-1}} }(Z_q^j)
	= \ker \lrpar{ \eth_{\dR}^{Z_q^j}, \cD_{B^1, \ldots, B^{j-1}}(\eth_{\dR}^{Z^j_q}) }
	\bigcap L^2(Z_q^j ; \Lambda^{\tfrac12\dim Z^j}({}^{\iie}T^*Z^j_q)),
\end{equation*}
with $B^j = B_{W^j}.$
As before, that these spaces form a bundle follows from \cite[Theorems 3.2 and 6.3]{Krainer-Mendoza:Kernel} 
once we establish the appropriate versions of {\em Assumption \ref{AssIndFamily}}, but also follow directly from 
Theorem \ref{thm:SelfDual} where these Hodge cohomology groups are identified with de Rham cohomology groups. 
The discussion of the flat connection in \S\ref{sec:TwoEdgeBC} extends to the general case essentially unchanged.

Every mezzoperversity determines Cheeger ideal boundary conditions for the de Rham operator on $X,$
\begin{equation*}
	\lrpar{ \eth_{\dR}, \cD_{B^1, \ldots, B^{k+1}}(\eth_{\dR})},
\end{equation*}
and we shall prove that this is a closed, self-adjoint Fredholm operator on $L^2,$ with compact resolvent.

\section{The model operators} \label{sec:ModelOps}
\subsection{The distributional asymptotic expansion}
We work as usual with a stratified space $\hat X$ with singular strata $Y^1, \ldots, Y^{\ell+1}$ ordered with increasing depth,
an $\iie$ metric $g$ on $X = \hat X^{\reg}$, and an elliptic $\iie$ differential operator $L \in \Diff^{1}_{\iie}(X;E)$.
For simplicity we assume that $L$ acts on sections of a single bundle $E$. 
We will show that imposing local ideal boundary conditions at the first $k$-strata allows one to impose local ideal boundary conditions at $Y^{k+1}.$\\

Assume that local ideal boundary conditions 
\begin{equation*}
	\bB = (B^1, \ldots, B^k)
\end{equation*}
have been chosen at the first $k$ strata, and that $L$ satisfies Assumptions \ref{Ass:Localizable}, \ref{Ass:CstIndicialRoots} 
and \ref{AssIndFamily} at $Y^{k+1}.$ We show now that elements in $\cD_{\max, \bB}(L)$ have distributional asymptotic 
expansions \eqref{eq:ExpYk1} at $Y^{k+1}$. 

To study expansions of sections at a singular stratum it will be useful to have a notation for functions supported near point on a stratum.
\begin{definition}
Let $A(Y^j) \subseteq \CI_{\Phi}(\wt X)$ be the subset of functions supported in a distinguished neighborhood $\cU_q$ of a point $q \in Y^j$ (i.e., a neighborhood of $q$ in which the boundary fibration structure is trivial, so that $\cU_q$ is diffeomorphic to $\bbB^{h_j} \times (0,1)_r \times \wt Z^j_q$) over which $E$ is trivial.
\end{definition}

A well-known approach to asymptotic expansions is through the Mellin transform. To apply it in our context, we identify a distinguished neighborhood $\cU_q$ of a point 
$q \in Y^{k+1}$ with $\bbB^{h_{k+1}} \times (0,1)_{x} \times Z_q^{k+1}$ and we fix a trivialization of $E,$ $E\rest{\cU_q} \cong \cU_q \times E\rest{Z_q^{k+1}}$.   
We also fix the measure $\frac{dx}xdy\dvol_{Z}$ on $\cU_q$, and set 
\begin{equation*}
	\gamma = -\frac{f+1}2,
\end{equation*}
so that $L^2 = L^2(\cU_q,g) = x^\gamma L^2(\cU_q,\frac{dx}xdy\dvol_Z).$ As before, we work with the partially completed operator
$P = x L$. 

\begin{lemma}\label{lem:ExpansionExistence}
If $u \in \cD_{\max, \bB}(L)$ and $\mathrm{supp}\, u \subset \cU_q$ then, as $x \to 0,$ 
\begin{equation}\label{Eq:Expansion}
	u \sim \sum_{\zeta_j \in \cS\cD(L, \bB) } u_{\zeta_j}(y,z) x^{\zeta_j} + \wt u,
\end{equation}
with $\cS\cD(L, \bB)$ from \eqref{eq:cScDDef} and with
\begin{equation*}
\begin{gathered}
	u_{\zeta_j}(y,z) \in H^{-\Re \zeta_j}(\bbB^{h_{k+1}}; L^2(Z^{k+1}_q ; E\rest{Z_q^{k+1}})), \\
	\wt u \in x^{1-} H^{-1}(\bbB^{h_{k+1}}, L^2( C(Z_q^{k+1}); E\rest{Z_q^{k+1}} )).
\end{gathered}
\end{equation*}
 The coefficients $(u_{\zeta_j})$ 
and the remainder $\wt u$ all depend continuously on $u \in \cD_{\max, \bB}(L)$ (with respect to the graph norm).

The coefficients $\{ u_{\zeta_j } \}$ are the images of $u$ under the projectors which are the leading coefficients in the Laurent 
expansion of the inverse of the indicial family at $\zeta_j$;  these projectors are finite rank on $L^2(Z^{k+1}_q)$ for each $y \in \bbB^{h_{k+1}}$,
and hence are really sections over $\bbB^{h_{k+1}}$ of the kernel of the indicial operator of $xL,$
\begin{equation*}
	u_{\zeta_j}(y,z) \in H^{-\Re \zeta_j}(\bbB^{h_{k+1}}; \ker (I(xL;\zeta_j), \cD_{I(\bB)}(I(xL;\zeta_j) ) ) ).
\end{equation*}
Note that the fact that these kernels fit together to form a vector bundle follows from the fact that the trace bundle is a bundle and the constancy of the indicial roots.
\end{lemma}

\begin{remark}
The remainder term is in general in the space above with weight $x^{1-},$ because it is possible for there to be an indicial root $\zeta$ such that 
\begin{equation*}
	x^\zeta \in x^{1-}L^2_{\loc} \setminus L^2_{\loc}
\end{equation*}
(e.g., this will happen for the de Rham operator at a `Witt stratum'). If this does not happen, then the remainder term satisfies
\begin{equation*}
	\wt u \in x H^{-1}( \bbB^{h_{k+1}}, L^2( C(Z_q^{k+1}); E\rest{Z_q^{k+1}} )).
\end{equation*}
\end{remark}

\begin{proof}
Following \cite[\S4.2]{ALMP}, a variant of the standard symbol calculus produces a ``small calculus'' parametrix $A$ 
for $\rho L$, where $\rho$ is a product of boundary defining function for each of the hypersurface boundaries 
of $\wt X$.  This does not require boundary conditions and has the following properties: if $u = \calO(x^\nu)$, then $Au = \calO(x^\nu)$ for any $\nu > -\gamma$; furthermore, if 
\begin{equation*}
	Q = \Id - A\rho L
\end{equation*}
and $w = Qu$, then $w$ is smooth in $X$, and $(x\pa_x)^i (x\pa_y)^j w \in L^2$ for all $i, j$. Writing $u=A\rho Lu + Qu$, then
these properties show that the term $A \rho Lu$ can be included into the error term $\wt u$. Hence we may assume
for the remainder of the proof that $u \in \mathrm{Image}(Q).$
To simplify notation we also replace $h_{k+1}$ and $Z^{k+1}_q$ by $h$ and $Z$ for the duration of this section.
 
Since $u \in L^2 = x^\gamma L^2(\cU_q,\frac{dx}xdy\dvol_Z)$, its Mellin transform in $x$, $\cM u(\zeta)$, takes values in
\begin{equation*}
	L^2( \{ \eta = \gamma\}, d\xi; L^2(\bbB^h \times Z; E) ), \qquad \zeta = \xi + i\eta.
\end{equation*}
This extends to a holomorphic function in the region 
\begin{equation*}
	\{ \eta < \gamma \} \lra L^2(\bbB^h \times Z; E),
\end{equation*}
and satisfies the ideal boundary conditions for $I(P)$ induced by $\bB$. 

Writing $P = I(P) + S$ so that $I(P)u = Pu - Su$, 
and using Assumption \ref{AssIndFamily}, we have that 
\begin{equation}\label{eq:HoloEquality}
	\cM u(\zeta) = I(P; i\zeta)^{-1}\lrpar{ \cM(P u - Su) (\zeta) } 
\end{equation}
is holomorphic in the region $\{ \eta < \gamma \}$ with values in $L^2( \bbB^h \times Z; E)$,
even though the right hand side is {\em a priori} a meromorphic function of $\zeta$ (because of the meromorphy
of the inverse of the indicial operator). 

Notice that $P u \in x^{\gamma+1}L^2(\cU_q,\frac{dx}xdy\dvol_Z)$, and hence 
$\cM(P u)$ is a holomorphic function in  the half-plane $\{ \eta < \gamma+1 \}$ 
with values in $L^2(\bbB^h \times Z; E)$.  Reasoning as in \cite{Mazzeo:Edge}, and recalling that
$u \in \mathrm{Image}(Q)$, we obtain that for every $s \in [0,1]$, $\cM(Su)$ is holomorphic in $\{ \eta < \gamma+s \}$
with values in $L^2(Z_q; H^{-s}(\bbB^h) \otimes E)$.

Altogether then, \eqref{eq:HoloEquality} implies that $\cM u$ extends to a meromorphic function
from $\{ \eta < \gamma+s \}$ to $L^2(Z_q; H^{-s}(\bbB^h) \otimes E)$ for any $s \in [0,1].$ 
Moreover, the poles all come from the inverse of the indicial operator and hence the leading term at each
pole is the image of $u$ under the projector appearing as the leading coefficient in the Laurent
expansion of $I(P; i\zeta)^{-1}$ at that pole. 
Finally $\cM(u)$ is, for any $\eps>0,$ in 
\begin{equation*}
	L^2( \{ \eta = \gamma+1-\eps \}, d\xi; H^{-1}(\bbB^h; L^2(Z; E) ) )
\end{equation*}
and taking inverse Mellin transform along $\{ \eta =\gamma+1-\eps \}$ yields the element $\wt u.$
\end{proof}

If $u \in \cD_{\max, \bB}(L)$ and $\chi \in A(Y^{k+1})$ then we can apply Lemma \ref{lem:ExpansionExistence} to $\chi u.$
The coefficients in the expansion of $\chi u$ are sections of the trace bundle over $Y^{k+1}$ and, since $\chi \in \CI_{\Phi}(\wt X),$ it only enters multiplicatively.
Thus associated to $u$ is a distributional section of the trace bundle defined over all of $\wt Y^{k+1}.$
We will refer to this section as the Cauchy data of $u$ and denote it $\cC_{Y^{k+1}} u$ or simply $\cC u$ when $Y^{k+1}$ is clear from context.

An immediate consequence of the continuity of the coefficients in the partial asymptotic expansion is that local ideal boundary conditions determine closed domains.
\begin{theorem}\label{Thm:ClosedDomains}
Let $\bB = (B^1, \ldots, B^{k})$ be local ideal boundary conditions for $L$. Then 
\begin{equation*}
	( L, \cD_{\max, \bB}(L) )
\end{equation*}
is a closed operator on $L^2(X;E).$
\end{theorem}
\begin{proof}
Suppose that $u_n$ is a sequence in $\cD_{\max, \bB}(L)$ which is Cauchy in the graph norm. Since $(L, \cD_{\max}(L))$ is 
closed, the limit $u = \lim u_n$ lies in $\cD_{\max}(L)$ and $Lu_n \to Lu$ in $L^2$ so it remains to show that $u$
satisfies the ideal boundary conditions, i.e.\ that $u \in \cD_{\max, \bB}(L).$ 

We may assume that the statement is proven at all strata $Y^j$, $j < k$. It follows that $u$ has an expansion
at $Y^{k}$. Since the coefficients $u_{\zeta_j}$ depend continuously on $u$, they are the limits of the 
corresponding coefficients of $u_n$, and hence belong to $\ker B^{k}$. This proves that $u \in \cD_{\max, \bB}(L)$.
\end{proof}

The regularity of $\cC u$ in Lemma \ref{lem:ExpansionExistence} can not be improved in general, see \cite[Example 7.8]{Mazzeo:Edge}.
If $u$ is required to be `smooth in $y$' then the Cauchy data of $u$ will inherit that regularity.
Indeed, note that if $u$ is supported in $\cU_q$ and has extra Sobolev regularity in directions tangent to $\bbB^{h_j} \times Z$ then its Mellin transform will be a meromorphic function into this Sobolev space and the coefficients will have this much extra Sobolev regularity in $\bbB^{h_j}$ (cf. the proof of \cite[Theorem 7.14]{Mazzeo:Edge}).
Thus let us define
\begin{multline}\label{eq:DefRegDom}
	\cD_{\max, \bB}^{\reg}(L) = \{ u \in \cD_{\max,\bB}(L): \Mforall \chi\in A(Y^{k+1}), N \in \bbN, \Mand (V_j) \subseteq \cV_b(\wt X), \\
	V_1\ldots V_N (\chi u) \in L^2(X;E) \}.
\end{multline}
(Here $\cV_b(\wt X)$ refers to vector fields tangent to all of the boundary hypersurfaces of $\wt X.$)
As we have just explained, the proof of Lemma \ref{lem:ExpansionExistence} shows that if $u \in \cD_{\max, \bB}^{\reg}(L)$ then $\cC u$ is a smooth section of the trace bundle and that, for any $\chi \in A(Y^{k+1}),$ the error term in the expansion of $(\chi u)$ is an element of $x^{1-}L^2.$ 
In particular this shows that $\cD_{\max, \bB}^{\reg}(L)\subseteq \cD_{\max, \bB}(L) \cap x^a L^2(X;E)$ for $a$ small.
In section \ref{sec:CoreDomain} below we will give a sufficient condition for $\cD_{\max, \bB}^{\reg}(L)$ to be a core domain for $\cD_{\max, \bB}(L).$

\subsection{Inversion of the normal operator} \label{sec:InvNormal}
In this section we describe a structure theorem for the generalized inverse of the 
normal operator of a first order $\iie$ differential operator.

We shall use some extensions and adaptations to the present setting of the calculus of pseudodifferential edge operators, 
$\Psi_e^*$, from \cite{Mazzeo:Edge} and already mentioned in \S 1.1 in the setting of depth $1$ spaces.
We refer to that paper for details on this calculus, and shall follow its language and notation freely. \\ 

Let $L \in \Diff^{1}_{\iie}(X;E)$ be elliptic, with local ideal boundary conditions 
\begin{equation*}
	\bB = (B^1, \ldots, B^k)
\end{equation*}
at the first $k$-strata of $\hat X$. We work near a point $q \in Y^{k+1}$ with distinguished neighborhood $\cU_q$, and denote by
$Z$ the link of $\hat X$ at $q$ and fix a boundary defining function $x$ for $Y^{k+1}$; we also write $\dim Y^{k+1} = h$ 
and $\dim Z = f$.  Briefly, the idea is to analyze $N_q(L)$ by treating it as an (incomplete) edge operator, making allowances
for the fact that the fibre $Z$ is stratified of depth $k$. By virtue of the inductive structure of the argument, the operator induced 
by $L$ on $Z$, together with the boundary conditions above, is Fredholm and hence can be treated essentially
as when $Z$ is closed and nonsingular.  Beyond Assumptions \ref{Ass:Localizable},  \ref{Ass:CstIndicialRoots} and \ref{AssIndFamily},
this requires the additional 
\begin{assume} [Compact domain] \label{AssCmpDomain}
The inverse $(I_q(P;\zeta), \cD_{\max, I_q(\bB;\zeta)}(I_q(P;\zeta)) )^{-1},$ when it exists, is a compact operator on $L^2(Z_q;\pi_q^*E).$
\end{assume}
We denote the space of compact operators on $L^2(Z_q;\pi_q^*E)$ by $\cK (L^2(Z_q;\pi_q^*E)).$
This assumption is verified for the de Rham operator in Theorem \ref{Thm:MainHodgeThm}.\\ 
 
As we have shown, the boundary conditions $\bB$ induce a domain $\cD_{\max, N(\bB)}(N_q(P))$ for $N_q(P)$ on 
$L^2(\bbR^+_s \times T_qY^{k+1} \times Z_q; \pi_q^*\Lambda^*(\Iie T^*X)).$ Taking the Fourier transform in 
$T_q Y^{k+1}$, writing the dual variable as  $\eta$, and rescaling by $t = s|\eta|$ (which reflects the dilation 
invariance of $P$ in $T_q^*Y^{k+1} \times \bbR^+$) we can transform $N_q(P)$ to a family of `Bessel-type' operators 
parametrized by points in the cosphere bundle over $Y^{k+1},$
\begin{equation*}
	\bbS^*_qY^{k+1} \ni
	\hat\eta \mapsto  
	\Theta_q(P)(\hat\eta) = 
	\sum_{j + |\alpha|+|\beta| \leq 1} a_{j,\alpha, \beta}(0,q,z) (t\pa_t)^j (t\hat\eta)^{\alpha} (V_{z}\rest{x=0, y=q})^{\beta}.
\end{equation*}
Here $\hat\eta = \eta/|\eta|$ is the variable in the cosphere bundle.  Each $\Theta_q(P)$ can be regarded near $t=0$ as 
an elliptic $b$-operator on $\bbR^+ \times Z$ depending smoothly on $(q, \hat\eta) \in \bbS^*Y^{k+1}$, and 
inherits the domain $\cD_{\max, \Theta(\bB)}(\Theta_q(P))$ by passing to the Bessel reductions of the boundary operators, 
$\Theta_q(B^1), \ldots, \Theta_q(B^k)$. However, as already discussed, there is an important proviso: in the standard $b$-theory, 
the link $Z$ is assumed to be a smooth compact manifold, while here it is a stratified space of depth $k$. 
By the assumptions above, and by induction, the indicial family, which is a family of $\iie$ operators on $Z$ with 
ideal boundary conditions, has a discrete set of poles and a good regularity theory for elements in its nullspace, 
and these two facts suffice to proceed with the main constructions of the $b$-calculus.  However, beyond
this behavior near $t=0$, the family $\Theta_q(P)$ has a `Bessel structure' as $t \to \infty$, and this plays an
important role too.

We now analyze $( \Theta_q(P), \Theta_q(\bB))$ following \cite[\S5]{Mazzeo:Edge} closely, cf.\ also \cite[\S 5.6]{ALMP}. 
In the following, use boundary defining functions $\rho_0 = t/(1+t)$ for $t=0$ and $\rho_{\infty} = 1/(1+t)$ for $1/t=0$ . 
\begin{lemma}
Suppose that no indicial root of $\Theta_q(P)$ has real part equal to $\delta+\tfrac {f+1}2.$ Then,  for any $\gamma \in \bbR$,  
\begin{equation*}
	\Theta_q(P)(\hat \eta): 
	\cD_{\max, \Theta(\bB)}(\Theta_q(P)) \cap \rho_0^\delta \rho_{\infty}^\gamma L^2(\bbR^+ \times Z_q;E)
	 \lra 
	\rho_0^\delta \rho_{\infty}^{\gamma+1} L^2(\bbR^+ \times Z_q;E)
\end{equation*}
is Fredholm. 
\label{lemfred}
\end{lemma}
\begin{proof}
It suffices to construct a right parametrix for $\Theta_q(P)$ since a left parametrix is obtained as the dual
of the right parametrix for $\Theta_q(P)^t$.  To do this, we construct a parametrix near $t=0$ and
another $t=\infty$ and then patch these together. 

Near $t=0,$ and for any $\hat\eta \in \bbS^*_qY^{k+1}$, let 
\begin{equation*}
	Q_0 = I(\Theta_q(P)(\hat\eta))^{-1} = \left(\sum_{j +|\beta| \leq 1} a_{j,0, \beta}(0,q,z) (t\pa_t)^j  (V_{z}\rest{x=0, y=q})^{\beta}\right)^{-1}.
\end{equation*}
(equivalently, $I_q(P)^{-1}$) on $t^\delta \cD_{\max, I(\bB)}(I_q(P)).$  This is obtained by restricting the inverse of the indicial family
to the line $\eta = \delta$ and taking the inverse Mellin transform by integrating along this line. 
That this is a parametrix for $\Theta_q(P)(\hat\eta)$ near $t=0$ follows easily from Assumption \ref{AssCmpDomain}. Indeed, 
\begin{equation*}
	\Theta_q(P)(\hat\eta) I_q(P)^{-1} - \Id = (\Theta_q(P)(\hat\eta)-I_q(P))(I_q(P)^{-1}) = \sum_{|\alpha|= 1} a_{0,\alpha, 0}(0,q,z)  (t\hat\eta)^{\alpha}I_q(P)^{-1},
\end{equation*}
and this is a compact operator on functions supported near $t=0$. 

As for its behavior when $t$ is large, conjugate $\Theta_q(P)(\hat\eta)$ by the Fourier transform in $t$. Letting $\tau$ be the variable 
dual to $t,$ we have
\begin{equation*}
	\Theta_q(P)(\hat\eta)u = (2\pi)^{-1}\int e^{it\tau}\wt\sigma(\Theta_q(P)(\hat\eta))(t, \tau) \hat u(\tau,z) \; d\tau,
\end{equation*}
where
\[
\wt \sigma(\Theta_q(P)(\hat\eta))(t,\tau) = 
i(a_{1,0,0}(t\tau) + \sum_{|\alpha|= 1} a_{0,\alpha, 0}(0,q,z)  (t\hat\eta)^{\alpha}) + \sum_{|\beta|=1} a_{0,0,\beta}(V_{z}\rest{x=0, y=q})^{\beta}
\]
is the `partial principal symbol' of $\Theta_q(P).$ As explained in \cite[Lemma 5.5]{Mazzeo:Edge}, the ellipticity of $P$ 
implies that 
\begin{equation*}
	Q_{\infty}(\hat\eta)u = (1-\phi(t)) \int e^{it\tau} \wt\sigma(\Theta_q(P)(\hat\eta))^{-1} \hat u (\tau,z) \; d\tau,
\end{equation*}
where $\phi$ is a cut-off function, induces a bounded map
\begin{equation}\label{eq:spaces}
	\rho_0^\delta \rho_{\infty}^{\gamma+1} L^2(\bbR^+ \times Z_q;E) \lra
	\rho_0^\delta \rho_{\infty}^\gamma  L^2(\bbR^+ \times Z_q;E)\cap \cD_{\max, \Theta(\bB)}(\Theta_q(P)(\hat\eta)).
\end{equation}
Patching together $Q_0$ and $Q_\infty$, we obtain an operator $Q,$ which is bounded as in \eqref{eq:spaces}, 
and such that $\Theta_q(P)Q - \Id$ is compact. Dualizing a right parametrix for $L^t$ gives a left parametrix.
\end{proof}

The reader may have noticed that the specific value of the weight parameter $\gamma$ at $t = \infty$ plays no real
role, and indeed, it follows from this construction that elements of the nullspace of $\Theta_q(P)$ which are tempered 
as $t \to \infty$ necessarily decrease rapidly: 
\begin{equation*}
\ker \Theta_q(P) \cap \lrpar{\bigcup_{\gamma \in \bbR} \rho_0^\delta \rho_{\infty}^\gamma L^2(\bbR^+ \times Z_q;E)}
\subseteq \bigcap_{\gamma \in \bbR^+} \rho_0^\delta \rho_{\infty}^\gamma L^2(\bbR^+ \times Z_q;E)
\end{equation*}

\medskip

Let us now describe the finer structure of $Q_0$, the portion of the parametrix acting on 
functions supported near $t=0$.  

The indicial family $I_q(P;\zeta),$ takes values in $\Diff_{\iie}^*(Z_q;E)$ and by assumption its inverse is 
a meromorphic function with simple poles valued in $\sK(L^2(Z_q;E)),$ the compact operators on $L^2(Z_q;E).$ 
Taking the inverse Mellin transform, the resulting inverse of the indicial operator is an element 
\begin{equation*}
	Q_0 \in \Psi^{-1,\cE}_b(\bbR^+) \otimes \sK(L^2(Z_q;E)) 
\end{equation*}
and satisfies
\begin{equation*}
	\Theta_q(P)Q_0 = \Id - R_0 \Mwith R_0 \in t\Psi_b^{-1,\cE}(\bbR^+) \otimes \sK(L^2(Z_q;E)).
\end{equation*}
Let us recall here that $\Psi_b^{-1,\cE}$ is the space of $b$-pseudodifferential operators $A$ of order $-1$ and
with index family $\cE = (E_{10}, E_{01}, E_{11})$ describing the exponents in the asymptotic expansions of the Schwarz 
kernel of $A$  at the various boundary components of the $b$-stretched product $(\RR^+)^2_b$. 
The index sets $E_{ij}$ are determined by the indicial roots of $P$ and the weight $\delta,$ and are given by
\begin{equation*}
\begin{gathered}
	E_{10} = \{ \zeta : \zeta \in \spec_b(P;\bB,q), \Re \zeta > \delta - \tfrac{f+1}2 \}, \\
	E_{01} = \{ -\zeta : \zeta \in \spec_b(P;\bB,q), \Re \zeta < \delta - \tfrac{f+1}2 \}, \\
	E_{11} = \bbN_0.
\end{gathered}
\end{equation*}
The error term $R_0$ contains an extra factor of $t$, which means that it vanishes to first order on the front face 
and also at the left face. (This latter vanishing is achieved by using a careful extension off of the front face, 
as in \cite[Proof of Proposition 5.43]{APSBook}).

Since $Q_0$ is obtained by integrating the inverse of the indicial family along a vertical line in the complex plane,
it is not surprising that the indicial roots determine the index sets. Since this inverse is meromorphic, changing 
the contour of integration to a parallel vertical line would produce another operator which differs from the first
by the (finite) sum of residues at the poles between these two lines. These residues are the coefficients of the expansion 
of $\cK_{Q_0}$ at the left and right faces, see \cite[Lemma 5.16]{APSBook}. In particular, if $\zeta_0$ is a simple pole 
of $I_q(P,\zeta)^{-1}$ with $\Re \zeta_0 >\delta,$ then the coefficient of $s^{\zeta_0}$ of $\cK_{Q_0}$ at the left face is 
the orthogonal projection onto $\ker I_q(P;\zeta_0).$ (Of course if $\Re\zeta_0 < \delta,$ this is the coefficient of 
$s^{-\zeta_0}$ of the expansion at the right face.) This is consistent with the fact that $Q_0$ maps into 
$\cD_{\max, \Theta(\bB)}(\Theta_q(P))$ and elements in this domain have partial asymptotic expansions with coefficients in the null 
space of the indicial family at the corresponding indicial root.\\

Before proceeding, let us note that tensoring with $\sK(L^2(Z_q;E))$ does not interact with the constructions of the $b$-calculus.
We now improve $Q_0$ to a finer parametrix, i.e.\ one with an even smaller remainder term. The first step in this is to solve away 
the expansion of the error $R_0$ at the left face. This is done as in \cite[Lemma 5.44]{APSBook}, using $I_q(P,\zeta)^{-1}$ again,
and in particular staying within the $b$-calculus twisted by $\sK(L^2(Z^2_q;E)).$ This leads to the next parametrix
\[
Q_1 \in \Psi^{-1,\cF}_b(\bbR^+) \otimes \sK(L^2(Z_q;E)),
\]
which satisfies
\[
\Theta_q(P)Q_1 = \Id - R_1 \Mwith R_1 \in t\Psi_b^{-1,(\infty,F_{01},0)}(\bbR^+) \otimes \sK(L^2(Z_q;E)),
\]
for some index set $\cF= (F_{10}, F_{01}, 0).$

The second step is to form the asymptotic sum of the Neumann series $\sum R_1^j.$ This is possible because 
the kernel of $R_1$ vanishes to order one at the front face and to infinite order at the left face. Writing this asymptotic
sum as $\Id + S$, then 
\[
S \in t\Psi_b^{-1,(\infty,G_{10},0)}(\bbR^+) \otimes \sK(L^2(Z_q;E))
\]
and $(\Id + S) - (\Id - R_1)^{-1}$ is the residual term
\[
R_2 = \Id - (\Id - R_1)(\Id + S) \in \Psi_b^{-\infty, (\infty, \Gamma, \infty)}(\bbR^+) \otimes \sK(L^2(Z_q;E)),
\]
for some index set $\Gamma$.  By construction, $\Theta_q(P)Q_2 = \Id - R_2$, where
\begin{equation*}
	Q_2 = Q_1 (\Id + S) \in \Psi_b^{-1,(E_{10}, J_{01}, J_{11})}(\bbR^+) \otimes \sK(L^2(Z_q;E))
\end{equation*}
for index sets satisfying $J_{01} \geq E_{01}$ and $J_{11} \geq 0.$

By the same sort of duality argument, we can also construct a parametrix $Q_2'$ on the left for $\Theta_q(P).$

Having constructed these two parametrices, then Lemma~\ref{lemfred} implies abstractly that $\Theta_q(P)$ has a 
generalized inverse $G_{\Theta_q(P)}$. A further argument, \cite[(4.25)]{Mazzeo:Edge}, gives the identity
\begin{equation*}
	G_{\Theta_q(P)} = Q_2 + R_2 G_{\Theta_q(P)} R_2' + R_2 Q_2' - R_2 \Pi_{\ker \Theta_q(P)} Q_2' - Q_2 \Pi_{\coker \Theta_q(P)},
\end{equation*}
which allows us to understand this generalized inverse as an element in the twisted $b$-calculus. 
(Note that for us the cokernel will always be identified with the orthogonal complement of the image.) In particular, 
if there is no indicial root in $\spec_b(P;\bB,q)$ with real part $\delta-\tfrac {f+1}2$, (to understand this shift, 
recall that $x^{\delta}L^2(x^f \; dx) = x^{\delta-(f+1)/2}L^2(\tfrac{dx}x)$), then we obtain finally that
\begin{equation*}
\begin{gathered}
	G_{\Theta_q(P)} 
	\in \Psi_b^{-1,\cH}(\bbR^+) \otimes \sK(L^2(Z_q;E)), \\
	\Pi_{\ker \Theta_q(P)}
	\in \Psi_b^{-\infty,\cE}(\bbR^+) \otimes \sK(L^2(Z_q;E)), \\
	\Pi_{\coker \Theta_q(P)}
	\in \Psi_b^{-\infty,\cF}(\bbR^+) \otimes \sK(L^2(Z_q;E)).
\end{gathered}
\end{equation*}
The index sets here are \cite[(4.22)]{Mazzeo:Edge}
\begin{equation}\label{eq:GralInvIndSets}
\begin{gathered}
	E_{10} = \{ \zeta \in \spec_b(P;\bB,q) : \Re \zeta > \delta - \tfrac{f+1}2 \}, \quad
	E_{01} = E_{10} - 2(\delta - \tfrac{f+1}2) \\
	F_{01} = \{ -\zeta : \zeta \in \spec_b(P;\bB,q), \quad \Re\zeta < \delta -\tfrac{f+1}2 \}, \quad
	F_{10} = F_{01} + 2(\delta - \tfrac{f+1}2) \\
	H_{10} = E_{10} \bar\cup F_{10}, \quad
	H_{01} = E_{01} \bar\cup F_{01}, \quad
	E_{11}=F_{11} = \infty, \quad H_{11} = \bbN_0.
\end{gathered}
\end{equation}

So far we have been working at a fixed $\hat\eta \in \bbS^*_qY^{k+1}.$ However, we can now apply the analysis 
in \cite[(5.11-19)]{Mazzeo:Edge} verbatim.
The generalized inverses of the Bessel operators $\Theta_q(P)(\hat\eta)$ can be reassembled into the generalized inverse of $N_q(P)$:  
\begin{equation}\label{eq:NfromB}
	\cK_{G_{N_q(P)}}(s, s', u, u', z, z')
	= \int e^{i(u-u')\cdot \eta} \cK_{G_{\Theta_q(P)}}(s|\eta|, s'|\eta|, z, z', \hat \eta)|\eta| \; d\eta.
\end{equation}
The change from the parameter $\eta$ to the variable $u$ means that we should interpret this not as a $b$-operator on $\bbR^+,$ but as an element of the $0$-calculus of \cite{Mazzeo:Hodge} on $\bbR^+ \times \bbR^h;$ recall that while the $b$-calculus $\Psi^*_b$ is a `microlocalization' of vector fields tangent to boundaries, the $0$-calculus $\Psi^*_0$ is a microlocalization of vector fields that vanish on boundaries.
Indeed, some analysis of formula \eqref{eq:NfromB} leads to the 
\begin{proposition}\label{prop:InvNormalOp}
If there is no element of $\spec_b(P;\bB,q)$ with real part equal to $\delta-\tfrac {f+1}2,$ and if $(N_q(P), \cD_{\max, N(\bB)}(N_q(P)))$ 
is either injective or surjective on $s^\delta L^2$, then there is an operator
\begin{equation}\label{eq:GNqP}
	G_{N_q(P)} \in \Psi_0^{-1,\cH}(\bbR^+ \times \bbR^{h}) \otimes \sK(L^2(Z_q;E)), 
\end{equation}
with image in $\cD_{\max, N_q(B^1)}(N_q(P))$ and such that 
\begin{equation*}
\begin{gathered}
	\Id - G_{N_q(P)}N_q(P) = \Pi_{\ker N_q(P)}
	\in \Psi_0^{-\infty,\cE}(\bbR^+\times \bbR^{h}) \otimes \sK(L^2(Z_q;E)), \\
	\Id - N_q(P)G_{N_q(P)} = \Pi_{\coker N_q(P)}
	\in \Psi_0^{-\infty,\cF}(\bbR^+\times \bbR^{h}) \otimes \sK(L^2(Z_q;E))
\end{gathered}
\end{equation*}
are the orthogonal projections onto the kernel and cokernel of $\cD_{\max, N(\bB)}(N_q(P))$. 
In particular, for any such $\delta,$ the operator $N_q(P)$ (which implicitly depends on $\delta$) has 
closed range. The integral kernels of these operators depend smoothly on $q \in Y^{k+1}.$
\end{proposition}

An advantage of this explicit description for this integral kernel is that we can read off the mapping properties 
on weighted $L^2$ spaces. To state this, define the `real part' of an index set, e.g.\ $E_{10}$, by
\begin{equation*}
	\Re E_{10} = \inf \{ \Re \zeta : (\zeta, p) \in E_{10} \}.
\end{equation*}
From \cite[Theorem 3.25]{Mazzeo:Edge}, if 
\begin{equation*}
	\cK_A \in \Psi_0^{-\infty, \cG}(\bbR^+ \times \bbR^h) \otimes \sK(L^2(Z_q;E)), \qquad \cG = (G_{10}, G_{01}, G_{11}),
\end{equation*}
then $A$ extends from $\CIc(\bbR^+ \times \bbR^h \times Z_q;E)$ to a bounded map
\[
A: x^a L^2(\tfrac{dx}xdy; L^2(Z_q;E)) \lra x^b L^2(\tfrac{dx}xdy; L^2(Z_q;E)) 
\]
if and only if 
\[
\Re G_{10} >b, \quad \Re G_{01} > -a, \quad	\Re G_{11} \geq b-a.
\]
Equivalently, in terms of the volume form of an $\iie$-metric, 
\begin{equation}\label{eq:MappingProperty}
\begin{split}
A: x^{a'} L^2(& x^f \; dxdy; L^2(Z_q;E)) \lra x^{b'} L^2(x^f \; dxdy; L^2(Z_q;E)) \\[0.3ex]
& \text{ if and only if } \\[0.3ex]
\Re G_{10} & >b'-\tfrac{f+1}2, \quad \Re G_{01} > \tfrac{f+1}2-a', \quad \Re G_{11} \geq b'-a',
\end{split}
\end{equation}
which again uses $L^2(x^f dx) = x^{-(f+1)/2}L^2(\tfrac{dx}x).$\\

Since it will be of particular importance below, let us spell out these mapping properties more carefully
when the index sets $\cE,$ $\cF,$ and $\cH$ are as in \eqref{eq:GralInvIndSets}. Define
\begin{equation}\label{eq:DefEta}
\begin{gathered}
	\eta^+ = -(\delta - \tfrac{f+1}2)  + \inf \{ \Re\zeta : \zeta \in \spec_b(P;\bB,q), \ \Re \zeta > \delta - \tfrac{f+1}2 \} \\[0.5ex]
	\eta^- = (\delta - \tfrac{f+1}2) - \sup \{ \Re\zeta :\zeta \in \spec_b(P;\bB,q), \ \Re \zeta < \delta - \tfrac{f+1}2 \}.
\end{gathered}
\end{equation}
Then $A \in \Psi_0^{-\infty, \cG}(\bbR^+ \times \bbR^h) \otimes \sK(L^2(Z_q;E))$ defines a bounded map 
\[
x^{a'} L^2(x^f \; dxdy; L^2(Z_q;E)) \lra x^{b'} L^2(x^f \; dxdy; L^2(Z_q;E))
\]
provided 
\begin{equation*}
	\begin{cases}
	\delta - \eta^+ < a'\leq b' < \eta^+ + \delta & \Mif \cG = \cE\\
	\delta - \eta^- < a'\leq b' < \eta^- + \delta & \Mif \cG = \cF\\
	\delta - \min(\eta^+,\eta^-) < a' = b' < \min(\eta^+,\eta^-) + \delta & \Mif \cG = \cH
	\end{cases}
\end{equation*}

\medskip

To conclude this section, let us say a bit more about the meaning of ``reassembling'' the integral kernels of the $G_{N_q}$ 
as $q \in Y$ varies. Let $\cU_q \cong [0,1)_x \times \bbB^h \times Z_q$ be a distinguished neighborhood of $q \in Y^{k+1},$ 
where $\bbB^h$ is the unit ($h$-dimensional) ball in $T_qY^{k+1}$, and fix a trivialization 
\begin{equation*}
	E\rest{\cU_q} \cong [0,1)_x \times \bbB^h \times E\rest{Z_q}.
\end{equation*}

We define the {\bf zero double space} of $\cU_q$ by radially blowing up the fiber diagonal of the boundary: 
\begin{equation}\label{eq:ZeroDouble}
	(\cU_q)^2_0 = ([0,1)_x \times \bbB^h)^2_0 \times Z_q^2
	= \lrspar{ [0,1)_x^2 \times (\bbB^h)^2; (0,0) \times \diag_{\bbB} } \times Z_q^2.
\end{equation}
This space has a natural blow-down map
\begin{equation*}
	\beta: (\cU_q)^2_0 \lra \cU_q^2
\end{equation*}
and we denote by $\beta_L,$ $\beta_R$ the compositions of this blow-down with the projections onto the left or right 
factor of $\cU_q.$ We refer to $\beta^{-1}_L(\{ x=0\})$, $\beta_R^{-1}(\{ x= 0\})$ and 
$\beta^{-1}( (0,0) \times \diag_{\bbB} \times Z_q^2)$ as the {\bf left}, {\bf right} and ${\bf front}$ boundary faces,
denoted $\cB_{10}$, $\cB_{01}$ and $\cB_{11}$, respectively. 

Choose coordinates $x,y, x', y', z, z'$ on $([0,1)_x \times \bbB^h)^2_0 \times Z_q^2$. Then, we define projective 
coordinates on $(\cU_q)^2_0$ away from the right face $x'=0$ by  
\begin{equation*}
	s = \frac{x}{x'}, \ \ 	u = \frac{y-y'}{x'}, \ \ x', \ \ y', \ \ 	z, \ \ z'.
\end{equation*}
In these coordinates, $\{s=0\}$ defines the left face while $x'$ is a defining function for the front face. 
We can use these coordinates to express the lifts from the left factor of functions and edge vector fields: 
\begin{equation*}
\begin{gathered}
	f(x,y,z) \mapsto \beta_L^*f(s,u,x',y',z,z') = f(sx', y'+x'u, z), \\[0.5ex]
	\beta_L^*(x\pa_x) = s\pa_s, \quad 	\beta_L^*(x\pa_y) = s\pa_u, \quad 	\beta_L^*(\pa_z) = \pa_z.
\end{gathered}
\end{equation*}
The front face is naturally the total space of a fibration with base $\diag_\bbB$ and fiber $Z_q^2$ times a quarter-sphere;
in these coordinates, this is the map $(s,u,y',z,z') \mapsto y'$ (with $x'=0$ since it is at the front face). 
Using $\beta_L$, we can extend this fibration to a neighborhood of the front face by
\begin{equation*}
	(s,u,x',y',z,z') \mapsto y'+x'u.
\end{equation*}

Now suppose that $\cK_{G_{y'}}(s,u,z,z')$ is the integral kernel of $G_{N_{y'}(P)}$ from \eqref{eq:GNqP}, which is the 
generalized inverse of $N_{y'}(P)$. We say that 
\begin{equation}\label{eq:Amalgamation}
	\cK_{G}(s, u, x',y',z,z') = \chi \cK_{G_{y'+x'u}}(s, u, z, z')
\end{equation}
is the {\bf amalgamation} of these integral kernels; here $\chi$ a smooth cut-off function on 
$\cU^2_q$ which equals one near the front face. By construction, $\cK_G$ is compatible with the 
extended fibration and
\begin{equation*}
	\cK_G\rest{x'=0, y'=p} = \cK_{G_p}.
\end{equation*}
Moreover, the conormal singularity in each $\cK_{G_{y'}}$ at $(s,u) = (1,0)$ gives rise to a conormal singularity of 
the same order for $\cK_G$ at $(s,u) =(1, 0),$ which is the (interior) lift of the diagonal in $[0,1)\times \bbB.$

Just as for the $b$-calculus, see \cite[Proof of Proposition 5.43]{APSBook}, the advantage of this extension is seen in the
behavior of the index set at the left face when we compose with $P.$  The restriction of the lift of $P$ to the fibres composed 
with the restriction of $G$ there is a bounded operator, and we have that $PG \in \Psi_0^{0,\cH}(\cU_q) 
\otimes \sB(L^2(Z_q;E))$. Consider just the expansion at the left face: since $G$ has index set $H_{10}$ there, 
we would expect $N_q(PG)$ to have index set $H_{10},$ but instead, because of the action of the indicial operator 
of $P$, it actually has index set $F_{10}$ at that face.  Essentially by definition, $I_q(P)$ acts on 
the leading terms of an asymptotic expansion (see especially \cite[(A.9)]{Mazzeo:Edge}). However, in the expansion of $N_q(G)$ 
at the left face, all coefficients of the leading terms in $H_{01}\setminus F_{01}$ are in the null space of $I_q(P)$, and hence
are annihilated. 

If, in local coordinates, 
\begin{equation*}
	P = \sum f_{j,\alpha,\beta}(x,y,z) (x\pa_x)^j(x\pa_y)^{\alpha}(\pa_z)^{\beta},
\end{equation*}
then in the projective coordinates on $(\cU_q)^2_0,$
\begin{equation*}
	\cK_{PG}(s,u,x',y',z,z') = 
	\lrspar{
	\sum f_{j,\alpha,\beta}(x's, y'+x'u,z) (s\pa_s)^j(s\pa_u)^{\alpha}(\pa_z)^{\beta} }	
	\cK_{G_{y'+x'u}}(s, u, z, z').
\end{equation*}
Clearly, the expansion of $\cK_{PG}$ at the left face is given by $I_{y'+x'u}(P)$ applied to the terms in the expansion of 
$\cK_{G_{y'+x'u}}$ at that face. In particular, we have been careful to match the base points of the indicial family with 
the extension of the Schwartz kernel of $G,$ so we do indeed obtain the improvement of the index set. 
This means that we can omit the leading terms in $H_{10}\setminus F_{10}$ in the index set of $PG$ at the left face, and 
hence take 
\begin{equation}\label{eq:ImprovedIndex}
	\wt H_{10} = 
	\lrpar{ F_{10} \cap \{ \Re \zeta <1+ \Re (H_{10} \setminus F_{10})\} }
	 \bigcup
	 \lrpar{ H_{10} \cap \{ \Re \zeta \geq 1+ \Re (H_{10} \setminus F_{10})\} }
\end{equation}
for the index set of $PG$ there. 

Finally, note that the amalgamated integral kernels form a $\sK(L^2(Z_q;E))$-twisted $0$-pseudodifferential operator.
The indicial roots may vary with the base point $y$, and we refer the reader to \cite{Krainer-Mendoza:Kernel} for
a detailed treatment of this, see also  \cite{Borthwick:ScatVarCurv}. For our purposes, it is enough to have 
uniform bounds on the index sets. For example, if $\wt A$ is obtained from amalgamating operators 
with index set $\cE$ as above, and if we assume that $\eta$ from \eqref{eq:DefEta} is uniformly bounded 
below by $\underline{\eta},$ then $A$ defines a continuous map
\begin{equation*}
	\wt A: x^{a'} L^2(\cU_q; E) \lra	x^{b'} L^2(\cU_q; E), \quad
	\Mwhenever \delta - \underline\eta < a'\leq b' < \underline\eta + \delta.
\end{equation*}
%

\subsection{A core domain} \label{sec:CoreDomain}
In this section we establish that the regular domain defined in \eqref{eq:DefRegDom} is a core domain.
In the process we will explain various properties of adjoint domains.

Let $L \in \Diff^{1}_{\iie}(X;E)$ be elliptic, with local ideal boundary conditions 
\begin{equation*}
	\bB = (B^1, \ldots, B^k)
\end{equation*}
at the first $k$-strata of $\hat X$. 
Recall that $\cS\cD(L)$ consists of those indicial roots $\zeta_j$ such that $x^{\zeta_j} \in L^2_{\loc}\setminus xL^2_{\loc}.$
Define
\begin{equation*}
	\min \cS\cD(L) = \min \{\Re(\zeta_j) : \zeta_j \in \cS\cD(L) \}, \quad
	\max \cS\cD(L) = \max \{\Re(\zeta_j) : \zeta_j \in \cS\cD(L) \}.
\end{equation*}
We will need the following assumption:
\begin{assume}[Injective/surjective normal operator]\label{Ass:InvNormalOp}
For each $q \in Y^{k+1},$ the normal operator 
$N_q(P)$ with domain $\cD_{\max, N(\bB)}(P)\cap s^{\delta}L^2(\bbR^+ \times T_qY \times Z_q;E)$
is injective if $\delta > \max \cS\cD$ and surjective if $\delta < \min \cS\cD.$
\end{assume}
It follows from our results in this section that this assumption is equivalent to asking that $N_q(L) $ be surjective on $\cD_{\max, N(\bB)}(N_q(L))$ and injective on $\cD_{\min, N(\bB)}(N_q(L))$ (defined below). Hence this assumption is a necessary condition for $N_q(L)$ to have an invertible extensions, and so for $L$ to have a Fredholm extension. We show in Proposition \ref{prop:NullNormal} below that this assumption holds for the de Rham operator.\\

Let us start by recalling some basic facts about the adjoints. The formal adjoint $L^*$ of $L$ is the differential operator determined by 
\begin{equation*}
	\ang{Lu,v} = \ang{u, L^*v} \Mforall u,v \in \CIc(X;E).
\end{equation*}
A simple local computation shows that the formal adjoint of an $\iie$ operator is again an $\iie$ operator, and has principal symbol
\begin{equation*}
	\sigma_{\iie}(L^*) = \sigma_{\iie}(L)^*,
\end{equation*}
where the adjoint on the right hand-side is the pointwise adjoint in $\Hom(E).$ Define the {\em boundary pairing} of $L$ by
\begin{equation*}
\begin{gathered}
	[\cdot, \cdot]_L:\cD_{\max}(L) \times \cD_{\max}(L^*) \lra \bbC \\
	[u,v]_L = \ang{Lu,v}-\ang{u,L^*v}.
\end{gathered}
\end{equation*}
Endowing these maximal domains with the graph norms for $L$ and $L^*$, respectively, this is continuous.
We refer to \cite{Gil-Mendoza:Adjoints} for a careful discussion of this pairing in the setting of isolated conic singularities. 
Note that if either $u \in \cD_{\min}(L)$ or $v \in \cD_{\min}(L^*)$, then $[u,v]_L=0$. More generally, if $\cD(L) \subseteq \cD_{\max}(L)$, 
then 
\begin{equation*}
	(L, \cD(L))^* \text{ has domain } \{ v \in \cD_{\max}(L^*) : [u,v]_L=0 \Mforall u \in \cD(L) \}.
\end{equation*}

Now let us examine more closely what this passage to the adjoint entails in our setting.  We claim that the adjoint of a domain $\cD(L)$ 
defined by ideal local boundary conditions at $Y^1, \ldots, Y^k$ is given by the adjoint ideal local boundary conditions $\bB^{\star}$ at 
these same strata.  To prove this, first recall that by Assumption~\ref{Ass:Localizable}, $\cD_{\max}(L)$ is localizable with respect to 
multiplication by functions in $\CI_{\Phi}(\wt X)$. 
Since the $\bB$ are local ideal boundary conditions, it is clear that $\cD_{\max, \bB}(L)$ is also localizable, and moreover, the identity
\begin{equation*}
	[fu,v]_L = [u,fv]_L \Mforall f \in \CI_{\Phi}(\wt X)
\end{equation*}
shows that the adjoint domain of a localizable domain is again localizable.  The assertion now follows by noting that for any
suitably localized function, its inclusion in $(\cD_{\max, \bB}(L))^*$ depends on its germ at each of the $Y^j$, and by Proposition 
\ref{prop:DminB} below, this only depends on its Cauchy data at these strata.  This can be done by induction on $k$.
We denote the adjoint domain of $\cD_{\max, \bB}(L)$ by 
\begin{equation*}
	\cD_{\min, \bB^{\star}}(L^*) = (\cD_{\max, \bB}(L))^*.
\end{equation*}
(Note that $\bB^{\star}$ denotes boundary conditions for the adjoint domain, and do not involve the adjoint operators of the individual $B_k.$) For the de Rham operator below we shall have $\bB^{\star} = \bB,$ so we shall not discuss the adjoint boundary conditions in detail. For a careful treatment in the depth one setting see \cite[Proposition 4.12]{Krainer-Mendoza3}.

In \S 2.2 extensions of various constructions from \cite{Mazzeo:Edge} have been worked out for the partially completed operator $P$ as
well as its model operators, $N_q(P),$ $\Theta_q(P),$ $I_q(P),$ with domains induced by $\bB.$ Now consider the effect of imposing a 
boundary condition $B^{k+1}$ at $Y^{k+1}.$ The model operators for $L$ itself are defined by
\begin{equation*}
	N_q(L) = \tfrac1s N_q(P), \quad \Theta_q(L) = \tfrac1t \Theta_q(P), \quad I_q(L) = \tfrac1s I_q(P);
\end{equation*}
each of these has a maximal domain with boundary conditions induced by $\bB,$ elements of which have Cauchy data at 
$s =0$ (or $t=0$). We may impose boundary conditions on each of these using $N_q(B^{k+1})$.  We consider now the special
case $B^{k+1}=\Id,$ or in other words, that all Cauchy data at $Y^{k+1}$ vanish. Later, in Proposition \ref{prop:Other}, we 
consider more general cases where $B^{k+1}\neq 0$. \\
 
Let us start by considering 
\begin{equation*}
	\bbS^*_qY^{k+1} \ni
	\hat\eta \mapsto  
	\Theta_q(L)(\hat\eta) = 
	\frac1t \sum_{j + |\alpha|+|\beta| \leq 1} a_{j,\alpha, \beta}(0,q,z) (t\pa_t)^j (t\hat\eta)^{\alpha} (V_{z}\rest{x=0, y=q})^{\beta}
\end{equation*}
as a family of unbounded operators
\begin{equation*}
	\Theta_q(L)(\hat\eta): \cD_{\max, \Theta(\bB)}(\Theta_q(L)(\hat\eta)) \subseteq
	L^2(\bbR^+ \times Z_q;E) \lra
	L^2(\bbR^+ \times Z_q;E).
\end{equation*}
For each $\hat\eta,$ this can be regarded as a conic operator on $\bbR^+$ near $t=0$ with coefficients in $\Diff_{\iie}^1(Z_q;E).$
The Bessel type behavior as $t \to \infty$ makes these operators Fredholm. Elements in $\cD_{\max, \Theta(\bB)}(\Theta_q(L)(\hat\eta))$ 
have asymptotic expansions as $t \to 0$ and we now focus on the subspace $\cD_{(\Theta(\bB),\Id)}(\Theta_q(L)(\hat\eta))$ of functions
with vanishing Cauchy data. It is known in this setting, see \cite{Gil-Mendoza:Adjoints, Gil-Krainer-Mendoza:Wedge}, that 
\begin{equation*}
	\cD_{(\Theta(\bB),\Id)}(\Theta_q(L)(\hat\eta)) = \cD_{\max,\Theta(\bB)}(\Theta_q(L)(\hat\eta)) \cap \rho_0^{1-}L^2(\bbR^+ \times Z_q;E).
\end{equation*}

The analogous statement is true for the domain of the normal operator. Indeed, taking the Fourier transform of any element 
\begin{equation*}
	u \in \cD_{(N(\bB), \Id)}(N_q(L))
\end{equation*}
and rescaling by $t = s|\eta|$ yields a family of elements in $\cD_{(\Theta(\bB),\Id)}(\Theta_q(L)(\hat\eta)) \subset \rho_0^{1-}L^2(\bbR^+ \times Z_q;E)$.
Hence 
\begin{equation}\label{eq:VanishingDN}
	\cD_{(N(\bB),\Id)}(N_q(L)) = \cD_{\max,N(\bB)}(N_q(L)) \cap s^{1-}L^2(\bbR^+ \times T_qY \times Z_q;E)
\end{equation}
as well.

By constancy and discreteness of the indicial roots, if $\delta \in (1-\eps, 1)$ for sufficiently small $\eps$, then $\delta + \tfrac{f+1}2$ 
is not an indicial root. Let $G_{N_q(P)}$ be the generalized inverse of $(N_q(P), \cD_{\max, N(\bB)}(N_q(P)))$ on $s^\delta L^2(\bbR^+ \times T_qY \times Z_q;E).$
By assumption, $N_q(P)$ is injective on $s^{\delta}L^2,$ so 
\begin{equation*}
	G_{N_q(P)}N_q(P) = G_{N_q(P)}sN_q(L) = \Id
\end{equation*}
on $\cD_{\max,N(\bB)}(N_q(P)) \cap s^\delta L^2(\bbR^+ \times T_qY \times Z_q;E)$. By  \eqref{eq:VanishingDN}, this includes $\cD_{(N(\bB),\Id)}(N_q(L))$.

Fix a distinguished neighborhood $\cU_q$ of $q \in Y^{k+1}$ and let $\bar T$ be the amalgamation of the integral kernels of $G_{N_q(P)}$ and 
write $T = \bar T x.$ Then 
\begin{equation*}
	(N_q(TL),\cD_{(N(B),\Id)}(L)) = \Id
\end{equation*}
and hence
\begin{equation*}
	TL = \Id - R \quad \Mon \cD_{\max, (\bB,\Id)}(L).
\end{equation*}
Moreover $T$ and $R$ are twisted $0$-pseudodifferential operators:
\begin{equation*}
\begin{gathered}
	T \in \Psi_0^{-1, \cH}([0,1)\times \bbB^h) \otimes \sK(L^2(Z_q;E)), \quad
	R \in x\Psi_0^{0, \cH}([0,1)\times \bbB^h) \otimes \sK(L^2(Z_q;E)) \\
	\Mwhere
	\Re(H_{10})  = \Re(E_{10}), \quad
	\Re(H_{01}) = 1+\Re (E_{01}), \quad
	\Re(H_{11}) = 1.
\end{gathered}
\end{equation*}
Using this we find that any $u \in \cD_{\max, (\bB,\Id)}(L)$ supported in $\cU_q$ satisfies
\begin{equation*}
	u = TLu + Ru \in x^{1-}L^2(X;E).
\end{equation*}
This holds near any $q\in Y^{k+1}$, so 
\begin{equation*}
	\cD_{\max,(\bB,\Id)}(L) \subseteq \cD_{\max, \bB}(L) \cap x^{1-}L^2(X;E).
\end{equation*}

\begin{lemma} $ $

\begin{enumerate}
\item 
The closure of $\cD_{\max, \bB}(L) \cap x^{0+}L^2(X;E)$ in the graph norm of $L$ is $\cD_{\max, \bB}(L).$
\item
If $u \in \cD_{\max, \bB}(L) \cap x^{1-}L^2(X;E)$ then $\chi u \in \cD_{\min, \bB}(L)$ for all $\chi \in A(Y^{k+1}).$
\end{enumerate}
\end{lemma}

\begin{proof}
(i) 
First we note that this assertion is true for the normal operator $(N_q(L),\cD_{\max, N(\bB)}(N_q(L))).$
Indeed, in the coordinates $s, u, z$ on $\bbR^+_s \times T_qY \times Z_q,$ the normal operator is a constant coefficient operator in $s$ and $u.$
Convolution in $u$ with a mollifier shows that $\cD_{\max, N(\bB)}^{\reg}(N_q(L))$ is dense in $\cD_{\max, N(\bB)}(N_q(L))$ with respect to the graph norm of $N_q(L).$
As mentioned above, $\cD_{\max, N(\bB)}^{\reg}(N_q(L)) \subseteq \cD_{\max, N(\bB)}(N_q(L))\cap s^a L^2$ for small enough $a.$

Secondly we use this to construct a pseudodifferential left generalized inverse of $N_q(L)$ on the domain $\cD_{\max, N(\bB)}(N_q(L)).$
Let $a> 0$ be small enough so that 
\begin{equation*}
	\cD_{\max, N(\bB)}^{\reg}(N_q(L)) \subseteq \cD_{\max, N(\bB)}(N_q(L))\cap s^a L^2, \quad
	\cS\cD \cap \{\zeta : 0 < \Re \zeta + \tfrac{f+1}2 \leq a\}= \emptyset.
\end{equation*}
By Assumption \ref{Ass:InvNormalOp} $N_q(P)$ is surjective on $\cD_{\max, N(\bB)}(N_q(P))\cap s^{a}L^2$ and so by the previous section its generalized inverse, $\bar T_q,$ is a twisted zero pseudodifferential operator. We denote
\begin{equation*}
	\Pi = \Id - \bar T_qN_q(P).
\end{equation*}
Since $\bar T_q$ and $\Pi$ are twisted zero pseudodifferential operators, both extend from bounded operators on $s^aL^2$ to bounded operators on $L^2.$
The equality $\Pi = \Id - \bar T_q s N_q(L)$ holds on $\cD_{\max, N(\bB)}(N_q(L))\cap s^a L^2,$ hence by density on $\cD_{\max, N(\bB)}(N_q(L)).$
Also by density $N_q(L)\Pi = 0$ on $s^a L^2$ implies $N_q(L)\Pi=0$ on $L^2.$
Altogether this shows that $\bar T_q s$ is a (possibly skew) generalized left inverse of $N_q(L)$ and a twisted zero pseudodifferential operator.

Let $\bar T$ be the amalgamation of the kernels of $\bar T_q$ and let $T = \bar T x.$
Note that $T$ maps $L^2$ into $x^{\delta}L^2$ for some $\delta>0$ (indeed any $\delta < \Re E_{10}$ where $E_{10}$ is the index set at the left face of $T$).
Now let
\begin{equation*}
	S = TL - \Id \Mon \cD_{\max, \bB}(L)
\end{equation*}
and note that $S$ is a twisted $0$-pseudodifferential operator of order minus one (on $\cU_q$) and hence defines a bounded operator on $L^2.$
Moreover, since we have $N_q(L T L) = N_q(L)(\Id + \Pi) = N_q(L),$ it follows that $N_q(LS)=0$ and so $LS$ is $x$ times a twisted $0$-pseudodifferential operator.

For $u \in \cD_{\max, \bB}(L)$ supported in $\cU_q$ we have $u = TLu + Su$ and we define
\begin{equation*}
	u_n = TLu + S(x^{1/n}u).
\end{equation*}
Since $x^{1/n}u \to u$ in $L^2,$ we have $S(x^{1/n}u) \to Su$ and $LS(x^{1/n}u) \to LSu.$
It follows that 
\begin{equation*}
	Lu_n = LTLu + LS(x^{1/n}) \to LTLu + LSu = Lu
\end{equation*}
and hence $u_n \to u$ in the graph norm of $L.$
Finally note that $TLu \in x^\delta L^2$ for some $\delta>0$ and, for $n$ sufficiently large $S(x^{1/n}u) \in x^{1/n}L^2,$ and so $u_n \in x^{0+}L^2$ as required.

(ii)
Assume that $\mbox{supp}\, u \subset \cU_q$, $q \in Y^{k+1}$. If $u \in \cD_{\max, \bB}(L) \cap x L^2(X;E)$ and $v \in \cD_{\max, \bB^{\star}}(L^*)$, then
\begin{equation*}
	\ang{Lu, v} = \ang{(Lx)\bar u, v} = \ang{\bar u, (Lx)^*v} = \ang{u, L^*v}
\end{equation*}
where the integration by parts is justified since $Lx$ is a partially completed edge operator at $Y^{k+1}$ and $u$ and $v$ satisfy adjoint boundary 
conditions at $Y^i,$ $i \leq k.$ Next, if $u \in \cD_{\max, \bB}(L) \cap x^{1-}L^2(X;E),$ then by the previous case, $u_n = x^{1/n}u \in 
\cD_{\min, \bB}(L)$ and clearly $u_n$ converges to $u$ in $L^2.$ Now, for $v \in \cD_{\max, \bB^{\star}}(L^*) \cap x^{0+}L^2(X;E),$ we have
\begin{equation*}
	[u,v]_L = 
	\ang{Lu, v} - \lim_{n\to\infty} \ang{u_n, L^*v}
	= \lim_{n\to\infty} \ang{L(u-u_n), v}
	= \lim_{n\to\infty} -\tfrac1n \ang{\sigma(L)(dx) x^{1/n-1}u, v}.
\end{equation*}
To see that this vanishes, first note first that $\sigma(L)(dx) \in \calC^\infty(\wt X; \operatorname{Hom}(E))$ induces a bounded operator on $L^2,$ and secondly that $(x^{-1}u,v)$ is in $L^1(X).$ Once we know for all $v \in \cD_{\max, \bB^{\star}}(L^*) \cap x^{0+}L^2(X;E)$ that $[u,v]_L=0,$  it follows from ($i$) that $[u,v]_L=0$ for all $v \in \cD_{\max, \bB^{\star}}(L^*)$ and hence that $u \in \cD_{\min, \bB}(L).$ 
\end{proof}

Combining this with the argument preceding this lemma, we have now shown that the `minimal domain at $Y^{k+1}$' coincides 
the subdomain of the maximal domain obtained by imposing zero Cauchy data at $Y^{k+1}.$ 
\begin{proposition}\label{prop:DminB}
If $u \in \cD_{\max, \bB}(L)$ is supported in a distinguished neighborhood of a point $q \in Y^{k+1},$ then 
\begin{equation*}
	u \in \cD_{\min, \bB}(L) \iff u \in \cD_{\max, \bB}(L) \cap x^{1-}L^2(X;E) \iff u \in \cD_{\max, (\bB,\Id)}(L).
\end{equation*}
\end{proposition}

To proceed, now fix $\tau \in (0,1]$ and consider the pairing 
\begin{equation*}
	\ang{\cdot,\cdot}_{x^\tau L^2 \times L^2}:x^{\tau}L^2 \times L^2 \lra \bbC
\end{equation*}
induced by the $x^{\tau/2}L^2$ inner product. The operator $P=xL$ has a formal transpose with respect to this pairing which we 
denote $P^{\dagger}$. This is related to the formal $L^2$-adjoint of $L$ by $P^\dagger = x^{\tau}L^*x^{1-\tau}.$

\begin{proposition}\label{prop:IntByParts}
Let $\tau \in (0,1],$
\begin{equation*}
	u \in \cD_{\max, \bB}(L), \quad v \in \cD_{\max, \bB^{*}}(L^*), \quad P^\dagger v \in x^\tau L^2
\end{equation*}
and assume that at least one of $u,v$ is in the regular domain \eqref{eq:DefRegDom} and supported in a distinguished neighborhood 
of a point in $Y^{k+1}.$ Then
\begin{multline*}
	\ang{ u, P^\dagger v}_{L^2 \times x^{\tau}L^2} 
	- \ang{P^{\phantom{\dagger}}u, v}_{x^{\tau}L^2\times L^2}
	= 
	\sum_{\substack{ \zeta_j \in \cS\cD(L) \\ \Re(\zeta_j)<\tau - \tfrac{f+1}2}}
	\oint_{\Gamma_j} \cM(u)(\zeta) \bar{ \cM(P^\dagger v)(-\zeta) } \; d\xi dy\dvol_Z \\
	= 
	\sum_{\substack{ \zeta_j \in \cS\cD(L) \\ \Re(\zeta_j)<\tau - \tfrac{f+1}2}}
	\oint_{\Gamma_j} \cM(P u)(\zeta) \bar{ \cM(v)(-\zeta) } \; d\xi dy\dvol_Z.
\end{multline*}
where $\Gamma_j$ is a small counterclockwise contour around $\zeta_j.$
\end{proposition}

\begin{proof}
For notational convenience, conjugating $P$ by $x^{\tau/2}$, we can assume that 
\begin{equation*}
	u, v \in x^{-\tau}L^2 \Mand 
	Pu, P^\dagger v \in x^{\tau}L^2.
\end{equation*}
The natural pairing between $x^{-\tau}L^2$ and $x^{\tau}L^2$ is the usual pairing on $L^2.$
Also for convenience assume that $u \in \cD_{\max, \bB}^{\reg}(L)$ and both $u$ and $v$ are supported in a distinguished neighborhood of a point in $Y^{k+1}.$

Parseval's formula for the Mellin transform gives
\begin{equation*}
\begin{gathered}
	\ang{u, P^\dagger v}_{L^2} = 
	\int_{\eta = -\tau} \cM(u)(\zeta) \bar{\cM(P^\dagger v)(-\zeta)} \; d\xi dy\dvol_Z \\
	\ang{Pu, v}_{L^2} =
	\int_{\eta = \tau} \cM(Pu)(\zeta) \bar{\cM(v)(-\zeta)} \; d\xi dy\dvol_Z.	
\end{gathered}
\end{equation*}
Using 
that $u \in \cD_{\max,\bB}^{\reg}(L),$ we can integrate by parts in the latter integral to get
\begin{equation*}
	\int_{\eta = \tau} \cM(u)(\zeta) \cM(P^\dagger v)(-\zeta) \; d\xi dy\dvol_Z
\end{equation*}
where the integrand is now interpreted as a pairing between 
\begin{equation*}
	H^{2\tau}(dy; L^2(d\xi \dvol_Z; E)) \Mand
	H^{-2\tau}(dy; L^2(d\xi \dvol_Z;E)).
\end{equation*}
Thus
\begin{equation}\label{eq:FirstExp}
	\ang{u, P^\dagger v}_{L^2} 
	- \ang{P^{\phantom{\dagger}}u, v}_{L^2} 
	= 
	\oint_{\Gamma} \cM(u)(\zeta) \cM(P^\dagger v)(-\zeta) \; d\xi dy\dvol_Z 
\end{equation}
where $\Gamma$ is a simple closed contour, traversed counterclockwise, surrounding the poles of the integrand with imaginary part between $-\tau$ and $\tau.$
In the same way we see that
\begin{equation}\label{eq:SecondExp}
	\ang{u, P^\dagger v}_{L^2} 
	- \ang{P^{\phantom{\dagger}}u, v}_{L^2} 
	= 
	\oint_{\Gamma} \cM(P u)(\zeta) \cM(v)(-\zeta) \; d\xi dy\dvol_Z.
\end{equation}
\end{proof}

We can now establish that the regular domain is a core domain.

\begin{theorem} \label{thm:CoreDomain}
Let $L \in \Diff_{\iie}^1(X;E)$ with  local ideal boundary conditions $\bB$ at $Y^1, \ldots, Y^k$ satisfying Assumption \ref{Ass:InvNormalOp}. 
The graph closure of the domain $\cD_{\max, \bB}^{\reg}(L)$ is $\cD_{\max, \bB}(L).$
\end{theorem}

\begin{proof}
Since $\cD_{\max, \bB}^{\reg}(L) \subseteq \cD_{\max, \bB}(L),$ it is enough to show that $\cD_{\max, \bB}^{\reg}(L)^* \subseteq \cD_{\max, \bB}(L)^*.$
Thus assume that $v \in \cD_{\max}(L^*)$ satisfies
\begin{equation}\label{eq:RegAdjoint}
	[u, v]_{L} = 0 \Mforall u \in \cD_{\max, \bB}^{\reg}(L). 
\end{equation}
We must show that $[u,v]_L=0$ for all $u \in \cD_{\max, \bB}(L).$

Directly from the definition of adjoint boundary conditions, if $v$ satisfies \eqref{eq:RegAdjoint} then $v \in \cD_{\max, \bB^{\star}}(L^*).$
Since $\cD_{\max, \bB}^{\reg}(L)$ and $\cD_{\max, \bB}(L)$ only differ at $Y^{k+1},$ it suffices to consider $v$ supported in a distinguished neighborhood of a point in $Y^{k+1}.$

From Proposition \ref{prop:IntByParts} with $\tau =1$ we see that \eqref{eq:RegAdjoint} is equivalent to $\cM(v)$ being holomorphic on $\{\Im \zeta < 1 - \tfrac12(f+1) \},$ for all $\chi \in A(Y^{k+1}).$ By Lemma \ref{lem:ExpansionExistence} this in turn is equivalent to knowing that the Cauchy data of $v$ at $Y^{k+1}$ vanishes.
Then, from Proposition \ref{prop:DminB}, this is equivalent to what we wished to show,
\begin{equation*}
	v \in \cD_{\min, \bB^{\star}}(L^*) = (\cD_{\max, \bB}(L))^*.
\end{equation*}
\end{proof}

Although the full strength of this theorem is used occasionally below,  most frequently we use the following consequence.

\begin{corollary}\label{cor:IdentityTransfer}
If $\bar a= \min \{ \Re\zeta : \zeta \in \cS\cD(L) \}$ and $a<\bar a$ then
\begin{equation*}
	\cD_{\max, \bB}(L) \cap x^a L^2(X;E)
\end{equation*}
is dense in $\cD_{\max, \bB}(L)$ with respect to the graph norm of $L.$
\end{corollary}

\begin{proof}
As pointed out above,
\begin{equation*}
	\cD_{\max, \bB}^{\reg}(L) \subseteq \cD_{\max, \bB}(L) \cap x^a L^2(X;E),
\end{equation*}
so the corollary follows from the theorem.
\end{proof}

\section{The de Rham operator}
In this section we assume that $\hat X$, with singular strata $Y^1, \ldots, Y^{k+1}$, ordered by increasing depth,
is endowed with a {\em rigid} $\iie$ metric $g$, and we study the associated de Rham operator $\eth_{\dR}.$
We first show that any choice of Cheeger ideal boundary conditions yields a self-adjoint extension. 
We then show inductively that Assumptions \ref{Ass:CstIndicialRoots}, \ref{AssIndFamily} and \ref{AssCmpDomain} 
all hold in this case. This is done through the analysis of its indicial and normal operators.

\subsection{The indicial operator of $\eth_{\dR}$} \label{sec:IndicialDeRham}
Our first task is to compute the indicial roots of $\eth_{\dR}$; in particular, we simplify and sharpen 
the computation from \cite{ALMP} and extend it to non-Witt spaces. Assume that we have chosen local ideal boundary conditions 
\begin{equation*}
	\bB = (B^1, \ldots, B^k)
\end{equation*}
at the first $k$ strata, and work near $q \in Y^{k+1}.$  The link of $\hat X$ at $q$ is denoted $Z$ and $x$ is a boundary 
defining function for $Y^{k+1}.$ Assume inductively that $(\eth_{\dR}^Z, \cD_{\bB}(\eth_{\dR}^Z))$ has compact resolvent,
and hence discrete spectrum, as an operator on $L^2(Z; \Lambda^*(\Iie T^*Z)).$ Assume too that there is a strong 
Kodaira decomposition on $Z$. This will be justified  in \S \ref{sec:L2Coho}. \\

It is not hard to deduce from \eqref{eq:dNearBdy} and \eqref{eq:deltaNearBdy} that
\begin{equation}\label{eq:NormalOperator}
	N_q(x\eth_{\dR}) =
	\begin{pmatrix} 
	s\eth_{\dR}^{\bbR^h} & -s\pa_s -f \\
	s\pa_s & -s\eth_{\dR}^{\bbR^h}
	\end{pmatrix}
	+
	\begin{pmatrix} 
	\eth_{\dR}^{Z_q} & \bN \\
	\bN & -\eth_{\dR}^{Z_q}
	\end{pmatrix} ,
\end{equation}
and we want to understand the indicial roots of this operator. 

First, $N_q(x\eth_{\dR})$ preserves $\ker \eth_{\dR}^Z,$ and in this subspace the indicial roots are the values of 
$\zeta$ for which 
\begin{equation*}
	\begin{pmatrix}
	0 & -\zeta +\bN -f \\ \zeta +\bN & 0 
	\end{pmatrix}
\end{equation*}
is not invertible. Thus
\begin{equation*}
	\spec_b(x\eth_{\dR}\rest{\ker \eth_{\dR}^Z}) = \{ -k, k-f : \ker \eth_{\dR}^Z  \cap \text{ k-forms } \neq \{ 0\}  \}.
\end{equation*}

To study the indicial roots on the orthogonal complement of this nullspace, we compute the indicial roots 
of a closely related operator. Conjugating $N_q(x\eth_{\dR}^Z)$ by $s^{-f/2}$ gives 
\begin{equation*}
	\begin{pmatrix} 
	s\eth_{\dR}^{\bbR^h} & -s\pa_s  \\
	s\pa_s & -s\eth_{\dR}^{\bbR^h}
	\end{pmatrix}
	+
	\begin{pmatrix} 
	\eth_{\dR}^{Z_q} & \bN -\tfrac f2\\
	\bN - \tfrac f2 & -\eth_{\dR}^{Z_q}
	\end{pmatrix} .
\end{equation*}
Then 
\begin{multline*}
\lrpar{ s^{f/2} N_q(x\eth_{\dR}^X) s^{-f/2} }^2 = s^{f/2} N_q( (x\eth_{\dR}^X)^2 ) s^{f/2} \\
	= (s^2\Delta^{\bbR^h} - (s\pa_s)^2)
	\begin{pmatrix} 1 & 0 \\ 0 & 1 \end{pmatrix}
	+
	\begin{pmatrix} 
	\Delta^Z + (\bN-\tfrac f2)^2 & -d^Z + \delta^Z \\
	d^Z - \delta^Z & \Delta^Z + (\bN-\tfrac f2)^2
	\end{pmatrix}.
\end{multline*}
The two summands here commute, so the next step is to compute the spectrum of the second summand
on the right; this operator is in fact the square of
\begin{equation*}
	D = \begin{pmatrix} \eth^Z_{\dR} & \bN - \tfrac f2 \\ \bN - \tfrac f2 & -\eth^Z_{\dR} \end{pmatrix}.
\end{equation*}
We assume inductively that $\eth_{\dR}^Z,$ and hence $D,$ is self-adjoint, with compact resolvent, and this means 
that the spectrum of $D^2$ is non-negative and discrete. 

Thus, fix any $\lambda^2>0,$ write $\wt\bN = \bN - \tfrac f2$, and decompose $a = a_d + a_{\delta}$, and similarly for $b$,
according to the Kodaira decomposition on $Z$. (We have already handled the case where $a$ and $b$ are harmonic.) Then
we must solve 
\begin{equation*}
	\begin{cases}
	(\Delta^Z+\wt\bN^2 - \lambda^2)a -db + \delta b =0\\
	(\Delta^Z+\wt\bN^2 - \lambda^2)b +da - \delta a =0
	\end{cases}
	\iff
	\begin{cases}
	(\Delta^Z+\wt\bN^2 - \lambda^2)a_d =db_{\delta}\\
	(\Delta^Z+\wt\bN^2 - \lambda^2)a_\delta = - \delta b_d\\
	(\Delta^Z+\wt\bN^2 - \lambda^2)b_d = - da_\delta \\
	(\Delta^Z+\wt\bN^2 - \lambda^2)b_\delta =  \delta a_d
	\end{cases}
\end{equation*}

Start with $a_d$: first apply $(\Delta^Z + (\wt\bN-1)^2 -\lambda^2)$ to both sides of the first equation to get
\begin{equation*}
\begin{gathered}
	(\Delta^Z + (\wt\bN-1)^2 -\lambda^2)(\Delta^Z+\wt\bN^2 - \lambda^2)a_d 
	=(\Delta^Z + (\wt\bN-1)^2 -\lambda^2)db_{\delta}
	=d(\Delta^Z + \wt\bN^2 -\lambda^2)b_{\delta}
	= \Delta^Z a_d\\
	\implies
	\lrspar{
	(\Delta^Z + (\wt\bN-1)^2 -\lambda^2)(\Delta^Z+\wt\bN^2 - \lambda^2) - \Delta^Z }a_d =0.
\end{gathered}
\end{equation*}
This last operator factors as 
\begin{equation*}
\begin{gathered}
	(\Delta^Z + (\wt\bN-1)^2 -\lambda^2)(\Delta^Z+\wt\bN^2 - \lambda^2) - \Delta^Z \\
	= 
	(\Delta^Z + \wt\bN^2 - \wt\bN -\lambda^2)^2 - \lambda^2\\
	= (\Delta^Z + \wt\bN^2 - \wt\bN -\lambda^2 + \abs{\lambda} )(\Delta^Z +\wt\bN^2 - \wt\bN - \lambda^2 - \abs{\lambda}).
\end{gathered}
\end{equation*}
We can remove the absolute value from $\lambda$ since $\Spec(D)$ is symmetric about $0$; restricting to an 
eigenspace of $\Delta^Z$ with eigenvalue $\mu$, we find that 
\begin{equation*}
	\lambda \in 
	\lrbrac{
	\pm \tfrac12 \pm \sqrt{ (k-\tfrac f2 - \tfrac12)^2 + \mu } : \mu \in \Spec(\Delta^Z\rest{k-\text{forms}}) }.
\end{equation*}

Next consider $a_{\delta}$: applying $(\Delta^Z + (\wt\bN+1)^2-\lambda^2)$ to both sides of the second equation above gives
\begin{equation*}
\begin{gathered}
	(\Delta^Z + (\wt\bN+1)^2 -\lambda^2)(\Delta^Z+\wt\bN^2 - \lambda^2)a_\delta
	= \Delta^Z a_\delta\\
	\implies
	(\Delta^Z + \wt\bN^2 + \wt\bN -\lambda^2- \lambda)
	(\Delta^Z + \wt\bN^2 + \wt\bN -\lambda^2+ \lambda)
	a_\delta =0
\end{gathered}
\end{equation*}
and hence
\begin{equation*}
	\lambda \in 
	\lrbrac{
	\pm \tfrac12 \pm \sqrt{ (k-\tfrac f2+\tfrac12)^2 + \mu } : \mu \in \Spec(\Delta^Z\rest{k-\text{forms}}) }.
\end{equation*}
Altogether we conclude that
\begin{equation*}
	\Spec(D) \subseteq 
	\bigcup_k
	\lrbrac{
	\pm \tfrac12 \pm \sqrt{ (k -\tfrac f2 \pm \tfrac12)^2 + \mu } : \mu \in \Spec(\Delta^Z\rest{k-\text{forms}}) }
\end{equation*}
where all of the $\pm$ signs can be chosen independently.\\

On the $\lambda^2$ eigenspace of $D^2,$ 
\begin{equation*}
	I_q( s^{f/2} N_q( (x\eth_{\dR}^Z)^2 ) s^{-f/2} )(\zeta) =
	(\lambda^2- \zeta^2)
	\begin{pmatrix} 1 & 0 \\ 0 & 1 \end{pmatrix}.
\end{equation*}
This means that the indicial roots of $x\eth_{\dR}^Z$ are contained in
\begin{equation*}
\begin{gathered}
	\bigcup_{k \in [0,f]} \{ -k, k-f : 0 \in \Spec(\Delta^Z\rest{k-\text{forms}}) \} \\
	\cup \bigcup_{k \in [0,f]} \lrbrac{
	-\tfrac f2\pm \tfrac12 \pm \sqrt{ (k-\tfrac f2 \pm \tfrac12)^2 + \mu } : \mu \in \Spec(\Delta^Z\rest{k-\text{forms}})\setminus \{ 0 \} }.
\end{gathered}
\end{equation*}

This proves the 
\begin{lemma}\label{lem:SuitableScaling}
Let $(\hat X, g)$ be a stratified pseudomanifold with $\iie$ metric, with all notation as before, and $\bB = (B^1, \ldots, B^k)$ 
a set of Cheeger ideal boundary conditions corresponding to a mezzoperversity at the first $k$ strata. Let $f$ be the dimension 
of the link of $\hat X$ at $Y^{k+1}$. 
\begin{itemize}
\item If $f$ is even and the induced Hodge Laplacians on the links $Z_q$ of $\hat X$ at $q\in Y^{k+1}$ satisfy
\begin{equation*}
	\Spec(\Delta^Z\rest{j-\text{forms}}) \cap [-\tfrac 34, \tfrac 34] \subseteq \{ 0 \} \quad
	\Mwhenever |j-\tfrac f2|\leq 1, 
\end{equation*}
then 
\begin{equation*}
	\spec_b(x\eth_{\dR}^X;\bB) \cap [-\tfrac f2 - \tfrac12, -\tfrac f2 + \tfrac12 ] \subseteq \{ -\tfrac f2 \}
\end{equation*}
\item
If $f$ is odd  and the induced Hodge Laplacians on the links $Z_q$ of $\hat X$ at $q\in Y^{k+1}$ satisfy
\begin{equation*}
	\Spec(\Delta^Z\rest{j-\text{forms}}) \cap (- 1,  1) \subseteq \{ 0 \} \quad
	\Mwhenever |j-\tfrac f2|=\tfrac 12, 
\end{equation*}
then 
\begin{equation*}
	\spec_b(x\eth_{\dR}^X;\bB) \cap [-\tfrac f2 - \tfrac12, -\tfrac f2 + \tfrac12 ] \subseteq \{ -\tfrac f2 \pm \tfrac12 \}
\end{equation*}
\end{itemize} 
In particular, any $\iie$ metric $g$ can be replaced by another one where the metric on the link at $Y^{k+1}$ is scaled by a 
large constant so that these condition as well as {\em Assumption} \ref{Ass:CstIndicialRoots} are all true. 
\end{lemma}

We say that $g$ is {\bf suitably scaled} if the Hodge Laplacians of the links satisfy the hypotheses of Lemma \ref{lem:SuitableScaling}.
We are abusing notation slightly by regarding this as a property of the metric, since it also depends on the ideal boundary conditions 
$\bB$. 

\begin{corollary}\label{cor:DeRhamExp}
Let $(\hat X,g)$ be a stratified pseudomanifold with Cheeger ideal boundary conditions $\bB$ at the first $k$ strata, and a 
suitably scaled metric. If $u \in \cD_{\max, \bB}(\eth_{\dR})$ is supported in a distinguished neighborhood 
$\cU_q \cong [0,1) \times \bbB^h \times Z_q$,  $q \in Y^{k+1},$ then 
as in Lemma \ref{lem:ExpansionExistence}, 
\begin{multline*}
u \sim x^{-f/2}u_{f/2} + \wt u, \qquad \text{where} \\[0.5ex]
u_{f/2} \in H^{-1/2}(\bbB^{h}; \ker I(x\eth_{\dR}; \bB, -f/2)) 
\Mand \wt u \in x^{1-} L^2( C(Z_q), H^{-1}(\bbB^h)\otimes \Lambda^*(\Iie T^*X) ).
\end{multline*}
If $f = \dim Z_q$ is odd or $\cH^{\mid}_{\bB(Z_q)}(Z_q) = 0$ then $u_{f/2}=0$; in all other cases, 
\begin{equation*}
	u_{f/2} = \alpha(u) + dx \wedge \beta(u), \quad
	\alpha(u), \beta(u) \in H^{-1/2}(\bbB^h; \Lambda^* (\Iie T^*Y) \otimes \cH^{\mid}_{\bB}(Z_q)).
\end{equation*}
\end{corollary}
To put this into perspective, recall that $dV_g \approx x^f \; dxdy\dvol_Z$; thus since  $x^{-f/2} \in x^{\eps}L^2_{\loc}(x^f \; dx)$ 
for every $\eps <1/2,$ the expansion here is of the form studied in Lemma \ref{lem:ExpansionExistence}.

\subsection{The normal operator of $\eth_{\dR}$}
We now turn to a consideration of the normal operator of $\eth_{\dR}$ at a point $q \in Y^{k+1}$. The link of $\hat X$ 
at $q$ is denoted $Z$ and $x$ is a boundary defining function for $Y^{k+1}.$ We assume that 
the metric is suitably scaled, 
and that there is Kodaira decomposition for forms on $Z$ as in \S\ref{sec:L2Coho}.

Since we will be working with $\iie$-differential forms on $\bbR^+_s \times T_qY^{k+1} \times Z^{k+1}_q,$ it simplifies
notation to write 
\begin{equation}\label{eq:AbbL2}
\begin{gathered}
	Y = Y^{k+1}, \quad Z = Z^{k+1}_q, \quad h = \dim Y, \quad f = \dim Z, \\
	T_qY^+ = \bbR^+_s \times T_qY, \quad 
	L^2 = L^2(T_qY^+ \times Z; \pi^*_q\Lambda^*(\Iie T^*X)), \quad 
	P = x\eth_{\dR}.
\end{gathered}
\end{equation}
The normal operator at $q$ acts on sections of $\Lambda^*(\Iie T^*X)$ pulled-back to $\cU_q$ and restricted to the fiber 
over $q$, which can be identified with $\Lambda^*(T^*_qY^+) \hat\otimes \Lambda^*(\Iie T^*Z)$ (the graded tensor product 
of forms on $Y^+$ and $\iie$ forms on $Z$). Thus we also write
\begin{equation*}
	\Lambda^*_q = \pi^*_q\Lambda^*(\Iie T^*X)) = \Lambda^*(T^*_qY^+) \hat\otimes \Lambda^*(\Iie T^*Z).
\end{equation*}

Once we impose the boundary conditions $\bB$ at $Y^1, \ldots, Y^k,$ the normal operator of $P$ at a point $q \in Y^{k+1}$ naturally induces an unbounded operator on $L^2$  but it is also profitable to consider the unbounded operators it induces on weighted $L^2$ spaces. For any $\eps \geq 0$ we have
\begin{equation*}
	N_q(P): \cD^{\eps}_{\max, N(\bB})(N_q(P)) \subseteq s^{\eps}L^2(T_qY^+ \times Z; \Lambda^*_q) \lra  s^{\eps}L^2(T_qY^+ \times Z; \Lambda^*_q)
\end{equation*}
where
\begin{equation*}
	\cD^{\eps}_{\max, N(\bB)}(N_q(P)) = \{ u \in \cD_{\max, N(\bB)}(N_q(P)) \cap s^{\eps}L^2 : N_q(P)u \in s^{\eps}L^2 \}.
\end{equation*}
This is a closed operator and, in \cite[\S5.4]{ALMP} we have shown that the two assumptions
\begin{equation}\label{Ass2}
\begin{gathered}
	a) \ \Spec(\eth_{\dR}^{Z}) \cap (-1,1) \subseteq \{0\},  \quad
	b)  \ \cH^{\tfrac12\dim Z}_{L^2}(Z)=0. 
\end{gathered}
\end{equation} 
imply that $(N_q(P), \cD_{\max, N(\bB)}^{\eps}(N_q(P)) )$ is injective for all $\eps \in (0,1).$

{\em Now consider the situation where ($a$) holds but ($b$) does not.} Thus we assume that $f = \dim Z$ 
is an even integer and, letting $\cD_{\bB(Z)}(\eth_{\dR}^Z)$ be the domain of $\eth_{\dR}^Z$ induced 
by the boundary conditions $\bB,$ that 
\begin{equation*}
	\cH^{f/2}_{\bB(Z)}(Z) = \ker (\eth_{\dR}^Z, \cD_{\bB(Z)}(\eth_{\dR}^Z))\rest{\text{degree} = f/2}\neq \{ 0 \}.
\end{equation*}

Because of the inductive hypothesis that the de Rham operator on the fibers is Fredholm and self-adjoint, 
and that the de Rham complex on the fibers satisfies a strong Kodaira decomposition and a Hodge theorem, 
we can use the computation in  \cite[Lemma 5.5]{ALMP} here. In particular, with $P = x\eth_{\dR},$ this 
shows that elements in $\ker (N_q(P), \cD_{\max, N(\bB)}(N_q(P)))$ arising from the fact that b) no longer holds
are necessarily of the form $v= \alpha + dx \wedge \beta$ with 
\begin{equation*}
	\alpha, \beta \in L^2(T_qY^+; \Lambda^*(T_qY^+) \hat\otimes \cH^{f/2}_{\bB(Z)}(Z))
\end{equation*}
From \eqref{eq:NormalOperator}, on forms of middle vertical degree, the equation $N_q(P)v=0$ reduces to
\begin{equation*}
\begin{gathered}
	s\eth_{\dR}^{\bbR^h}\alpha - s\pa_s \beta - \tfrac{f}2\beta =0 \\
	s\pa_s \alpha + \tfrac{f}2\alpha - s\eth_{\dR}^{\bbR^h}\beta =0
\end{gathered}
\end{equation*}
Setting $(\alpha', \beta')= (s^{f/2}\alpha, s^{f/2}\beta)$, then these become
\begin{equation}\label{NullSpacePair}
	\begin{cases}
	\eth^{\bbR^h}_{\dR}\alpha' = \pa_s \beta' \\[0.5ex]
	\eth^{\bbR^h}_{\dR}\beta' = \pa_s \alpha'.
	\end{cases}
\end{equation}
Since $\pa_s$ and $\eth^{\bbR^h}_{\dR}$ commute, \eqref{NullSpacePair} implies that $\alpha', \beta' \in \ker (-\pa_s^2 + \Delta^{\bbR^h})$.
Taking Fourier transform in $\bbR^h,$ (with dual variable $\eta$), we find that 
\begin{equation*}
	\pa_s^2\cF(\alpha') = |\eta|^2\cF(\alpha') \Longrightarrow 
	\cF(\alpha') = A(\eta)e^{-|\eta|s} + B(\eta) e^{|\eta|s}.
\end{equation*}
However, a solution lying in any $s^\delta L^2$ must have $B(\eta) \equiv 0$. Moreover, 
\begin{equation*}
	\int_0^{\infty} \int_{\bbR^h} |A(\eta)e^{-|\eta|s}|^2 \; ds \; d\eta 
	= \int_{\bbR^h} \frac{|A(\eta)|^2}{2|\eta|} \; d\eta
\end{equation*}
so the $L^2$ condition becomes 
\begin{equation*}
	\cF(\alpha') = A(\eta)e^{-|\eta|s}, \Mwith A(\eta) \in |\eta|^{1/2}L^2(\bbR^h; \Lambda^* T^*_q Y \hat \otimes  \cH^{f/2}_{\bB(Z)} (Z)).
\end{equation*}
Given $\alpha',$ then \eqref{NullSpacePair} determines $\beta'$, so altogether
\begin{multline*}
	\ker (N_q(P), \cD_{\max,\bB'}( N_q(P) ) )\\
	=s^{-f/2}\{ \cF^{-1}_q (A(\eta) e^{-s|\eta |}, -i\cl{\hat{\eta}} A(\eta) e ^{-s|\eta |})
		\;|\: A(\eta)\in |\eta|^{1/2} L^2 (\bbR^h, \Lambda^* T^*_q Y\otimes \cH^{f/2}_{\bB(Z)} (Z) )\}\,
\end{multline*} 
where $\hat{\eta}= \frac{ \eta}{| \eta |}$.

Since the relation between $\alpha$ and $\beta$ in $\ker N_q(P)$ is symmetric, we have
\begin{equation}\label{Eq:SigmaInvolution}
	\alpha = \hat\sigma(q) \beta, \quad \beta = \hat\sigma(q) \alpha
\end{equation}
where $\hat\sigma$ is the involution $\cF^{-1}_q \circ \tfrac1i\cl{\hat\eta} \circ \cF_q.$

\begin{proposition}\label{prop:NullNormal}
The operator
\begin{equation*}
	N_q(P):  \cD_{\max, N(\bB)}^{\eps}(N_q(P)) \subseteq
	s^{\eps}L^2 \lra s^{\eps}L^2
\end{equation*}
is injective for $\eps\geq 1/2$ and has closed range if $0 < \eps < 1$, $\eps \neq 1/2$.  Moreover if $\eps<1/2$,
then its null space is 
\begin{equation}\label{null-normal}
	s^{-f/2}\{ \cF^{-1}_q (A(\eta) e^{-s|\eta |},  -i\cl{\hat{\eta}} A(\eta) e ^{-s|\eta |}): 
		A(\eta)\in |\eta|^{1/2} L^2 (\bbR^h, \Lambda^* T^*_q Y\hat \otimes \cH^{f/2}_{\bB(Z)} (Z) )\}
\end{equation}
Observe that the space \eqref{null-normal} is an infinite dimensional subspace of $s^{1/2-}L^2$ and hence, in particular, the normal operator is Fredholm
if and only if it is invertible.
\end{proposition}

\begin{proof}
Closedness of the range follows from Proposition \ref{eq:GNqP}. The other statements follow directly from the computation.
\end{proof}

Since, by induction, the de Rham operator with Cheeger ideal boundary conditions coming from a mezzoperversity 
is self-adjoint, see Theorem \ref{thm:SelfDual}, the following holds by duality.
\begin{corollary}\label{cor:NormalOpSurj}
The normal operator with domain $\cD_{\min, N(\bB)}^{\eps}(N_q(x\eth_{\dR}))$, defined by duality from 
$\cD_{\max,N(\bB)}^{1-\eps}(N_q(x\eth_{\dR})),$ is surjective when $\eps<1/2.$
\end{corollary}

This description of the null space yields a direct sum decomposition of the trace bundle of the de Rham operator at $Y.$
Assume now that $W = W^{k+1}$ is a sub-bundle of 
\begin{equation*}
	\cH^{f/2}_{\bB(Z)}(H/Y) \lra Y
\end{equation*}
with $H\lra Y$ the fibration with typical fiber $Z.$

\begin{lemma}\label{lem:decomp} There is a direct sum decomposition
\begin{equation}\label{Eq:DecopToShow}
	H^{-1/2}( T_qY; \Lambda^*T^*_qY \hat\otimes \cH^{f/2}_{\bB(Z)}(Z))
	\oplus
	H^{-1/2}( T_qY; \Lambda^*T^*_qY \hat\otimes \cH^{f/2}_{\bB(Z)}(Z))
	  = C_{\ker} \oplus C_W
\end{equation}
where 
\begin{equation*}
\begin{gathered}
	C_{\ker}=\{ \cF^{-1}_y (A(\eta) , -i\cl{\hat{\eta}} A(\eta) ) :
	A(\eta)\in |\eta|^{1/2} L^2 (\R^h, \Lambda^*T^*_q Y \hat \otimes \cH^{f/2}_{\bB(Z)} (Z) )\} \\
	C_W=
	H^{-1/2}( T_qY; \Lambda^*T^*_qY \hat\otimes W_q)
	\oplus H^{-1/2}( T_qY; \Lambda^*T^*_qY \hat\otimes W^{\perp}_q)
\end{gathered}
\end{equation*}
\end{lemma}

\begin{proof}
Note that 
\begin{equation*}
	\alpha 
	\in H^{-1/2}( T_qY; \Lambda^*T^*_qY \hat\otimes W_q) 
	\implies
	\cF_y^{-1}( -i \cl{\hat\eta} \cF_y(\alpha))
	\in H^{-1/2}( T_qY; \Lambda^*T^*_qY \hat\otimes W_q)
\end{equation*}
and so $C_{\ker} \cap C_W = \{0 \}.$ Thus we only need to show that $C_{\ker} \oplus C_W$ contains the left hand side of \eqref{Eq:DecopToShow}.

Decompose any $\gamma, \kappa \in H^{-1/2}( T_qY; \Lambda^*T^*_qY \hat\otimes \cH^{f/2}_{\bB(Z)}(Z))$ as 
\begin{equation*}
	\gamma = \gamma_W + \gamma_{W^\perp}, \quad
	\kappa = \kappa_W + \kappa_{W^\perp}
\end{equation*}
according to the splitting $C_W$. Let
\begin{equation*}
\begin{gathered}
	( \cF_y^{-1}( -i \cl{\hat\eta} \cF_y(\kappa_W)), \kappa_W) := (\omega, \kappa_W) \in C_{\ker} \\
	(\gamma_{W^{\perp}}, \cF_y^{-1}( -i\cl{\hat\eta} \cF_y(\gamma_{W^{\perp}}) ) ) := (\gamma_{W^{\perp}}, \eta) \in C_{\ker}
\end{gathered}
\end{equation*}
and notice that $(\omega, \eta) \in C_W.$ We can now decompose an arbitrary pair $(\gamma, \kappa)$ as
\begin{equation*}
	(\gamma, \kappa) = (\omega+\gamma_{W^\perp}, \kappa_W+\eta) +  (\gamma_W - \omega, \kappa_{W^\perp}-\eta) \in C_{\ker} \oplus C_W.
\end{equation*}
\end{proof}

Define
\begin{equation}\label{eq:DefNeth}
\begin{gathered}
	N_q(\eth_{\dR}) = \tfrac1sN_q(x\eth_{\dR}), \\
	\cD_{\max, N(B)}(N(\eth_{\dR})) 
	= \{ u \in \cD_{\max, N(B)}(N_q(x\eth_{\dR})) : N_q(\eth_{\dR})u \in L^2 \},
\end{gathered}
\end{equation}
then by Lemma \ref{lem:ExpansionExistence}, elements in $N_q(\eth_{\dR})$ have a partial asymptotic expansion at $s=0.$
The same indicial root computation of the de Rham operator for a suitably scaled $\iie$ metric implies that this partial 
asymptotic expansion has only one term, and this has exponent $-f/2$, and coefficient in \eqref{Eq:DecopToShow}. 
Denote the corresponding Cauchy data map by $\sC_q,$
\begin{equation}\label{eq:CauchyNq}
	\cD_{\max,N(\bB)}(N_q(\eth_{\dR})) \ni u 
	\xrightarrow{\sC_q} (\alpha_0, \beta_0) \in \lrpar{ H^{-1/2}( T_qY; \Lambda^*T^*_qY \hat\otimes \cH^{f/2}_{\bB(Z)}(Z)) }^2.
\end{equation}

From Lemma \ref{lem:decomp}, there is a (not necessarily orthogonal) projection into the null space
\begin{equation*}
	\Phi_{q}: 
	\cD_{\max, N(\bB)}(N_q(\eth_{\dR})) \lra
	\ker
	(N_q(\eth_{\dR}), \cD_{\max, N(\bB)}(N_q(\eth_{\dR})) ).
\end{equation*}
Indeed, if
\begin{equation}\label{eq:DefPiW}
	\Pi_W: C_{\ker} \oplus C_W \lra C_W
\end{equation}
is the natural projection, and 
\begin{equation*}
\begin{gathered}
	\sP: 
	H^{-1/2}( T_qY; \Lambda^*T^*\bbR \hat\otimes \Lambda^*T^*_qY \hat\otimes \cH^{f/2}_{\bB(Z)}(Z))
	\lra
	s^{1/2-}L^2(\bbR^+_s \times T_qY; \Lambda^*T^*\bbR \hat\otimes \Lambda^*T^*_qY \hat\otimes \cH^{f/2}_{\bB(Z)}(Z))\\
	\sP = s^{-f/2}\cF_y^{-1} \circ e^{-s|\eta|} \circ \cF_y
\end{gathered}
\end{equation*}
then we can set
\begin{equation*}
	\Phi_{q} = \sP (\Id - \Pi_W) \sC_q, 
\end{equation*}
and it follows from \eqref{null-normal} that $\Phi_{q}$ maps into $\ker(N_q(\eth_{\dR}), \cD_{\max, N(\bB)}(N_q(\eth_{\dR})) ).$

\begin{lemma}\label{lem:PhiEps}
The operator $\Phi_{q}$ is a twisted $0$-pseudodifferential operator
\begin{equation*}
	\Phi_{q} \in \Psi_0^{-\infty, \cJ}(\bbR^+\times \bbR^h) \otimes \sK(L^2(Z_q;\Lambda^*))
\end{equation*}
with index set 
\begin{equation*}
\cJ=(J_{10}, J_{01}, J_{11}), \qquad 	J_{10} = E_{10}, \quad J_{01} = H_{01}, \quad J_{11} = \bbN_0,
\end{equation*}
using the notation \eqref{eq:GralInvIndSets} (with $\delta \in (0,1/2)$). 
\end{lemma}

\begin{proof}
We compute the integral kernel of $\Phi_{q}.$ 

Fix $\eps \in (0,1/2),$ we will describe the integral kernel of $\cC_q$ using the generalized inverse 
\begin{equation*}
	G_q = G_{N_q(P)} \quad \text{ of } \quad ( N_q(x\eth_{\dR}), \cD_{\max, N(\bB)}^{\eps}(N_q(x\eth_{\dR})) ).
\end{equation*}
Note that by Corollary \ref{cor:NormalOpSurj}, $( N_q(x\eth_{\dR}), 
\cD_{\max, N(\bB)}^{\eps}(N_q(x\eth_{\dR})) )$ is surjective, so Proposition \ref{prop:InvNormalOp} can be applied 
to $P=x\eth_{\dR}.$

First, since $\Phi_q$ is the identity on $\ker N_q(P),$ we have $\Phi_{q}\Pi_{\ker N_q(P)}v = \Pi_{\ker N_q(P)}v,$ 
so it suffices to focus on those $v\in \cD_{\max, N(\bB)}(N_q(\eth_{\dR})$ orthogonal to $\ker N_q(P).$
Thus 
\begin{equation*}
	v = G_q N_q(P)v
\end{equation*}
and we are looking for the coefficient of $s^{-f/2}$ in the expansion of the left hand side.

Since $G_q \in \Psi_0^{-1,\cH}(\bbR^+\times \bbR^h) \otimes \sK(L^2(Z_q;\Lambda^*_q))$ and 
$N_q(P)v \in sL^2,$ this coefficient comes from the expansion of $\cK_{G_q}$ at the left face, $\cB_{10},$ of $(\bbR^+\times \bbR^h)^2_0 \times Z_q^2.$ Indeed, there is a term
\begin{equation*}
	s^{-f/2} \cK_{(G_q,-f/2)}(s',u,u',z,z')
\end{equation*}
in the expansion of $\cK_{G_q}$ at the left face, and $\sC v$ is given by the action of $\cK_{(G_q,-f/2)}(s',u,u',z,z')$ on $N_q(P)v.$
It is convenient to write this coefficient in terms of the generalized inverse of the Bessel-type operator $\Theta_q(P).$
Indeed, from \eqref{eq:NfromB} (cf. \cite[(5.20)]{Mazzeo:Edge}) we have
\begin{equation*}
	\cK_{(G_q,-f/2)}(s',u,u',z,z') 
	= \int e^{i(u-u')\cdot \eta} \cK_{(G_{\Theta_q(P)}, -f/2)}(s'|\eta|, z, z', \hat \eta)|\eta|^{-f/2} \; d\eta
\end{equation*}
where $\cK_{(G_{\Theta_q(P)}, -f/2)}$ is the coefficient in the corresponding coefficient in the asymptotic expansion of $\cK_{G_{\Theta_q(P)}}$ 
at the left face.
 
Thus for all $v$ orthogonal to $\ker N_q(P),$ we have
\begin{equation*}
	\cK_{\wt\Phi_{\eps,q}}(s, s', u, u', z, z') = s^{-f/2}
	\int e^{i(u-u')\cdot \eta} |\eta|^{-f/2} e^{-s|\eta|} (\Id - \Pi_W) \cK_{(G_{\Theta_q(P)}, -f/2)}(s'|\eta|, z, z', \hat \eta) \; d\eta
\end{equation*}
and $\Phi_{q}v$ is given by the action of $\cK_{\wt\Phi_{\eps,q}}$ on $N_q(P)v.$
So to see that $\Phi_{q}$ acts as a $\sK(L^2(Z_q;E))$-twisted $0$-pseudodifferential operator, we need only 
note that $\wt \Phi_{\eps,q}$ is such an operator. Directly from this expression we see that $J_{10} = -f/2.$ 
On the other hand, the restriction of the kernel of $G_{\Theta_q(P)}$ to the left face has index set $H_{11} = \bbN_0$ 
at the front face and $H_{01}$ at the right face, and so we can take these as the index sets for $\wt\Phi_{\eps,q}.$
Finally, for general $v,$ we have
\begin{equation*}
	\Phi_{q}v = \Phi_{q}(v-\Pi_{\ker N_q(P)}v) + \Pi_{\ker N_q(P)}v
\end{equation*}
and we know that $\Pi_{\ker N_q(P)}$ has index sets $E_{10}$ at the left face and $E_{01}\subseteq H_{01}$ at the right face.
\end{proof}

\section{Parametrix construction}
\subsection{Non-Witt strata} \label{sec:NonWittStrata}
Continuing as above, fix $\eps \in (0,1/2)$ and the boundary operators 
\begin{equation*}
	\bB = (B^1, \ldots, B^k)
\end{equation*}
at the first $k$ singular strata of $\hat X.$ Assume that $Y^{k+1}$ is {\bf non-Witt} for $(\eth_{\dR}, \cD_{\max, \bB}(\eth_{\dR})),$ 
in the sense that $ f = \dim Z$ even and 
\begin{equation}\label{eq:NonWittStratum}
	\cH^{f/2}_{\bB(Z)}(Z) = \ker (\eth_{\dR}^Z, \cD_{\bB(Z)}(\eth_{\dR}^Z))\rest{\text{degree} = f/2}\neq \{ 0 \}.
\end{equation}
As explained in \S\ref{sec:ArbitraryDepthBVP} the union of these null spaces form a flat vector bundle over $Y^{k+1}$.
We now choose a flat sub-bundle
\begin{equation*}
	W \lra Y^{k+1},
\end{equation*}
and let $B^{k+1} = ( \Id- \cP_W, \cP_W)$ be the Cheeger ideal boundary condition associated to $W.$
In this section we show that any $u \in \cD_{\max, (\bB, B^{k+1})}(\eth_{\dR})$ supported in a distinguished neighborhood of $q\in Y^{k+1}$ is in $x^{1/2-}L^2.$\\

Fix a distinguished neighborhood $\cU_q \cong [0,1)_x \times \bbB^h \times Z_q,$ of $q \in Y^{k+1}$, a trivialization 
\begin{equation*}
	\Lambda^*(\Iie T^*X) \cong \Lambda^*( [0,1)_x \times \bbB^h ) \hat\otimes \Lambda^*(\Iie T^*Z_q),
\end{equation*}
and flat trivializations
\begin{equation*}
	\cH^{\mid}_{\bB(Z)}(H^k/Y^k) \cong  \bbB^h \times \cH^{\mid}_{\bB(Z_q)}(Z_q), \quad
	W\rest{\bbB^h} \cong \bbB^h \times W_q.
\end{equation*}

We start by modifying the generalized inverse $G_q$ of the normal operator $N_q(x\eth_{\dR})$ with domain 
$\cD_{\max, N(\bB)}^{\eps}(N_q(x\eth_{\dR}))$ (this is surjective by Corollary \ref{cor:NormalOpSurj}) using the 
projector $\Phi_{q}$ from Lemma \ref{lem:PhiEps}. Recall that
\begin{equation*}
	N_q(\eth_{\dR}) = s^{-1}N_q(x\eth_{\dR}).
\end{equation*}
In \eqref{eq:DefNeth} we defined a domain for $N_q(\eth_{\dR})$ using the first $k$ boundary conditions $\bB$.
We refine this to
\begin{equation*}
	\cD_{N(\bB, B^{k+1})}(N_q(\eth_{\dR}))
	= \{ u \in \cD_{\max, N(\bB)}(N_q(\eth_{\dR})) : \sC u = \Pi_W \sC u \}
\end{equation*}
where $\Pi_W$ is as in \eqref{eq:DefPiW}. Now define
\begin{equation}\label{eq:DefGBq}
	T_{q} = G_{q}s
	- \Phi_{q} G_{q}s 
	= \bar T_q s; 
\end{equation}
the premultiplication by $s$ serves to mediate between $N_q(x\eth_{\dR})$ and $N_q(\eth_{\dR}).$

\begin{proposition}\label{prop:Other}
For any choice of Cheeger ideal boundary conditions $(\bB, B^{k+1}),$ the operator $T_q$ satisfies
\begin{equation*}
\begin{gathered}
	T_q \in \cB(s^{\eps}L^2), \quad \Image(T_q) \subseteq \cD_{\max, N(\bB, B^{k+1})}(N_q(\eth_{\dR})), \quad
	N_q(\eth_{\dR})T_q = \Id \\
	T_q N_q(\eth_{\dR}) = \Id - \Phi_{q} \Mon \cD_{\max,N(\bB)}^{\eps}(N_q(x\eth_{\dR})), \Mand \\
	T_q N_q(\eth_{\dR}) = \Id  \Mon \cD_{\max, N(\bB,B^{k+1})}(N_q(\eth_{\dR})), 
\end{gathered}
\end{equation*}
hence
$\lrpar{ N_q(\eth_{\dR}), \cD_{N(\bB, B^{k+1})}(N_q(\eth_{\dR})) }$
is invertible with inverse $T_q.$

Moreover, the composition 
$\eth_{\dR}^Z \circ \bar T_q$ is a bounded operator on 
$s^{\eps}L^2(T_qY^+ \times Z; \pi^*_q\Lambda^*(\Iie T^*X)).$
\end{proposition}

\begin{proof}
The operator $G_q$ satisfies
\begin{equation*}
\begin{gathered}
	G_q \in 
	\cB(s^{\eps}L^2), \quad \Image(G_q) \subseteq \cD_{\max, N(\bB)}^{\eps}(N_q(x\eth_{\dR})),\\
	\Id - G_qN_q(x\eth_{\dR}) = \Pi_{\ker N_q(x\eth_{\dR})}, \quad \Id - N_q(x\eth_{\dR})G_q = 0
\end{gathered}
\end{equation*}
where $\Pi_{\ker N_q(x\eth_{\dR})}$ is the orthogonal projection onto $\ker \lrpar{N_q(x\eth_{\dR}), \cD_{\max, N(\bB)}^{\eps}(N_q(x\eth_{\dR})))}.$
Now, conjugation by $s$ shows that
\begin{equation*}
	\Id = s^{-1}(\Id) s = s^{-1}(N_q(x\eth_{\dR})G_q)s = N_q(\eth_{\dR})G_qs \quad
	\Mon s^{\eps}L^2,
\end{equation*}
so that $\Image(G_qs) \subseteq \cD_{\max, N(\bB)}(N_q(\eth_{\dR})).$
Since $N_q(\eth_{\dR})\sP_{\eps} =0,$ we have
\begin{equation*}
	N_q(\eth_{\dR})T_q = N_q(\eth_{\dR})G_q s = \Id \quad \Mon s^{\eps}L^2,
	\Mhence \Image(T_q) \subseteq \cD_{\max, N(\bB)}^{\eps}(N_q(\eth_{\dR})).
\end{equation*}
This shows that $\sC \circ T_q$ makes sense, and then directly from the definition of $T_q$ 
\begin{equation*}
	\sC T_q = \Pi_W \sC G_qs \Longrightarrow \Image(T_q) \subseteq \cD_{\max, N(\bB, B^{k+1})}^{\eps}(N_q(\eth_{\dR})).
\end{equation*}

This null space coincides with that of $(N_q(\eth_{\dR}),\cD_{\max, N(\bB)}(N_q(\eth_{\dR}))),$ so 
\begin{equation*}
	\Id - G_qs N_q(\eth_{\dR}) = \Pi_{\ker N_q(\eth_{\dR})} \quad 
	\Mon \quad \cD_{\max, N(\bB)}(N_q(\eth_{\dR})) \cap s^{\eps}L^2.
\end{equation*}
Since $\cD_{\max, N(\bB)}(N_q(\eth_{\dR})) \cap s^{\eps}L^2$ is dense in $\cD_{\max, N(\bB)}(N_q(\eth_{\dR}))$ in the graph norm of $N_q(\eth_{\dR})$ (Theorem \ref{thm:CoreDomain}), and $\Pi_{\ker N_q(\eth_{\dR})}$ extends to a bounded operator on $L^2,$
it follows that 
\begin{equation*}
	\Id - G_qs N_q(\eth_{\dR}) = \Pi_{\ker N_q(\eth_{\dR})} \quad 
	\Mon \quad \cD_{\max, N(\bB)}(N_q(\eth_{\dR})).
\end{equation*}
and so $G_qs$ is the generalized inverse of $(N_q(\eth_{\dR}), \cD_{\max, N(\bB)}(N_q(\eth_{\dR}))).$\\
Moreover, 
\begin{equation*}
	T_q N_q(\eth_{\dR}) = (\Id - \Phi_{q})( \Id - \Pi_{\ker N_q(\eth_{\dR})}) = \Id - \Phi_{q} \Mon \cD_{\max, N(\bB)}(N_q(\eth_{\dR})),
\end{equation*}
since $\Phi_{q}\Pi_{\ker N_q(\eth_{\dR})} = \Pi_{\ker N_q(\eth_{\dR})}.$ Directly from the definition of $\Pi_W$ we have
$(\Id - \Pi_{W})\sC = 0$ on $\cD_{N(\bB,B^{k+1})}(N_q(\eth_{\dR})).$ So indeed
$(N_q(\eth_{\dR}), \cD_{N(\bB,B^{k+1})}(\eth_{\dR}))$ is invertible with inverse $T_q.$\\

To prove that $\eth_{\dR}^Z \circ \bar T_q$ is a bounded operator, first consider $\eth_{\dR}^Z\circ G_q.$ Using that
\begin{equation*}
	\Id = N_q(x\eth_{\dR})G_q = 
	\lrspar{
	\begin{pmatrix}
	\Id & 0 \\ 0 & -\Id
	\end{pmatrix}
	\eth_{\dR}^Z + 
	\begin{pmatrix}
	s\eth_{\dR}^{T_qY} & -s\pa_s +\bN - f\\
	s\pa_s + \bN & -s\eth_{\dR}^{T_qY}
	\end{pmatrix}
	}G_q. 
\end{equation*}
Since $G_q \in \Psi^{-1, \cH}_0(\bbR^+ \times \bbR^h) \otimes \sK(L^2(Z_q;\Lambda^*_q))$, the composition of the 
second summand above with $G_q$ is bounded. It is clear from the definition of $\Phi_{q}$ that its image consists 
of sections of the vertical Hodge bundle and so $\eth_{\dR}^{Z_q} \circ \Phi_{q}=0,$ thus
\begin{equation}\label{eq:Claim3}
	\text{ $\eth_{\dR}^{Z_q} \circ \bar T_q$ is a bounded operator on $s^{\eps}L^2(T_qY^+ \times Z; \pi^*_q\Lambda^*(\Iie T^*X)).$ }
\end{equation}
Note that this means that both $d^Z \circ \bar T_q$ and $\delta^Z \circ \bar T_q$ are bounded operators.
\end{proof}

Let us point out some further properties of $T_q.$ We know from Proposition \ref{prop:InvNormalOp} that 
\begin{equation*}
	G_q \in \Psi^{-1,\cH}_0(\bbR^+ \times \bbR^h)\otimes \sK(L^2(Z_q;\Lambda^*_q)),
\end{equation*}
with $\cH$ given by \eqref{eq:GralInvIndSets}. From Lemma \ref{lem:PhiEps}, 
\begin{equation*}
	\Phi_{q} \in \Psi^{-1,\cJ}_0(\bbR^+ \times \bbR^h)\otimes \sK(L^2(Z_q;\Lambda^*_q)).
\end{equation*}
Thus altogether,
\begin{equation}\label{eq:TqZeroPsi}
	\bar T_q = (\Id - \Phi_{q})G_q \in  \Psi^{-1,\cL}_0(\bbR^+ \times \bbR^h)\otimes \sK(L^2(Z_q;\Lambda^*_q));
\end{equation}
using \cite[Theorem 3.15]{Mazzeo:Edge}, the composition formula for $0$-pseudodifferential operators, the 
index sets are given by
\begin{equation*}
\begin{gathered}
	L_{10} = E_{10} \bar\cup F_{10} \bar\cup E_{10}, \\
	L_{01} = E_{01} \bar\cup F_{01} \bar\cup E_{01} \bar\cup F_{01}, \\
	L_{11} =  \bbN_0 \bar\cup (E_{10} + E_{01} \bar\cup F_{01} + h+1).
\end{gathered}
\end{equation*}
(The operator $T_q = \bar T_q s$ lies in a similar space, but with $L_{01}$ replaced by $ 1+ L_{01},$ because of the 
factor $s$ on the right in the definition of $T_q.$) Note that since 
\begin{equation*}
	N_q(x\eth_{\dR})\bar T_q 
	= \Id 
\end{equation*}
the indicial operator of $x\eth_{\dR}$ must annihilate all of the terms in the expansion of $\cK_{\bar T_q}$ at the left face (see the discussion above \eqref{eq:ImprovedIndex}).\\

We can read off the mapping properties of $\bar T_q$ from \eqref{eq:TqZeroPsi}.
Indeed, in terms of $\eta^+$ and $\eta^-$ from \eqref{eq:DefEta}, we have
\begin{equation*}
	\Re L_{10} = \min(\eta^+, \eta^-) + \eps - \tfrac{f+1}2, \quad
	\Re L_{01} = \min(\eta^+, \eta^-) -(\eps - \tfrac{f+1}2), \quad
	\Re L_{11} = 0
\end{equation*}
and hence $\bar T_q$ defines a bounded map $x^{a'} L^2 \lra x^{b'} L^2$ provided that
\begin{equation*}
	\eps - \min(\eta^+,\eta^-) < a' = b' < \min(\eta^+,\eta^-) + \eps.
\end{equation*}
Using Lemma \ref{lem:SuitableScaling} we can be more explicit:
\begin{equation}\label{eq:MappingL}
\begin{gathered}
	\eps>\tfrac14 \implies \min(\eta^+, \eta^-) = \tfrac12-\eps \implies 2\eps-\tfrac12<a'= b'<\tfrac12 \\
	\eps<\tfrac14 \implies \min(\eta^+, \eta^-)>\eps, \text{ so we can take } 0 \leq a' = b' <2\eps.
\end{gathered}
\end{equation}
The corresponding values for $T_q$ are
\begin{equation*}
	2\eps-\tfrac32<a'= b'<\tfrac12 \Mfor \eps >\tfrac14, \quad
	\Mand  -1 \leq a' = b' <2\eps \Mfor \eps <\tfrac14.
\end{equation*}
$ $

We now restrict to the distinguished neighborhood $\cU_q$. Following the procedure in \S\ref{sec:InvNormal},
we amalgamate the family of operators $\bar T_q$ to an operator $\bar T_B$ on $\hat X$ supported in this neighborhood.
Thus $\bar T_B$ is defined through its integral kernel
\begin{equation}\label{eq:DefKerGB}
	\cK_{\bar T_B}(s, u, x', y', z, z') = \chi \cK_{\bar T_{(y'+x'u)}}(s,u, z, z'),
\end{equation}
where $\chi$ is a smooth cut-off function equal to one in a neighborhood $\wt\cU_q^2.$
We define $T_B$ similarly, 
\begin{equation*}
	\cK_{T_B}(s, u, x', y', z, z') = \cK_{\bar T_B}(s, u, x', y', z, z')(x's).
\end{equation*}
As in \cite[(3.5)]{Mazzeo:Edge}, $T_B$ acts on a section $f$ by
\begin{equation*}
	T_Bf(x,y,z) = \int \cK_{T_B}(\tfrac{x}{x'},\tfrac{y-y'}{x'}, x', y' , z, z') f(x', y', z') \; d\mu(x',y',z')
\end{equation*}
for the appropriate measure $\mu.$ Notice two features here: first, since $x=x's,$
\begin{equation*}
	N_q(T_B\eth_{\dR}) = N_q(\bar T_B x \eth_{\dR})  = \bar T_q N_q(x\eth_{\dR}) = T_qN_q(\eth_{\dR}),
\end{equation*}
and second, 
\begin{equation}\label{eq:KBGoodCoord}
	\cK_{T_B}(\tfrac{x}{x'},\tfrac{y-y'}{x'}, x', y' , z, z') = \chi \cK_{T_y}(\tfrac x{x'}, \tfrac{y-y'}{x'}, z, z')x.
\end{equation}
From \eqref{eq:TqZeroPsi} we see that
\begin{equation*}
	\bar T_B \in \Psi^{-1,\cL}_0([0,1)\times \bbB^h) \otimes \sK(L^2(Z_q;\Lambda^*_q)),
\end{equation*}
so $\bar T_B$ satisfies properties analogous to those of $\bar T_q,$ while $T_B$ satisfy similar mapping properties with
\begin{equation*}
	2\eps-\tfrac32<a'\leq b'<\tfrac12 \Mfor \eps >\tfrac14, \quad
	\Mand  -1 \leq a' \leq b' <2\eps \Mfor \eps <\tfrac14
\end{equation*}
as long as $b'-a'\leq 1,$ since $(x's)$ vanishes to first order at both the front face and the left face. (Note that though the index 
sets vary with $q$, the bounds on their real parts, hence the mapping properties, discussed above are uniform in $q.$)

Let us show that
\begin{equation}\label{Eq:requirement}
	\text{ $\eth_{\dR}^Z \circ \bar T_B$ is a bounded operator on $x^{\eps}L^2(\hat X).$}
\end{equation}
From the way $\bar T_B$ acts, we see that
\begin{equation*}
	\eth_{\dR}^Z(x,y)( \bar T_Bf)(x,y,z) 
	= 
	\int  \eth_{\dR}^Z(x, y)\cK_{\bar T_B}(\tfrac{x}{x'},\tfrac{y-y'}{x'}, x', y' , z, z') f(x',y',z') d\mu(x',y',z')
\end{equation*}
where $\eth_{\dR}^Z(x,y)$ is the vertical de Rham operator for the metric at $(x,y).$  Since the metric $g$ is rigid
(in fact, this argument works for a slightly larger class of metrics), we have $\eth_{\dR}^Z(x,y) = \eth_{\dR}^Z(0,y)
= \eth_{\dR}^{Z_y}$. By \eqref{eq:Claim3}, this composes with $\bar T_y$ to a bounded operator on $L^2.$

We have shown that $\bar T_B$ satisfies the assumption preceeding \eqref{eq:ImprovedIndex}, so if 
$\wt H=(\wt H_{10}, \wt H_{01}, \wt H_{11})$ is the collection of index sets for  $\Id - x\eth_{\dR} \bar T_B$, then 
\begin{equation*}
\begin{gathered}
	\Re \wt H_{10} = 1+ \Re L_{10} = 1+ \min(\eta^+, \eta^-) + \eps - \tfrac{f+1}2, \\
	\Re \wt H_{01} = \Re L_{01} = \min(\eta^+, \eta^-) -(\eps - \tfrac{f+1}2), \\ 
	\Re \wt H_{11} \geq 1.
\end{gathered}
\end{equation*}
These, in turn, imply that $x\eth_{\dR}\bar T_q: x^{a'} L^2 \lra x^{b'} L^2$ is bounded so long as
\begin{equation*}
	\eps - \min(\eta^+,\eta^-) < a' \leq b' < 1+ \min(\eta^+,\eta^-) + \eps, \quad \Mand b'-a'\leq 1.
\end{equation*}
In particular, we have shown that
\begin{equation*}
	x\eth_{\dR} \bar T_B: x^{\eps}L^2(\hat X; \Lambda^*(\Iie T^*X)) \lra x^{1+\eps}L^2(\hat X; \Lambda^*(\Iie T^*X))
\end{equation*}
is bounded, and hence $\eth_{\dR} \bar T_B$ is a bounded operator on $x^{\eps}L^2.$
Essentially the same reasoning shows that $\eth_{\dR} T_B$ is also a bounded operator on $x^{\eps}L^2.$\\
Define $Q = \Id - \wt\chi \bar T_B x\eth_{\dR}$, where $\wt\chi$ is a cut-off function on a neighborhood of $Y$ smaller 
than the one to which we extended $T_B$. From the mapping properties \eqref{eq:MappingL} of $T_B,$ we see that $Q$ defines bounded operators on both $x^{\eps}L^2$ and $L^2.$

\begin{theorem}\label{thm:WitlessDecay}
If $(\hat X,g)$ is a stratified space with singular strata $Y^1, \ldots, Y^{k+1}$, ordered with increasing depth and
rigid $\iie$ metric $g$. Suppose that 
\begin{equation*}
	(B^1, \ldots, B^{k+1})
\end{equation*}
are Cheeger ideal boundary conditions at the first $k+1$ strata, and $Y^{k+1}$ is non-Witt (so \eqref{eq:NonWittStratum} 
holds there). Then 
\begin{equation}\label{eq:DecayCon}
	\cD_{\max, (B^1, \ldots, B^{k+1})}(\eth_{\dR}) \subseteq \bigcap_{\eps \in (0,1/2)} x^{\eps}L^2(X; \Lambda^*(\Iie T^*X)).
\end{equation}
\end{theorem}
\begin{proof}
We may assume that $u$ is supported in a distinguished neighborhood $\cU_q$ of $q \in Y^{k+1}.$
First suppose that $0 < \eps < 1/4.$ 

We have established that 
\begin{equation}\label{eq:xEps}
	\Id - \wt\chi\bar T_B(x\eth_{\dR}) = Q 
\end{equation}
holds on $\cD_{\max, (\bB, B^{k+1})}(\eth_{\dR}) \cap x^{\eps}L^2(X;E).$
By Corollary \ref{cor:IdentityTransfer}, $\cD_{\max, (\bB, B^{k+1})}(\eth_{\dR}) \cap x^{\eps}L^2(X;E)$ is dense in $\cD_{\max, (\bB, B^{k+1})}(\eth_{\dR})$ with respect to the graph norm of $\eth_{\dR}.$ Since $\Id,$ $\wt\chi\bar T_Bx,$ and $Q$ are bounded operators on $L^2$ it follows that \eqref{eq:xEps} continues to hold on $\cD_{\max, (\bB, B^{k+1})}(\eth_{\dR}).$

Thus to show that $u\in x^{\eps}L^2,$ since we already know that $\bar T_B(x\eth_{\dR}u)  \in x^{\eps}L^2,$ it suffices to show that $Qu$ has extra vanishing.
To discuss the normal operator of $Q,$ we recall that 
\begin{equation*}
	Q \in \Psi^{0, \cL}_0([0,1) \times \bbB^h) \otimes \sB( \cD_{\max}(\eth_{\dR}^Z), L^2(Z_q;\Lambda^*_q)).
\end{equation*}
The normal operator of $Q$ at the point $q \in \bbB^h$ is the operator 
\begin{equation*}
	\Phi_{q} \in \Psi_0^{-\infty, \cJ}(\bbR^+\times \bbR^h) \otimes \sK(L^2(Z_q;\Lambda^*_q))
\end{equation*}
from Lemma \ref{lem:PhiEps}.
Let $\Phi_{B}$ be the amalgamation of the family $q \mapsto \Phi_{q}$ as in \eqref{eq:Amalgamation},
so that  we have
\begin{equation*}
	Q - \Phi_{B} \in \Psi^{0, \cL+\{1\}_{11}}_0([0,1) \times \bbB^h) \otimes \sB( \cD_{\max}(\eth_{\dR}^Z), L^2(Z_q;\Lambda^*_q)), 
\quad \cL + \{ 1\}_{11} = (L_{10}, L_{01}, L_{11}+1).
\end{equation*}
This operator satisfies \eqref{eq:MappingL} with $b'-a'\leq 1,$ so in particular
\begin{equation*}
	Q-\Phi_{B} :
	\cD_{\max, (\bB, B^{k+1})}^{0}(\eth_{\dR}) \lra x^{\eps}L^2([0,1) \times \bbB^h; L^2(Z_q;E))=x^{\eps}L^2.
\end{equation*}
Finally, since $N_q(\Phi_{B}) = \sP_{\eps}(\Id - \Pi_{W_q})\sC_q,$ if $\sC u$ is a section of $W$ over $Y^{k+1}$ then 
 $\Phi_{B} u =\cO(x).$

This establishes the extra decay when $\eps \in (0,1/4).$  It was necessary to restrict to these $\eps$ because 
we needed to take $a'=0$ in \eqref{eq:MappingL}.   But we can now repeat the argument starting with any $\eps$,
which allows us to take $a'<1/4$ in \eqref{eq:MappingL}. By inspection, we can thus take any $\eps \in (0,3/8).$ 
Iterating this $n$ times, we establish decay for any $\eps \in (0, 1/2-(1/2)^{n+1})$ and hence, since $n$ is
arbitrary, for any $\eps < 1/2$. This proves  \eqref{eq:DecayCon}.
\end{proof}

\subsection{Witt strata}
Now return to the situation at the beginning of this section. Namely assume that
\begin{equation*}
	\bB = (B^1, \ldots, B^k)
\end{equation*}
are local ideal boundary conditions for the de Rham operator of a rigid $\iie$ metric on $\hat X.$ 
We say that $Y^{k+1}$ is a {\bf Witt stratum} for $(\eth_{\dR}, \cD_{\max, \bB}(\eth_{\dR}))$ either if $f$ is odd, or else if $f$ is even and
\begin{equation}\label{eq:WittStratum}
	\cH^{f/2}_{\bB(Z)}(Z) = \ker (\eth_{\dR}^Z, \cD_{\bB(Z)}(\eth_{\dR}^Z))\rest{\text{degree} = f/2} = \{ 0 \}.
\end{equation}
The argument in \cite[\S5.4]{ALMP} now shows that
\begin{equation*}
	(N_q(P), \cD_{\max, N(\bB)}^{\eps}(N_q(P)) )
\end{equation*}
is invertible for all $\eps \in (0,1).$  We prove here that any $u \in \cD_{\max, \bB}(\eth_{\dR})$ (supported in the 
distinguished neighborhood $\cU_q$) must lie in $x^{1-}L^2.$ 

The construction in this case is simpler than the one above for a non-Witt stratum. Fix $\eps \in (0,1).$ Since 
$(N_q(P), \cD_{\max, N(\bB)}^{\eps}(N_q(P)) )$ is invertible, the construction in \S\ref{sec:InvNormal} produces an inverse
\begin{equation*}
\begin{gathered}
	G_q \in 
	\cB(s^{\eps}L^2), \quad \Image(G_q) \subseteq \cD_{\max, N(\bB)}(N_q(x\eth_{\dR})),\\
	G_qN_q(P) = \Id, \quad N_q(P)G_q = \Id.
\end{gathered}
\end{equation*}
Arguing as in Proposition \ref{prop:Other}, we see that $G_qs$ is an inverse for $N_q(\eth_{\dR}) = s^{-1}N_q(x\eth_{\dR})$ with domain $	\cD_{N(\bB)}(N_q(\eth_{\dR}))$
and that $\eth_{\dR}^Z \circ G_qs$ is a bounded operator on $s^{\eps}L^2.$

Continuing on, define the integral kernel 
\begin{equation*}
	\cK_B(s, u, x', y', z, z') = \cK_{G_{(y'+x'u)}}(s,u, z, z')x',
\end{equation*}
of $G_B$ by amalgamating the family of operators $G_q$. This lies in $\Psi^{-1,\cH}_0([0,1)\times \bbB^h) \otimes 
\sK(L^2(Z_q;\Lambda^*_q)))$; the index sets are such that $G_B$ is a bounded map on $x^{\eps}L^2(X;\Lambda^*(\Iie T^*X))$ with image in $\cD_{\max, \bB}(\eth_{\dR}) \cap x^{\eps}L^2,$ and $Q = \Id - G_B\eth_{\dR}$ extends from $x^{\eps}L^2$ to a bounded operator on $L^2.$ We can thus
appeal to the density of $\cD_{\max,\bB}(\eth_{\dR})\cap x^{\eps}L^2,$ and so obtain the analogue of Theorem \ref{thm:WitlessDecay}.

\begin{theorem}\label{thm:WittyDecay}
Let $(\hat X, g)$ and $\bB$ be as in Theorem \ref{thm:WitlessDecay}, but suppose that \eqref{eq:WittStratum} holds
at $Y^{k+1}$, so this is a Witt stratum. Then
\begin{equation*}
	\cD_{\max, (B^1, \ldots, B^{k})}(\eth_{\dR}) \subseteq \bigcap_{\eps \in (0,1)} x^{\eps}L^2(X; \Lambda^*(\Iie T^*X)).
\end{equation*}
\end{theorem}

\subsection{The de Rham operator is essentially injective}

We have shown how to construct a (left) parametrix at all singular strata. This construction on spaces of depth $k+1$ 
relies on inductive information about the de Rham operator on singular spaces of depth less than or equal to $k$, and
in particular that this operator with Cheeger ideal boundary conditions from a mezzoperversity is self-adjoint.
We now use these parametrices to prove that this de Rham operator is essentially injective, i.e., has closed range 
and finite dimensional null space.  Then, once we have proved Theorem \ref{thm:SelfDual}, which asserts 
that the de Rham operator on spaces of depth $k+1$ is self-adjoint, then this essential injectivity proves that it
must actually be Fredholm. 
\begin{theorem}\label{Thm:MainHodgeThm}
Consider $(\hat X,g)$ as before, where $g$ is suitably scaled, and $\bB = (B^1, \ldots, B^{k+1})$ a set of
Cheeger ideal boundary conditions. Then $\eth_{\dR}$ with domain $\cD_{\bB}(\eth_{\dR})$ is closed, 
essentially injective and has compact resolvent.  Indeed if $\rho$ is a total boundary defining function 
on $\wt X,$ the resolution of $\hat X,$ then 
\begin{equation}\label{eq:CompactInclusionDomain}
	\cD_{\bB}(\eth_{\dR}) \subseteq H^1_{\loc}(X;\Lambda^*(\Iie T^*X)) \cap \bigcap_{\eps \in (0,1/2)} \rho^{\eps}L^2(X;\Lambda^*(\Iie T^*X)). 
\end{equation}
In particular, the de Rham operator satisfies {\em Assumptions} \ref{AssIndFamily} and \ref{AssCmpDomain} at depth $k+1$. 
\end{theorem}
\begin{proof}
Assume that this result is true for all spaces of depth less than or equal to $k$. If $u \in \cD_{\bB}(\eth_{\dR})$ then 
$\chi u \in \cD_{\bB}(\eth_{\dR})$ for all $\chi$ 
indicator functions supported in the distinguished neighborhood $\cU_q$ of $q \in Y^{k+1}$, and using the inductive hypothesis and Theorems \ref{thm:WitlessDecay} and \ref{thm:WittyDecay},
\begin{equation*}
	\chi u \in \bigcap_{\eps \in (0,1/2)} \rho^{\eps}L^2(X; \Lambda^*(\Iie T^*X)).
\end{equation*}
Since $q$ is arbitrary, $u$ itself lies in this space. Since this space is compactly included in $L^2(X; \Lambda^*(\Iie T^*X)),$ 
this shows that $(\eth_{\dR}, \cD_{\bB}(\eth_{\dR}))$ is essentially injective and has compact resolvent. 
\end{proof}

\section{$L^2$ cohomology and Kodaira decomposition} \label{sec:L2Coho}
Let $(\hat X,g)$ be a stratified pseudomanifold endowed with a rigid $\iie$ metric. In this section we relate the domains of the 
de Rham operator to domains of the exterior derivative. Useful consequences include the self-adjointness of the de Rham 
operator, a Kodaira decomposition for $L^2$ differential forms, and an identification of the Hodge cohomology associated 
to a mezzoperversity with the associated de Rham cohomology.\\

The exterior derivative $d$ and its formal adjoint $\delta$ have two canonical closed extensions from the smooth forms of compact support on $X=\hat X^{\reg}.$
Namely,
\begin{equation*}
\begin{gathered}
	\cD_{\min}(d) =
	\{ u \in L^2(\hat X; \Lambda^* (\Iie T^*X)) : \exists u_n \in  \CIc(X; \Lambda^* (\Iie T^*X)) \Mst u_n \to u \Mand
	du_n \text{ is Cauchy} \}, \\
	\cD_{\max}(d) =
	\{ u \in L^2(\hat X; \Lambda^* (\Iie T^*X)) : du \in L^2(\hat X; \Lambda^* (\Iie T^*X)) \}.
\end{gathered}
\end{equation*}
In the first case, `Cauchy' means Cauchy in $L^2,$ while in the second $du$ is computed distributionally. We define
in the same way $\cD_{\min}(\delta)$ and $\cD_{\max}(\delta).$
It is clear from these definitions that
\begin{equation*}
	(d, \cD_{\min}(d))^* = (\delta, \cD_{\max}(\delta)), \quad (\delta, \cD_{\max}(\delta))^* = (d, \cD_{\min}(d))
\end{equation*}
and similarly,  with $\cD_{\min}$ and $\cD_{\max}$ exchanged. \\

To identify the adjoints of domains (for $d,$ $\delta$ or $\eth_{\dR}$) intermediate between $\cD_{\min}$ and $\cD_{\max},$
introduce the boundary pairings
\begin{equation*}
\begin{gathered}
	[\cdot, \cdot]_d: \cD_{\max}(d) \times \cD_{\max}(\delta) \lra \bbC, \quad [u,v]_d = \ang{du,v} - \ang{u,\delta v} \\
	[\cdot,\cdot]_{\eth_{\dR}}: \cD_{\max}(\eth_{\dR}) \times \cD_{\max}(\eth_{\dR}) \lra \bbC, \quad
	[u,v]_{\eth_{\dR}}  = \ang{\eth_{\dR}u, v} - \ang{u, \eth_{\dR}v}
\end{gathered}
\end{equation*}
where $\ang{\cdot, \cdot}$ is the $L^2$ pairing. These are closely related: in fact, $[\cdot,\cdot]_d$ can be expressed in terms of 
$[\cdot,\cdot]_{\eth_{\dR}}$ as follows.
Let 
\begin{equation}\label{eq:DefuDelta}
	u_{\delta} = \text{ orthogonal projection of $u$ onto $\bar{\delta(\cD_{\max}(\delta))}.$}
\end{equation}
Note that since $\bar{\delta(\cD_{\max}(\delta))}^{\perp} = \ker(d, \cD_{\min}(d)),$ we have
\begin{equation*}
	\ang{du,v} = \ang{ du_{\delta},v} \Mand
	\ang{u, \delta v} = \ang{ u_{\delta}, \delta v} \Mforall v \in \cD_{\max}(\delta).
\end{equation*}
Similarly, if $v_d$ denotes the orthogonal projection of $v$ onto $\bar{ d(\cD_{\max}(d)) },$ 
we have
\begin{equation*}
	\ang{du,v} = \ang{ du,v_d } \Mand
	\ang{u, \delta v} = \ang{ u, \delta v_d} \Mforall u \in \cD_{\max}(d).
\end{equation*}
Thus it is always true that 
\begin{equation*}
	[u,v]_d = [u_{\delta}, v_d]_d.
\end{equation*}
Moreover, since $\bar{\delta(\cD_{\max}(\delta))} \subseteq \ker ( \delta, \cD_{\max}(\delta) )$ we have
\begin{equation*}
	u_\delta \in \cD_{\max}(d+\delta), \quad (d+\delta) u_{\delta} = d u_{\delta}
\end{equation*}
and similarly $\bar{d(\cD_{\max}(d))} \subseteq \ker (d, \cD_{\max}(d))$ so we also have
\begin{equation*}
	v_d \in \cD_{\max}(d+\delta), \quad (d+\delta)v_d = \delta v_d.
\end{equation*}
Thus altogether,
\begin{equation}\label{eq:Reduction}
	[u,v]_d = [u_{\delta}, v_d]_d = [u_{\delta}, v_d]_{\eth_{\dR}}.
\end{equation}
\medskip

\subsection{Depth one} 

Let us start by considering the setting of $\hat X$ with a single singular stratum $Y.$
For any $w \in \cD_{\max}(\eth_{\dR})$ denote the leading term in its expansion from Lemma \ref{lem:ExpansionExistence} and Corollary \ref{cor:DeRhamExp} by $\alpha(w) + dx \wedge \beta(w).$ Recall from Theorem \ref{thm:CoreDomain} the core subdomain 
$\cD_{\max}^{\reg}(\eth_{\dR})$ with the property that $\alpha(w)$ and $\beta(w)$ are smooth instead of just distributional sections, and that the remainder is in $xL^2.$
\begin{lemma}\label{lem:StokesThmEdge}
If $u,v \in \cD_{\max}(\eth_{\dR})$ and one of them is in $\cD_{\max}^{\reg}(\eth_{\dR})$ then from Stokes' theorem
\begin{equation*}
	[u,v]_d = \ang{\alpha(u), \beta(v)}_Y, \quad
	[u,v]_{\eth_{\dR}} = 
	\ang{\alpha(u),\beta(v)}_Y - \ang{\beta(u), \alpha(v)}_Y
\end{equation*}
where on $Y$ we have the pairing between smooth forms and currents with coefficients in the vertical Hodge bundle, induced by the $L^2$-pairing.
\end{lemma}
\begin{proof}
We always have
\begin{equation*}
	[u,v]_d = \lim_{\eps \to 0} \lrpar{\ang{du,v}_{x\geq \eps} - \ang{u,\delta v}_{x\geq \eps} }
	= \lim_{\eps \to 0} \ang{u,v}_{x=\eps}.
\end{equation*}
If both $u, v$ are in $\cD_{\max}^{\reg}(\eth_{\dR})$ then 
$\ang{u,v}_{x=\eps} = \ang{\alpha(u), \beta(v)}_Y + \cO(\eps),$ and so $[u,v]_d = \ang{\alpha(u), \beta(v)}_Y.$
If only one of them, say $u$ is in $\cD_{\max}^{\reg}(\eth_{\dR}),$ we can find $v_n \in \cD_{\max}^{\reg}(\eth_{\dR})$
such that $v_n \to v$ in the graph norm. By continuity of the coefficients of the asymptotic expansion in Lemma \ref{lem:ExpansionExistence}, we have $\beta(v_n) \to \beta(v)$ as distributions, so
\begin{equation*}
	[u,v]_d = \lim [u,v_n]_d = \lim \ang{\alpha(u), \beta(v_n)}_Y = \ang{\alpha(u), \beta(v)}_Y.
\end{equation*}
The pairing $[\cdot,\cdot]_{\eth_{\dR}}$ is analogous.
\end{proof}

This lemma shows the usefulness of having elements in $\cD_{\max}^{\reg}(\eth_{\dR}),$ so we define analogous domains for $d$ and $\delta.$
The domain
\begin{equation*}
	\cD_{\max}^{\reg}(d) = \{ u \in \cD_{\max}(d) : u_{\delta} \in \cD_{\max}^{\reg}(\eth_{\dR}) \}
\end{equation*}
is a core domain for $\cD_{\max}(d)$ by the same argument that showed that $\cD_{\max}^{\reg}(\eth_{\dR})$ is a core domain for $\cD_{\max}(\eth_{\dR}).$
Similarly we define
\begin{equation*}
	\cD_{\max}^{\reg}(\delta) = \{ v \in \cD_{\max}(\delta) : v_d \in \cD_{\max}^{\reg}(\eth_{\dR}) \}
\end{equation*}
and point out that this is a core domain for $\cD_{\max}(\delta).$

\begin{lemma}\label{lem:PrescribedForms}
Let $\cU_q$ be a distinguished neighborhood of $q \in Y$ over which we have trivializations
\begin{equation*}
	\cU_q \cong [0,1)_x \times \bbB \times Z_q, \quad
	\bigcup_{y \in \cU_q \cap Y} \cH^{\mid}(Z_y) = \bbB \times \cH^{\mid}(Z_q),
\end{equation*}
where we recall that $\cH^{\mid}(Z_q)$ refers to the $L^2$ harmonic forms on $Z_q$ of degree equal to half of the dimension of $Z_q.$
For differential forms $u \in L^2(X; \Lambda^*(\Iie T^*X))$ supported in $\cU_q$ we have
\begin{equation*}
\begin{gathered}
	u \in \cD_{\min}(d) \iff 
	u \in \cD_{\max}(d) \Mand \alpha(u_{\delta})=0\\
	u \in \cD_{\min}(\delta) \iff 
	u \in \cD_{\max}(\delta) \Mand \beta(u_{d})=0.
\end{gathered}
\end{equation*}

Moreover, given $\eta \in \CIc(\bbB; \Lambda^*(\Iie T^*Y) \otimes \cH^{\mid}(Z_q)),$
\begin{itemize}
\item [i)] there exists $u(\eta) \in \cD_{\min}(d) \cap \cD_{\min}(\delta) \cap \cA_{\phg}^*$ such that $\delta u(\eta) \in \cD_{\max}^{\reg}(\eth_{\dR}),$
\begin{equation*}
	\alpha(\delta u(\eta)) = \eta,
\end{equation*}
\item [ii)] there exists $v(\eta) \in \cD_{\min}(d) \cap \cD_{\min}(\delta)\cap \cA_{\phg}^*$ such that $d v(\eta) \in \cD_{\max}^{\reg}(\eth_{\dR}),$
\begin{equation*}
	\beta(d v(\eta)) = \eta.
\end{equation*}
\item [iii)] For any $w \in \cD_{\max}^{\reg}(\eth_{\dR})$ compactly supported in $\cU_q$ with $\alpha(w) = a,$ $\beta(w) = b,$  there exists $\omega(a,b) \in \cD_{\max}^{\reg}(\eth_{\dR}) \cap \cD_{\max}(d) \cap \cD_{\max}(\delta)\cap \cA_{\phg}^*$ such that
\begin{equation*}
	\alpha(\omega(a,b)) = a, \quad
	\beta(\omega(a,b)) = b.
\end{equation*}
\end{itemize}

\end{lemma}

\begin{proof}
Let $\phi$ be a smooth function identically equal to one on the support of $\eta,$ and vanishing outside a compact subset of $\cU_q.$ \\
Let
\begin{equation*}
	v(\eta) = 
	\begin{cases}
	(\log x) \phi \eta & \Mif f=2 \\
	(x^{-f/2+1}/(-f/2+1)) \phi \eta & \text{else}
	\end{cases}
\end{equation*}
and let $u(\eta) = -dx \wedge v(\eta).$
Since $u$ and $v$ are smooth and in $xL^2(X; \Lambda^*(\Iie T^*X),$  we clearly have
\begin{equation*}
	u, du, \delta u, v, dv, \delta v \in \cD_{\max}^{\reg}(d) \cap \cD_{\max}^{\reg}(\delta) \cap \cA_{\phg}^*.
\end{equation*}
Also $\alpha(\delta u(\eta)) = \eta$ and $\beta(dv(\eta)) = \eta.$\\

Now we can identify 
\begin{equation*}
\begin{gathered}
	\cD_{\min}(d) = \cD_{\max}^{\reg}(\delta)^* 
	= \{ u \in \cD_{\max}(d) : [u,v]_d = 0 \Mforall v \in \cD_{\max}^{\reg}(\delta) \} \\
	= \{ u \in \cD_{\max}(d) : \ang{\alpha(u{\delta}), \beta(v_d)}_Y = 0 \Mforall v \in \cD_{\max}^{\reg}(\delta) \}
\end{gathered}
\end{equation*}
and we have just shown that $\beta(v_d)$ can be prescribed arbitrarily, so $u \in \cD_{\min}(d)$ precisely when $\alpha(u_{\delta})=0.$
Similarly, from $\cD_{\min}(\delta) = \cD_{\max}^{\reg}(d)^*$ we can characterize $\cD_{\min}(\delta)$ as required.
Moreover we can now recognize that $u(\eta), v(\eta) \in \cD_{\min}(d) \cap \cD_{\min}(\delta),$ so we have established ($i$) and ($ii$).\\

For ($iii$), let $\omega(a,b) = x^{-f/2}\phi (a + dx \wedge b).$
In $\cU_q,$ we know from \eqref{eq:dNearBdy}, \eqref{eq:deltaNearBdy} that we can write
\begin{equation*}
\begin{gathered}
	d = 
	\begin{pmatrix}
	\tfrac1x d^Z + d_{\cT}^{0,1} + x R & 0 \\
	\pa_x + \tfrac1x \bN & -(\tfrac1x d^Z + d_{\cT}^{0,1} + x R)
	\end{pmatrix}, \\
	\delta = 
	\begin{pmatrix}
	\tfrac1x \delta^Z + (d_{\cT}^{0,1})^* + (x R)^* & -\pa_x + \tfrac1x(\bN -f) \\
	0 & -(\tfrac1x \delta^Z + (d_{\cT}^{0,1})^* + (x R)^*)
	\end{pmatrix}
\end{gathered}
\end{equation*}
and we know from Lemma \ref{lem:ExpansionExistence} that $(a + dx\wedge b)$ is in the null space of
\begin{equation*}
	I(x\eth_{\dR};-f/2)
	=\begin{pmatrix}
	d^Z + \delta^Z  & f/2 + (\bN -f) \\
	-f/2+  \bN & -(d^Z + \delta^Z)
	\end{pmatrix}.
\end{equation*}
It follows from the Kodaira decomposition on $Z_q$ and the explicit expressions for $d$ and $\delta$ that 
\begin{equation*}
	\omega(a,b) \in \cD_{\max}(d) \cap \cD_{\max}(\delta). 
\end{equation*}
\end{proof}

Since 
\begin{equation*}
	\cD_{\min}(d) \cap \cD_{\max}(\delta) \Mand
	\cD_{\max}(d) \cap \cD_{\min}(\delta)
\end{equation*}
are both closed domains for $\eth_{\dR},$  they both contain $\cD_{\min}(\eth_{\dR}),$ and hence so does their intersection
\begin{equation}\label{InclusionDmin}
	\cD_{\min}(\eth_{\dR}) \subseteq \cD_{\min}(d) \cap \cD_{\min}(\delta).
\end{equation}
On the other hand, in the pairing $[u,v]_{\eth_{\dR}}$ this lemma shows that we can prescribe $\alpha(u)$ and $\beta(u)$ arbitrarily and so 
\begin{equation*}
	\cD_{\min}(\eth_{\dR}) = (\cD_{\max}^{\reg}(\eth_{\dR}))^* = \{ v \in \cD_{\max}(\eth_{\dR}) : \alpha(v) = \beta(v) = 0 \}
	= \cD_{\min}(d) \cap \cD_{\min}(\delta).
\end{equation*}
In particular, if $w \in \cD_{\max}^{\reg}(\eth_{\dR})$ then we have
\begin{equation}\label{eq:DmaxregDecomp}
	w - \omega( \alpha(w), \beta(w)) \in \cD_{\min}(d) \cap \cD_{\min}(\delta), \Mhence
	w \in \cD_{\max}(d) \cap \cD_{\max}(\delta).
\end{equation}
$ $

Another useful consequence of this lemma is that whenever $u \in \cD_{\max}(\eth_{\dR}),$
\begin{equation}\label{eq:SameExpansion}
	\alpha(u_{\delta}) = \alpha(u), \quad
	\beta( u_d ) = \beta( u).
\end{equation}
Indeed, we have
\begin{equation*}
	\ang{\alpha(u), \beta(v)} = [u,v]_d = [u_{\delta},v]_d = \ang{\alpha(u_{\delta}), \beta(v)}
\end{equation*}
for any $v \in \cD_{\max}^{\reg}(\eth_{\dR})$ and from the Lemma we can prescribe $\beta(v)$ arbitrarily.
This shows that $\alpha(u_{\delta}) = \alpha(u)$ and a similar argument shows $\beta( u_d ) = \beta( u).$\\

Returning to \eqref{eq:Reduction}, we see that
whenever $u \in \cD_{\max}(d),$ $v \in \cD_{\max}(\delta)$ and either $u_{\delta}$ or $v_d$ is in $\cD_{\max}^{\reg}(\eth_{\dR}),$
\begin{equation*}
	\ang{\alpha(u_\delta), \beta(v_d)} 
	= \ang{\alpha(u_\delta),\beta(v_d)}_Y - \ang{\beta(u_{\delta}), \alpha(v_d)}_Y
\end{equation*}
and so
\begin{equation}\label{eq:Vanishing}
	 \ang{\beta(u_{\delta}), \alpha(v_d)}_Y =0.
\end{equation}
Let us see that this is a reflection of the flatness of the bundle of vertical harmonic forms over $Y.$ 

\begin{lemma}\label{lem:AlphaNabla}
For $u \in \cD_{\max}(d),$ $v \in \cD_{\max}(\delta),$
\begin{equation*}
	\alpha( (du)_{\delta} ) = \nabla^{\cH} \alpha(u_{\delta}), \quad
	\beta( (\delta v)_d ) = -(\nabla^{\cH})^* \beta(v_d)
\end{equation*}
where $\nabla^{\cH}$ is the flat connection on the bundle of vertical harmonic forms discussed in \S\ref{sec:SimpleEdgeBC}.
\end{lemma}

\begin{proof}
Any $u \in \cD_{\max}(\eth_{\dR})$ has a partial asymptotic expansion 
\begin{equation*}
	x^{-f/2}(\alpha(u) + dx\wedge \beta(u)) + xH^{-1}_e.
\end{equation*}
If $u$ is polyhomogeneous then it has a full asymptotic expansion
\begin{equation*}
	u \sim 
	x^{-f/2}(\alpha(u) + dx\wedge \beta(u))
	+ \sum_{\zeta \in \cE} a_{\zeta} x^{\zeta}
\end{equation*}
for some index set $\cE.$ From the form of the partial asymptotic expansion above, we must have $\Re\cE\geq (1-f)/2.$

Consider a differential form $u$ such that $u, du \in \cD_{\max}(d+\delta) \cap \cA_{\text{phg}}^*.$
Applying $d$ to the expansion of $u$ we have
\begin{equation*}
	du \sim x^{-f/2}(d_{\cT}^{0,1}\alpha(u) - dx \wedge d_{\cT}^{0,1}\beta(u) ) 
	+ \sum_{\zeta \in \cE} ((d^Za_{\zeta} + \zeta dx \wedge a_{\zeta}) x^{\zeta-1} + b_{\zeta} x^{\zeta})
\end{equation*}
on the other hand, we have 
\begin{equation*}
	du \sim 
	x^{-f/2}(\alpha(du) + dx\wedge \beta(du))
	+ \sum_{\zeta \in \cE'} c_{\zeta} x^{\zeta}
\end{equation*}
with $\Re\cE'\geq (1-f)/2.$ Comparing coefficients we must have
\begin{equation*}
	\alpha(du) = d_{\cT}^{0,1}\alpha(u) + d^Z\eta(u)
\end{equation*}
for some form $\eta(u).$ From the Kodaira decomposition, 
\begin{equation*}
	\ker (d_Z, \cD_{\max}(d_Z)) \cap \ker (\delta_Z, \cD_{\min}(\delta_Z))
	\oplus
	d_Z(\cD_{\max}(d_Z))
	\oplus
	\delta_Z(\cD_{\min}(\delta_Z))
\end{equation*}
we see that projecting $d_{\cT}^{0,1}\alpha(u) \in \ker(d_Z, \cD_{\max}(d_Z))$ onto the harmonic forms comes down to subtracting an exact form, say $d_{\cT}^{0,1}\alpha(u) = \nabla^{\cH}\alpha(u) + d^Z\eta'(u).$
Thus we have
\begin{equation*}
	\alpha(du) - \nabla^{\cH}\alpha(u) = d^Z(\eta(u)-\eta'(u)),
\end{equation*}
however, note that the left hand side is harmonic while the right hand side is exact. Another appeal to the Kodaira decomposition shows that both sides must vanish, so we have shown
\begin{equation*}
	\alpha(du) = \nabla^{\cH}\alpha(u).
\end{equation*}
Similar reasoning shows that whenever $v, \delta v \in \cD_{\max}(\eth_{\dR}) \cap \cA_{\text{phg}}^*,$ we have
\begin{equation*}
	\beta(\delta v) = - (\nabla^{\cH})^*\beta(v).
\end{equation*}

Now for any $u \in \cD_{\max}(d),$ let us take $v$ such that $v, \delta v \in \cD_{\max}(d+\delta) \cap \cA_{\text{phg}}^*,$ and consider the boundary pairings
\begin{equation*}
	[du,v]_d = -\ang{du, \delta v} = -[u, \delta v]_d.
\end{equation*}
On the one hand, we can write
\begin{equation*}
	[du,v]_d = \ang{\alpha( (du)_{\delta} ), \beta(v) }_Y,
\end{equation*}
while on the other we have
\begin{equation*}
	-[u, \delta v]_d
	= -\ang{ \alpha(u_{\delta}), \beta( \delta v) }_Y
	= \ang{ \alpha( u_{\delta}), (\nabla^{\cH})^*\beta(v) }_Y
	= \ang{ \nabla^{\cH} \alpha(u_{\delta}), \beta(v) }_Y.
\end{equation*}
Since from Lemma \ref{lem:PrescribedForms} $\beta(v)$ can be arbitrarily prescribed, we conclude that
\begin{equation*}
	\alpha( (du)_{\delta}) = \nabla^{\cH}\alpha(u_{\delta}).
\end{equation*}
Similarly we see that $\beta( (\delta v)_d ) = -(\nabla^{\cH})^* \beta( v_d).$
\end{proof}

Thus in \eqref{eq:Vanishing}, $\beta(u_{\delta})$ is in the closure of the image of $(\nabla^{\cH})^*$ while $\alpha(v_d)$ is in the closure of the image of $\nabla^{\cH}.$
Since the connection is flat these images are orthogonal, which explains \eqref{eq:Vanishing}.\\

Next consider Cheeger ideal boundary conditions, $\bB,$ corresponding to the flat bundle $W \lra Y.$
We define
\begin{equation*}
\begin{gathered}
	\cD_{\bB}^{\reg}(d) = \{ u \in \cD_{\max}(d) : u_{\delta} \in \cD_{\bB}^{\reg}(\eth_{\dR}) \}\\
	\cD_{\bB}^{\reg}(\delta) = \{ u \in \cD_{\max}(\delta) : u_d \in \cD_{\bB}^{\reg}(\eth_{\dR}) \}
\end{gathered}
\end{equation*}
and then set $\cD_{\bB}(d)$ and $\cD_{\bB}(\delta)$ equal to the closures in the respective graph norms.
The discussion above allows us to compute the adjoints of these domains and then their closures.

First, $v \in \cD_{\bB}(d)^*$ is equivalent to $[u,v]_d=0$ for all $u \in \cD_{\bB}^{\reg}(d),$ i.e., 
\begin{equation*}
	\ang{ \alpha(u_{\delta}), \beta(v_d) }_Y =0 \Mforall u \in \cD_{\bB}^{\reg}(d).
\end{equation*}
From Lemma \ref{lem:PrescribedForms}, we can prescribe $\alpha(u_{\delta})$ arbitrarily among smooth forms on $Y$ with coefficients in $W.$
Thus
\begin{equation*}
	\cD_{\bB}(d)^* = \{ v \in \cD_{\max}(\delta) : \beta(v_d) \text{ is a current on $Y$ with coefficients in $W^{\perp}$} \}.
\end{equation*}
This domain has a core given by
\begin{equation*}
	\cE_{\bB}(\delta) = \{ v \in \cD_{\max}(\delta) : \beta(v_d) \text{ is a smooth form on $Y$ with coefficients in $W^{\perp}$} \}.
\end{equation*}

It is easy to see that a form $u \in \cD_{\max}(d)$ will be in $\cD_{\bB}(d) = \cD_{\bB}^{\reg}(d)^{**} = \cE_{\bB}(\delta)^*$ precisely when 
\begin{equation*}
	[u,v]_d=0 \Mforall v \in \cE_{\bB}(\delta).
\end{equation*}
This pairing is equal to $\ang{\alpha(u_\delta), \beta(v_d)}_Y,$ and from Lemma \ref{lem:PrescribedForms} we can prescribe $\beta(v_d)$ arbitrarily,
so we conclude that
\begin{equation*}
	\cD_{\bB}(d) = \{ u \in \cD_{\max}(d): \alpha(u_\delta)  \text{ is a current on $Y$ with coefficients in $W$} \}.
\end{equation*}
A similar computation lets us identify $\cD_{\bB}(\delta)$ and we find $\cD_{\bB}(\delta) = \cD_{\bB}(d)^*.$\\

Since $(d, \cD_{\bB}(d))$ is a closed operator,
\begin{multline*}
	\cD_{\bB}(d) \cap \cD_{\bB}(d)^*
	= \{ u \in \cD_{\max}(d) \cap \cD_{\max}(\delta) : 
	\alpha(u_{\delta}) \text{ is a current on $Y$ with coefficients in $W,$ and }\\
 	\beta(u_d) \text{ is a current on $Y$ with coefficients in $W^{\perp}$} \}
\end{multline*}
is a closed domain for $\eth_{\dR}.$
Notice that since $u \in \cD_{\max}(d) \cap \cD_{\max}(\delta)$ is in $\cD_{\max}(\eth_{\dR})$ we can use observation \eqref{eq:SameExpansion} to replace $\alpha(u_{\delta})$ and $\beta(u_d)$ by $\alpha(u)$ and $\beta(u).$ Thus directly from its description we have
\begin{equation*}
	\cD_{\bB}(d) \cap \cD_{\bB}(d)^*
	= \cD_{\bB}(\eth_{\dR}) \cap \cD_{\max}(d) \cap \cD_{\max}(\delta).
\end{equation*}
Since this domain clearly contains
\begin{equation*}
	\cD_{\bB}^{\reg}(\eth_{\dR})
	= \cD_{\bB}^{\reg}(\eth_{\dR}) \cap \cD_{\max}(d) \cap \cD_{\max}(\delta)
\end{equation*}
and the only closed domain for $\eth_{\dR}$ containing $\cD_{\bB}^{\reg}(\eth_{\dR})$ is $\cD_{\bB}(\eth_{\dR}),$ we have the first part of the following proposition:

\begin{proposition} \label{prop:DeRhamCohoSimpleEdge}
Let $\hat X$ be a stratified pseudomanifold with a single singular stratum  and $\bB$ a Cheeger ideal boundary condition corresponding to a mezzoperversity at $Y.$
We have 
\begin{equation*}
	\cD_{\bB}(\eth_{\dR}) = \cD_{\bB}(d) \cap \cD_{\bB}(d)^*,
\end{equation*}
so in particular $(\eth_{\dR}, \cD_{\bB}(\eth_{\dR}))$ is self-adjoint.
Together with Theorem \ref{Thm:MainHodgeThm}, this shows that it is also Fredholm.\\

The differential forms in $\cD_{\bB}(d)$ together with $d$ form a Hilbert complex, whose cohomology we denote
\begin{equation*}
	\tH^*_{\bB}(\hat X).
\end{equation*}
Since $(\eth_{\dR}, \cD_{\bB}(\eth_{\dR}))$ is Fredholm, this is a Fredholm complex and we have
\begin{equation}\label{Kodaira1}
\begin{gathered}
	L^2(X;\Lambda^* (\Iie T^*X) )
	= \cH^*_{\bB}(\hat X)  \oplus d(\cD_{\bB}(\eth_{\dR})) \oplus \delta(\cD_{\bB}(\eth_{\dR})) \\
	\cH^*_{\bB}(\hat X) \cong \tH^*_{\bB}(\hat X).
\end{gathered}
\end{equation}

Moreover, the domain $\cD_{\bB}(d)$ and the cohomology $\tH^*_{\bB}(\hat X)$ are independent of the choice of metric (within suitably scaled rigid $\iie$ metrics).
\end{proposition}

\begin{proof}
The discussion above the proposition shows that $\cD_{\bB}(\eth_{\dR}) = \cD_{\bB}(d) \cap \cD_{\bB}(d)^*.$
We know that $(d, \cD_{\bB}(d))$ is a closed operator, so it will form a Hilbert complex iff 
\begin{equation*}
	d(\cD_{\bB}(d)) \subseteq \cD_{\bB}(d)
\end{equation*}
which, from the definition of $\cD_{\bB}(d),$ is equivalent to 
\begin{equation*}
\begin{gathered}
	\alpha(u_{\delta})  \text{ is a current on $Y$ with coefficients in $W$} \\
	\implies \alpha( (du)_{\delta} ) = \nabla^{\cH}( \alpha(u_{\delta}) )
	\text{ is a current on $Y$ with coefficients in $W,$}
\end{gathered}
\end{equation*}
which follows from the flatness of $W.$
The well-known general theory of Hilbert complexes of Br\"uning-Lesch \cite{Bruning-Lesch} implies that this complex is a Fredholm complex satisfying a Kodaira decomposition and a Hodge theorem \eqref{Kodaira1}.\\

For metric independence, let us start with $u \in \cD_{\bB}^{\reg}(d)$ compactly supported in a distinguished neighborhood $\cU_q$ of $q \in Y.$
If we denote $u_0 = u - u_{\delta},$ then we know that $u_0 \in \ker (d, \cD_{\min}(d))$ and that $u_{\delta} \in \cD_{\max}^{\reg}(\eth_{\dR}).$
So from \eqref{eq:DmaxregDecomp}, $u - \omega( \alpha(u_{\delta}), \beta(u_{\delta}) ) \in \cD_{\min}(d).$

Now if $g'$ is another suitably scaled rigid $\iie$ metrics on $X,$ then it is quasi-isometric to $g$ since they both define metrics on $\Iie T^*X$ over all of $\wt X$ which is compact.
Thus the sets $L^2_g(\hat X, \Lambda^*(\Iie T^*X))$ and $L^2_{g'}(\hat X, \Lambda^*(\Iie T^*X))$ coincide and the $L^2$-norms are equivalent. Directly from the definitions we see that
\begin{equation*}
	\cD_{\max}(d;g) = \cD_{\max}(d;g') 
	\Mand 
	\cD_{\min}(d;g) = \cD_{\min}(d;g').
\end{equation*}
The form $\omega( \alpha(u_{\delta}), \beta(u_{\delta}) )$ is in $\cD_{\max}^{\reg}(d;g')$ with leading tangential term $\alpha(u_{\delta}),$ hence it is in $\cD_{\bB}^{\reg}(d;g'),$
and so from the description above we have $u \in \cD_{\bB}^{\reg}(d;g').$
Finally, the graph closure of $\cD_{\bB}^{\reg}(d)$ with respect to $g$ is the same as the graph closure with respect to $g',$ so we conclude that
\begin{equation*}
	\cD_{\bB}(d;g) = \cD_{\bB}(d;g').
\end{equation*}
\end{proof}

\subsection{General depth} \label{sec:GralDepthDeRham}
Now consider a general stratified space $\hat X$, with strata $Y^1, \ldots, Y^{k+1}$ and Cheeger ideal boundary conditions 
with mezzoperversity, $\bB$. Recall that for any $u \in \cD_{\max}(d)$ we denote $u_{\delta}$ the orthogonal projection 
onto $\bar{\delta(\cD_{\max}(\delta))},$ and that $u_{\delta}\in \cD_{\max}(\eth_{\dR}).$ Similarly given $v \in \cD_{\max}(d)$ 
we denote the orthogonal projection onto $\bar{d(\cD_{\max}(d))}$ by $v_d$ and this is an element of $\cD_{\max}(\eth_{\dR}).$ Define
\begin{equation*}
\begin{gathered}
	\cD_{\bB}^{\reg}(d) = \{ u \in \cD_{\max}(d) : u_{\delta} \in \cD_{\bB}^{\reg}(\eth_{\dR}) \}\\
	\cD_{\bB}^{\reg}(\delta) = \{ u \in \cD_{\max}(\delta) : u_d \in \cD_{\bB}^{\reg}(\eth_{\dR}) \}
\end{gathered}
\end{equation*}
and then set ${\cD_{\bB}(d)}$ and ${\cD_{\bB}(\delta)}$ equal to the closures in the respective graph norms. 
In this section we will show that 
$\cD_{\bB}(\eth_{\dR}) = \cD_{\bB}(d) \cap \cD_{\bB}(d)^*,$ so that the de Rham cohomology and Hodge cohomologies with Cheeger ideal boundary conditions coincide.
In the next section we will show that surprisingly the de Rham cohomology is actually independent of the choice of metric (among suitably scaled, rigid $\iie$ metrics).\\

For $w \in \cD_{\bB}(\eth_{\dR})$ we will denote the leading term at the stratum $Y^k$ by $\alpha_k(w) + dx\wedge \beta_k(w).$
We will make strong use of the boundary pairing $[\cdot,\cdot]_d$ from the previous section.

We start with $v \in \cD_{\bB}(d)^*$ which is equivalent to $v \in \cD_{\max}(\delta)$ and $[u,v]_d = 0$ for all $u \in \cD_{\bB}^{\reg}(d).$
We know that $v_d \in \cD_{\max}(\eth_{\dR})$ and so it has a distributional asymptotic expansion at $Y^1.$
Choosing $u \in \cD_{\bB}^{\reg}(d)$ supported in a distinguished neighborhood of a point $q \in Y^1$ we see that
\begin{equation*}
	[u,v]_d = \ang{\alpha_1(u_{\delta}), \beta_1(v_d) }_{Y^1}
\end{equation*}
and as before we can choose $\alpha_1(u_{\delta})$ arbitrarily so we must have
\begin{equation*}
	\cD_{\bB}(d)^* \subseteq 
	\cD_{\max, B^1}(\delta) = \{ v \in \cD_{\max}(\delta) : \beta_1(v_d) \text{ is a current on $Y^1$ with coefficients in $(W^1)^{\perp}$} \}.
\end{equation*}
We have a similar domain for $d$
\begin{equation*}
	\cD_{\max, B^1}(d) = \{ u \in \cD_{\max}(d) : \alpha_1(u_{\delta}) \text{ is a current on $Y^1$ with coefficients in $W^1$} \}
\end{equation*}
and since $(d, \cD_{\max, B^1}(d))$ is a closed operator extending $(d, \cD_{\bB}^{\reg}(d))$ we must have $\cD_{\bB}(d) \subseteq \cD_{\max, B^1}(d).$
Thus we have shown that to determine $\cD_{\bB}(d)^*,$ we can restrict the pairing $[\cdot,\cdot]_d$ to
\begin{equation*}
	\cD_{\max, B^1}(d) \times \cD_{\max, B^1}(\delta) \ni (u,v) \mapsto [u,v]_d \in \bbC.
\end{equation*}

\begin{lemma}\label{lem:DefMinDomB1}
In terms of
\begin{equation*}
	\dCI_{Y^1} = \dCI_{Y^1}(X;\Lambda^*(\Iie T^*X)) = \{ u \in \CIc(\hat X; \Lambda^*(\Iie T^*X)) : \supp u \subseteq X \cup Y^1 \}
\end{equation*}
the domains
\begin{equation*}
\begin{gathered}
	\cD_{\min,B^1}(\delta) = \{ v \in \cD_{\max, B^1}(\delta) : 
	\exists (v_n) \subseteq \dCI_{Y^1} \cap \cD_{\max, B^1}(\delta) \Mst v_n \to v, \Mand (\delta v_n) \text{ Cauchy} \}, \\
	\cD_{\min, B^1}(d) = \{ u \in \cD_{\max, B^1}(d) : 
	\exists (u_n) \subseteq \dCI_{Y^1} \cap \cD_{\max, B^1}(d) \Mst u_n \to u, \Mand (d u_n) \text{ Cauchy} \}.
\end{gathered}
\end{equation*}
are, respectively, equal to $\cD_{\max, B^1}(d)^*$ and $\cD_{\max, B^1}(\delta)^*.$
\end{lemma}

\begin{proof}
Notice that $\cD_{\min, B^1}(\delta)$ is a closed domain for $\delta$ since it is the graph closure of $\dCI_{Y^1} \cap \cD_{\max, B^1}(\delta).$
If $u \in \cD_{\max, B^1}^{\reg}(d)$ and $v \in \cD_{\min, B^1}(\delta)$ then 
\begin{equation*}
	[u,v]_d = \lim [u, v_n]_d = \lim \ang{ \alpha_1(u_\delta), \beta_1((v_n)_d) }_{Y^1} = 0
\end{equation*}
so $\cD_{\min, B^1}(\delta) \subseteq (\cD_{\max, B^1}(d))^*.$

On the other hand,if $u \in L^2(X; \Lambda^*(\Iie T^*X))$ is such that 
\begin{equation*}
	\dCI_{Y^1} \cap \cD_{\max, B^1}(\delta) \ni v \mapsto \ang{u, \delta v} \text{ is continuous }
\end{equation*}
then restricting to those $v$ supported in a distinguished neighborhood of $q \in Y^1$ we have a continuous map $v \mapsto \ang{du,v} + \ang{\alpha(u_\delta), \beta(v)}_Y$
Letting the $L^2$-norm of $v$ go to zero while keeping $\beta(v)$ fixed shows that we must have $\ang{\alpha(u_\delta), \beta(v)}_Y=0,$ and hence that $u \in \cD_{\max, B^1}(d).$ Thus we have shown that
\begin{equation*}
	\lrpar{ \delta, \dCI_{Y^1} \cap \cD_{\max, B^1}(\delta) }^* \subseteq (d, \cD_{\max, B^1}(d))
\end{equation*}
and so $(\cD_{\max, B^1}(d))^*$ is contained in the closure of $\lrpar{ \delta, \dCI_{Y^1} \cap \cD_{\max, B^1}(\delta) }$ which is $\cD_{\min, B^1}(\delta).$
Putting these together,
\begin{equation*}
	\cD_{\min, B^1}(\delta) = (\cD_{\max, B^1}(d))^*,
\end{equation*}
and reversing the roles of $d$ and $\delta$ yields
\begin{equation*}
	\cD_{\min, B^1}(d) = (\cD_{\max, B^1}(\delta))^*.
\end{equation*}
\end{proof}

Since $\cD_{\max, B^1}(d)$ is a closed domain we have
\begin{equation*}
	L^2(X; \Lambda^* \Iie T^*X) = \ker (\delta, \cD_{\min, B^1}(\delta)) \oplus \bar{ d(\cD_{\max, B^1}(d)) }.
\end{equation*}
Given $v \in \cD_{\max,B^1}(\delta)$ let us denote its projection onto the second summand by 
\begin{equation*}
	v_{d,B^1} \in \bar{ d(\cD_{\max,B^1}(d))}.
\end{equation*}
Notice that, for any $v \in \cD_{\max, B^1}(\delta),$ we have $\delta v = \delta v_{d, B^1}$ 
and $v_{d, B^1} \in \ker (d, \cD_{\max, B^1}(d))$ so $v_{d, B^1} \in \cD_{\max}(\eth_{\dR})$
and in particular $v_{d, B^1}$ has an asymptotic expansion at $Y^1.$
Since $v_{d, B^1}$ differs from $v_d$ by an element of $\ker(\delta, \cD_{\min, B^1}(\delta))$ we see that
\begin{equation*}
	\beta_1(v_{d,B^1}) \text{ is a current on $Y$ with coefficients in $(W^1)^{\perp}$}
\end{equation*}
and since $v_{d, B^1} \in \bar{ d(\cD_{\max,B^1}(d))},$ we have, using the analogue of Lemma \ref{lem:AlphaNabla} and that $W^1$ is a  flat bundle,
\begin{equation*}
	\alpha_1(v_{d,B^1}) \text{ is a current on $Y$ with coefficients in $(W^1).$}
\end{equation*}
Thus altogether we have shown the important fact that
\begin{equation*}
	v \in \cD_{\max,B^1}(\delta) \implies v_{d,B^1} \in \cD_{\max, B^1}(\eth_{\dR})
\end{equation*}
which shows that $v_{d,B^1}$ has a distributional asymptotic expansion at $Y^2.$

Let us similarly use the decomposition
\begin{equation*}
	L^2(X; \Lambda^* \Iie T^*X) = \ker (d, \cD_{\min, B^1}(d)) \oplus \bar{ \delta(\cD_{\max, B^1}(\delta)) }
\end{equation*}
to define $u_{\delta, B^1}$ to be the projection of $u \in \cD_{\max}(d)$ onto the second summand.
We have
\begin{equation*}
	u \in \cD_{\max, B^1}(d) \implies u_{\delta, B^1} \in \cD_{\max, B^1}(\eth_{\dR}).
\end{equation*}
The pairing $[\cdot,\cdot]_d,$ restricted to $\cD_{\max, B^1}(d) \times \cD_{\max, B^1}(\delta)$ satisfies
\begin{equation*}
	[u,v]_d = [u_\delta, v_d]_d = [ u_{\delta, B^1}, v_{d, B^1} ]_d.
\end{equation*}
If $u \in \cD_{\bB}^{\reg}(d)$ is supported in a distinguished neighborhood of $q \in Y^2$ then, just as in Lemma \ref{lem:StokesThmEdge}, Stokes' theorem shows that
\begin{equation*}
	[u,v]_d = \ang{ \alpha_2( u_{\delta, B^1} ), \beta_2( v_{d, B^1} ) }
\end{equation*}
and, as in Lemma \ref{lem:PrescribedForms}, we may pick $\alpha_2( u_{\delta, B^1} )$ arbitrarily among the smooth sections of $W^2,$ so we conclude that
\begin{multline*}
	\cD_{\bB}(d)^* \subseteq 
	\cD_{\max, (B^1,B^2)}(\delta) = \{ v \in \cD_{\max}(\delta) : 
	\beta_1(v_d) \text{ is a current on $Y^1$ with coefficients in $(W^1)^{\perp}$} \\ \Mand
	\beta_2(v_{d,B^1}) \text{ is a current on $Y^2$ with coefficients in $(W^2)^{\perp}$} \}.
\end{multline*}
We have a similar domain for $d,$
\begin{multline*}
	\cD_{\max, (B^1,B^2)}(d) = \{ u \in \cD_{\max}(d) : 
	\alpha_1(u_\delta) \text{ is a current on $Y^1$ with coefficients in $W^1$} \\ \Mand
	\alpha_2(u_{\delta,B^1}) \text{ is a current on $Y^2$ with coefficients in $W^2$} \},
\end{multline*}
and reasoning as above shows that
\begin{equation*}
	\cD_{\bB}(\delta)^* \subseteq \cD_{\max, (B^1,B^2)}(d).
\end{equation*}
$ $

We now continue inductively in this way to define 
\begin{equation*}
	\cD_{\max,(B^1, \ldots, B^k)}(d), \quad
	\cD_{\max,(B^1, \ldots, B^k)}(\delta)
\end{equation*}
for all $1 \leq k < T,$ and show that
\begin{equation*}
	\cD_{\bB}(d) \subseteq \cD_{\max,(B^1, \ldots, B^k)}(d), \quad
	\cD_{\bB}(d)^* \subseteq \cD_{\max, (B^1, \ldots, B^k)}(\delta).
\end{equation*}
Once this is done for the first $T-1$ strata, then the argument above shows 
\begin{equation*}
	\cD_{\bB}(d)^* = \cD_{\max, (B^1, \ldots, B^T)}(\delta)
\end{equation*}
since at this stage we have $v \in \cD_{\bB}(d)^*$ if and only if $[u,v]_d=0$ for all $u \in \cD_{\bB}^{\reg}(d)$ supported in a distinguished neighborhood of a point $q \in Y^T.$
This final domain has a useful core domain
\begin{multline*}
	\cE_{\bB}(\delta) = 
	\{ v \in \cD_{\max}(\delta) : 
	\beta_1(v_d) \text{ is a smooth form on $Y^1$ with coefficients in $(W^1)^{\perp}$}, \\
	\beta_2(v_{d,B^1}) \text{ is a smooth form on $Y^2$ with coefficients in $(W^2)^{\perp}$}, \\
	\ldots\\
	\beta_T(v_{d,(B^1,\ldots, B^{T-1})}) \text{ is a smooth form on $Y^T$ with coefficients in $(W^T)^{\perp}$}
	 \}.
\end{multline*}
Computations analogous to those carried out above show that
\begin{equation*}
	\cD_{\bB}(d) = \cE_{\bB}(\delta)^* = \cD_{\max, (B^1, \ldots, B^T)}(d)
\end{equation*}
and that $\cD_{\bB}(\delta) = \cD_{\max, (B^1, \ldots, B^T)}(\delta).$

Finally, directly from the definition of $\cD_{\bB}(\eth_{\dR})$ we can conclude that
\begin{equation*}
	\cD_{\bB}(d)\cap \cD_{\bB}(d)^* = \cD_{\bB}(\eth_{\dR}) \cap \cD_{\max}(d) \cap \cD_{\max}(\delta).
\end{equation*}
To see that the final two intersections are superfluous we can proceed as we did on spaces of depth one.
Indeed, the same argument we gave there shows that
\begin{equation*}
	\cD_{\max, B^1}^{\reg}(\eth_{\dR}) \subseteq \cD_{\max,B^1}(d) \cap \cD_{\max,B^1}(\delta)
\end{equation*}
and iterating we find
\begin{equation*}
	\cD_{\bB}^{\reg}(\eth_{\dR}) \subseteq \cD_{\bB}(d) \cap \cD_{\bB}(d)^*.
\end{equation*}
Now since the latter space is a closed domain for $\eth_{\dR}$ it must contain the closure of the former, namely $\cD_{\bB}(\eth_{\dR}).$

\begin{theorem}\label{thm:SelfDual}
If $(\hat X,g)$ is a stratified pseudomanifold with a suitably scaled $\iie$ metric and $W^i \lra Y^i$ flat bundles forming a mezzoperversity with associated Cheeger ideal boundary conditions $\bB,$
then $(d, \cD_{\bB}(d))$ forms a Hilbert complex with dual complex $(\delta, \cD_{\bB}(\delta))$ and associated de Rham operator
\begin{equation*}
	(\eth_{\dR}, \cD_{\bB}(\eth_{\dR}))
\end{equation*}
which is thus self-adjoint.
In particular, using Theorem \ref{Thm:MainHodgeThm}, this operator is Fredholm and so $(d, \cD_{\bB}(d))$ is a Fredholm complex with corresponding Kodaira decomposition and Hodge theorem:
\begin{equation}\label{Kodaira2}
\begin{gathered}
	L^2(X;\Lambda^* (\Iie T^*X) )
	= \cH^*_{\bB}(\hat X)  \oplus d(\cD_{\bB}(\eth_{\dR})) \oplus \delta(\cD_{\bB}(\eth_{\dR})) \\
	\cH^*_{\bB}(\hat X) \cong \tH^*_{\bB}(\hat X).
\end{gathered}
\end{equation}
\end{theorem}

\subsection{Metric independence}

Let $(\hat X,g)$ be a stratified pseudomanifold with a rigid, suitably scaled $\iie$ metric, with Hodge mezzoperversity 
\begin{equation*}
	\cW = (W^1 \lra Y^1, \ldots, W^T \lra Y^{T})
\end{equation*}
as in Definition \ref{Def:FlatSystem}, with associated Cheeger ideal boundary conditions $\bB.$
Recall that the definition of the mezzoperversity is that
\[
	W^j \text{ is a flat subbundle of }\cH^{\tfrac12 \dim Z^j}(H^j/Y^j), \qquad j = 1, \ldots,T. 
\]

We now want to show that the de Rham cohomology is independent of the choice of appropriate $\iie$ metric. 
To state this properly, we need to rephrase the boundary conditions slightly since the Hodge cohomology 
will not be independent of the choice of metric.
\begin{definition}
Let $(\hat X,g)$ be a stratified pseudomanifold, as above. 
A {\bf de Rham mezzoperversity} consists of a list of bundles
\begin{equation*}
	\cW = \{ W^1 \lra Y^1, \ldots, W^T \lra Y^T \}
\end{equation*}
with
\[
	W^j \text{ flat subbundle of }\tH^{\tfrac12 \dim Z^j}_{W^1, \ldots, W^{j-1}}(H^j/Y^j), \quad j = 1, \ldots, T.
\]
\end{definition}
Given a metric $g$, a Hodge mezzoperversity $\cW$ is equivalent to a de Rham mezzoperversity; we denote this by $[\cW],$ 
since we have shown that the Hodge cohomology groups coincide with the de Rham cohomology groups. The following result, 
using the notation from \S \ref{sec:GralDepthDeRham}, allows us to identify domains defined using different metrics.

\begin{lemma}\label{lem:CoefIsExact}
Let $(\hat X, g, \bB)$ be a stratified pseudomanifold with a suitably scaled rigid $\iie$ metric and Hodge mezzoperversity.
Let $x$ be a boundary defining function for $Y^{k+1},$ $Z$ the link of $\hat X$ at $Y^{k+1}$ and $f = \dim Z.$

If $w \in \ker( d, \cD_{\min, (B^1, \ldots, B^k)}(d))$ has an expansion at $Y^{k+1}$ with leading term 
\begin{equation*}
	x^{-f/2}(\alpha(w) + dx \wedge \beta(w)), \quad \Mwhere \alpha(w), \beta(w) \in \ker d^Z
\end{equation*}
then $\alpha(w) \in d^Z(\cD_{\bB(Z)}(d^Z)).$
\end{lemma}

\begin{proof}
Consider the pairing
\begin{equation*}
	[\cdot,\cdot]_{d, (B^1, \ldots, B^k)}: \cD_{\max, (B^1, \ldots, B^k)}(d) \times \cD_{\max, (B^1, \ldots, B^k)}(\delta) \lra \bbC
\end{equation*}
given as usual by $(u,v)\mapsto \ang{du,v}-\ang{u,\delta v}.$
Since $\ker( d, \cD_{\min, (B^1, \ldots, B^k)}(d))$ is orthogonal to the image of $\delta,$ $\delta(\cD_{\max, (B^1, \ldots, B^k)}(\delta)),$ we have
\begin{equation*}
	[w,\cdot]_{d, (B^1, \ldots, B^k)} = 0.
\end{equation*}

If $v \in \cD_{\max, (B^1, \ldots, B^k)}^{\reg}(\delta)$ is supported in a distinguished neighborhood of $q \in Y^{k+1}$ then we have
\begin{equation*}
	0 = [w,v]_{d, (B^1, \ldots, B^k)} = \lim_{\eps \to 0} \ang{w,v}_{\eps} 
	= \ang{ \alpha(w), \beta(v_{d, (B^1, \ldots, B^k)}) }_{(H,g)}.
\end{equation*}
(Note that since $\alpha(w)$ is not a harmonic form, this pairing does not descend to $Y^{k+1}.$)
As this vanishes for all $v$ and we can prescribe $\beta(v_{d, (B^1, \ldots, B^k)})$ arbitrarily -- among $\iie$-forms on $Y$ with coefficients in the harmonic forms on $Z$ satisfying the boundary conditions $(B^1, \ldots, B^k)$ -- we can conclude that $\alpha(w)$ is a closed form perpendicular to harmonic forms, i.e., exact:
\begin{equation*}
	\alpha(w) \in d^Z(\cD_{\bB(Z)}(d^Z)),
\end{equation*}
since $d^Z(\cD_{\bB(Z)}(d^Z))$ is closed.
\end{proof}

\begin{theorem}\label{thm:MetricIndep}
Let $\hat X$ be a stratified pseudomanifold as above. 
Let $g$ and $g'$ are suitably scaled rigid $\iie$ metrics on $\hat X$ with Hodge mezzoperversities $\cW$ and $\cW',$ 
respectively, so that $\cW$ and $\cW'$ are both equivalent to the same de Rham mezzoperversity, $[\cW]$. Then  
\begin{equation*}
	\cD_{\cW}(d;g) = \cD_{\cW'}(d;g'), \quad \Mand \quad
	\cH^*_{\cW}(\hat X, g) \cong \cH^*_{\cW'}(\hat X, g').
\end{equation*}

Thus the de Rham cohomology groups $\tH^*_{[\cW]}(\hat X)$ are independent of the choice of suitably scaled rigid $\iie$ metric.
\end{theorem}

\begin{proof}
It suffices to prove that $\cD_{\cW}(d;g) = \cD_{\cW'}(d;g'),$ as this implies that the de Rham cohomologies are isomorphic and hence so are the Hodge cohomologies.\\

If $\hat X$ is a space of depth one, then the result was proven in Proposition \ref{prop:DeRhamCohoSimpleEdge}. We now inductively carry out the same argument in the general case.

Since $g$ and $g'$ are quasi-isometric, we have
\begin{equation*}
	\cD_{\max}(d,g) = \cD_{\max}(d,g') = \cD_{\max}(d), \quad 
	\cD_{\min}(d,g) = \cD_{\min}(d,g') = \cD_{\min}(d).
\end{equation*}
Given $u \in \cD_{\max}(d)$ its projection $u_{\delta}$ from \eqref{eq:DefuDelta} will be different for $g$ or $g',$ so we denote them by $u_{\delta_g}$ and $u_{\delta_{g'}}$ respectively. We know that $u-u_{\delta_g}$ and $u - u_{\delta_{g'}}$ are both in $\ker (d, \cD_{\min}(d)),$ hence so is
\begin{equation*}
	w_{g,g'} = u_{\delta_g} - u_{\delta_{g'}}.
\end{equation*}
Since both $u_{\delta_g}$ and $u_{\delta_{g'}}$ have asymptotic expansions at $Y^1$ of the sort required in Lemma \ref{lem:CoefIsExact}, we can conclude that
\begin{equation*}
	\alpha_1(w_{g,g'}) \in d^{Z^1}( \cD(d^{Z^1}) )
\end{equation*}
and so we have the crucial fact 
\begin{equation*}
	[\alpha(u_{\delta_g})] = [\alpha(u_{\delta_{g'}})] \Min \tH^{\mid}(H^1/Y^1).
\end{equation*}
Note that the Hodge theorem allows us to write
\begin{equation*}
	\cD_{\max, B^1}(d;g) = \{ u \in \cD_{\max}(d) : [\alpha_1(u_{\delta_g})] \text{ is a current on $Y^1$ with coefficients in $[W^1]$} \} 
\end{equation*}
and so our observation implies
\begin{equation*}
	\cD_{\max, B^1}(d;g) = \cD_{\max, B^1}(d;g') = \cD_{\max, B^1}(d).
\end{equation*}
Then from Lemma \ref{lem:DefMinDomB1} it follows that
\begin{equation*}
	\cD_{\min, B^1}(d;g) = \cD_{\min, B^1}(d;g') = \cD_{\min, B^1}(d).
\end{equation*}

Inductively the same argument shows that once we know 
\begin{equation*}
	\cD_{\max, (B^1, \ldots, B^k)}(d;g) = \cD_{\max, (B^1, \ldots, B^k)}(d;g') = \cD_{\max, (B^1, \ldots, B^k)}(d)
\end{equation*}
then we see that
\begin{equation*}
	\cD_{\min, (B^1, \ldots, B^k)}(d;g) = \cD_{\min, (B^1, \ldots, B^k)}(d;g') = \cD_{\min, (B^1, \ldots, B^k)}(d)
\end{equation*}
and from Lemma \ref{lem:CoefIsExact} that
\begin{equation*}
	[\alpha(u_{\delta_g, (B^1, \ldots, B^k)})]
	= [\alpha(u_{\delta_{g'}, (B^1, \ldots, B^k)})]
	\Min \tH^{\mid}_{B^1, \ldots, B^k}(H^{k+1}/Y^{k+1})
\end{equation*}
for all $u \in \cD_{\max, (B^1, \ldots, B^k)}(d).$
This in turn shows that
\begin{equation*}
	\cD_{\max, (B^1, \ldots, B^k, B^{k+1})}(d;g) = \cD_{\max, (B^1, \ldots, B^k, B^{k+1})}(d;g') = \cD_{\max, (B^1, \ldots, B^k, B^{k+1})}(d)
\end{equation*}
completing the induction and the proof.
\end{proof}

An easy corollary is the invariance of de Rham cohomology under stratified diffeomorphism.
We show in \cite{ALMP13.2} the much less straightforward invariance under stratified homotopy equivalence.

Let $F: \wt X \lra \wt M$ be a stratified diffeomorphism with inverse map $G: \wt M \lra \wt X.$ 
Recall that this means that $F, G$ are smooth maps intertwining the boundary fibration structures and mutually inverse to one another.
Since $F$ intertwines the boundary fibration structures, it defines maps between the $\iie$-tangent bundles and the $\iie$-differential forms,
\begin{equation*}
\begin{gathered}
	DF: \Iie T\wt X \lra \Iie T\wt M\\
	F^*: 
	\CIc(\wt M; \Lambda^*(\Iie T^*\wt M)) \lra
	\CIc(\wt X; \Lambda^*(\Iie T^*\wt X)).
\end{gathered}
\end{equation*}
If we endow $\wt M$ with an $\iie$ metric, the first map can be used to define an $\iie$ metric on $\wt X,$ and with respect to these metrics the second map is an $L^2$-isometry. Since all $\iie$ metrics are quasi-isometric this shows that, regardless of which $\iie$ metrics are used, $F^*$ extends to a bounded isomorphism
\begin{equation*}
	F^*:
	L^2(\wt M; \Lambda^*(\Iie T^*\wt M)) \lra
	L^2(\wt X; \Lambda^*(\Iie T^*\wt X)).
\end{equation*}
Since $F^*$ intertwines the exterior derivatives, it defines isomorphisms
\begin{equation*}
	F^*:\cD_{\min}(d_M) \lra \cD_{\min}(d_X), \quad
	F^*:\cD_{\max}(d_M) \lra \cD_{\max}(d_X).
\end{equation*}
If $\hat X$ (and hence $\hat M$) are Witt, then it follows that $F^*$ descends to an isomorphism in cohomology.

\begin{theorem}\label{thm:StratDiffInv}
A stratified diffeomorphism  $F: \wt X \lra \wt M$ induces by pull-back a bijection between de Rham mezzoperversities,
\begin{equation*}
	\cW_M \leftrightarrow F^*\cW_M = \cW_X
\end{equation*}
and corresponding isomorphisms between de Rham cohomology groups,
\begin{equation*}
	F^*:\tH_{\cW_M}^*(\hat M) \lra \tH_{F^*\cW_M}(\hat X).
\end{equation*}
\end{theorem}

\begin{proof}
We induct over the depth of $\hat X,$ with the base case being smooth manifolds.

Suppose $\hat X$ has depth $k$ and we have established the theorem for space of depth less than $k.$
Let $\cW_M = (\cW_M', W_M(k))$ be a mezzo-perversity over $M$ with $\cW_M'$ the restriction to strata of depth less than $k.$
Let $Y\subseteq \hat X$ be a stratum of depth $k$ with link $Z,$ and let $N$ be the corresponding stratum of $\hat M$ with link $R.$
Let $H_Y$ and $H_N$ denote the corresponding boundary hypersurfaces, so that $F$ restricts to a fiber bundle map
\begin{equation*}
	\xymatrix 
	{H_Y \ar[d]^{\phi_Y} \ar[r]^{F_H} & H_N \ar[d]^{\phi_N} \\
	Y \ar[r]^{F_Y} & N }
\end{equation*}
and denote by $F_{Z_q}: \wt Z_q \lra \wt R_{F_Y(q)}$ the induced map on the fibers.
By inductive hypothesis each $F_{Z_q}$ induces via pull-back a mezzoperversity $\cW_X'$ on $Z_q$ and an isomorphism between de Rham cohomology groups
\begin{equation*}
	F_{Z_q}^*:\tH_{\cW_M'}^*(\hat R_{F_Y(q)}) \lra \tH_{\cW_X'}(\hat Z_q).
\end{equation*}
These maps fit together to define a vector bundle map
\begin{equation*}
	F_{H/Y}^*:\tH_{\cW_M'}^*(H_N/N) \lra \tH_{\cW_X'}(H_Y/Y)
\end{equation*}
which, because $F_H^*$ intertwines the exterior derivatives on $H_Y$ and $H_N,$ intertwines the flat connections.
Thus the pull-back of the flat bundle $W_M(k)$ defines a flat subbundle $F_{H/Y}^*W_M(k) = W_X(k),$ and thus pull-back by $F$ of the mezzoperversity $\cW_M$ determines a mezzoperversity $\cW_X.$

By localization and inductive hypothesis, pull-back by $F$ restricts to an isomorphism
\begin{equation*}
	F^*:\cD_{\max, \cW_M'}(d_M) \lra \cD_{\max, \cW_X'}(d_X).
\end{equation*}
Since the flat bundles $W_M(k)$ and $W_X(k)$ are related by pull-back by $F_{H/Y},$ we see that
\begin{equation*}
	F^*:\cD_{\cW_M}^{\reg}(d_M) \lra \cD_{\cW_X}^{\reg}(d_X).
\end{equation*}
By continuity of $F$ this must extend to the closure, $\cD_{W_M}(d_M),$ and since $F$ has an inverse this extension must be an isomorphism
\begin{equation*}
	F^*:\cD_{\cW_M}(d_M) \lra \cD_{\cW_X}(d_X).
\end{equation*}
It follows that the induced map on cohomology is an isomorphism, as required.
\end{proof}

\section{Generalized Poincar\'e Duality} \label{sec:SignLag} 
On any smooth, oriented, closed manifold $M,$ there is a natural intersection pairing
\begin{equation*}
	Q: \CI(M;\Lambda^*M) \times \CI(M;\Lambda^*M) \lra \bbR, \quad
	Q(\eta, \omega) = \int \eta \wedge \omega;
\end{equation*}
the signature of this pairing is by definition the signature of the manifold. Note that by Stokes' theorem
\begin{equation*}
	Q(d\eta, \omega) = \pm Q(\eta, d\omega)
\end{equation*}
and so if one of $\eta,$ $\omega$ is closed and the other exact then $Q(\eta, \omega)=0.$
Thus $Q$ descends to a map
\begin{equation*}
	Q: \tH^*(M) \times \tH^*(M) \lra \bbR.
\end{equation*}
Poincar\'e duality on $M$ is the assertion that this quadratic form is non-degenerate.
If we endow $M$ with a Riemannian metric $g$ and denote the Hodge star by $*,$ then non-degeneracy is immediate from
\begin{equation*}
	Q(\eta, *\eta) = \norm{\eta}^2_{L^2}.
\end{equation*}

The generalized Poincar\'e duality of Goresky-MacPherson \cite{GM1} is, for every stratified pseudomanifold, a non-degenerate map
\begin{equation*}
	\tI_{\bar p}\tH^*(\hat X) \times \tI_{\bar q}\tH^*(\hat X) \lra \bbR
\end{equation*}
induced by the usual intersection product, but involving two dual perversities $\bar p$ and $\bar q.$
In \cite{GM2} Goresky-MacPherson adopted a sheaf hypercohomology approach and showed that this generalized Poincar\'e duality 
can be deduced from Verdier duality. For Witt spaces, the intersection cohomology groups with `lower middle'
$\bar m$ and `upper middle' $\bar n$ perversities coincide, so the generalized Poincar\'e duality statement involves 
only a single sequence of groups. Cheeger \cite{Cheeger:Hodge} showed that Poincar\'e duality could be realized as the same integration 
map $Q$ on $L^2$ differential forms for an $\iie$ metric. For non-Witt spaces with conic singularities, he also 
indicated  \cite{Cheeger:Conic}  how to incorporate ideal boundary conditions to refine Poincar\'e duality. The sheaf approach has 
been extended by Banagl \cite{BanaglShort} to a refined Poincar\'e duality using Verdier duality.

In this section we generalize Cheeger's analytic approach to generalized Poincar\'e duality results to non-Witt spaces. 
Given a mezzoperversity $\cW,$ we introduce a dual mezzoperversity $\sD\cW$ and show that there is a non-degenerate 
intersection pairing 
\begin{equation*}
	Q: \tH^*_{\cW}(\hat X) \times \tH^*_{\sD\cW}(\hat X) \lra \bbR
\end{equation*}
realizing a refined generalized Poincar\'e duality.
In companion papers \cite{ALMP13.2, ABLMP} we will show that this is consistent with Banagl's sheaf theoretic approach.

\subsection{Dual mezzoperversities} 
Recall from the end of \S\ref{sec:SimpleEdgeBC} the construction of the natural flat connection on the vertical cohomology of a fibration of closed manifolds.
If we assume that $M$ is the total space of a fibration $F - M \xlra{\phi} B$ with all spaces closed and oriented, then we have a natural push-forward
\begin{equation*}
	\phi_*: \CI(M, \Lambda^*T^*M) \lra \CI(B;\Lambda^{*-f}T^*B)
\end{equation*}
given by integration over the fibers.
This factors as 
\begin{equation*}
	\CI(M, \Lambda^*T^*M) \lra \CI(M; \Lambda^*\phi^*(T^*B) \hat\otimes \Lambda^fT^*M/B) \lra \CI(B;\Lambda^*T^*B),
\end{equation*}
and satisfies
\begin{equation}\label{eq:PropPhi}
	d_B \phi_* = \phi_* d_M, \quad \phi_* d_{M/B} = 0. 
\end{equation}
We can use $\phi_*$ to define the family intersection pairing
\begin{equation*}
	Q_{M/B}:\CI(B;\tH^*(M/B)) \times \CI(B;\tH^*(M/B)) \lra \CI(B;\Lambda^0T^*B) = \CI(B)
\end{equation*}
by using the surjection
\begin{equation*}
	\CI(M;\Lambda^*M) \cap \ker d_{M/B} \xlra{\psi} \CI(B;\Lambda^*T^*B \hat \otimes \tH^*(M/B))
\end{equation*}
and then setting
\begin{equation*}
	Q_{M/B}(\psi(\eta), \psi(\omega)) = \phi_*( \eta \wedge \omega ).
\end{equation*}
Notice that this pairing is non-degenerate and from \eqref{eq:PropPhi} that
\begin{equation*}
	d_BQ_{M/B}(\cdot, \cdot) = Q_{M/B}(\nabla^{\tH}\cdot, \cdot) + Q_{M/B}(\cdot, \nabla^{\tH}\cdot).
\end{equation*}
In particular this shows that whenever $W \subseteq \tH^*(M/B)$ is a flat sub-bundle then so is $W^{\perp_Q},$ the $Q_{M/B}$-orthogonal complement of $W.$\\

At the end of \S\ref{sec:TwoEdgeBC} we discussed how one can adapt the construction of the connection $\nabla^{\tH}$ 
to the setting of $L^2$ cohomology with ideal boundary conditions. That discussion extends to this context and shows that 
\begin{multline*}
	W \text{ flat subbundle of } \tH^{\mid}_{W^1, \ldots, W^k}(H^{k+1}/Y^{k+1}) \\
	\implies
	W^{\perp_Q} \text{ flat subbundle of } \tH^{\mid}_{W^1, \ldots, W^k}(H^{k+1}/Y^{k+1}).
\end{multline*}
Notice that we can rewrite $Q_{H/Y}$ in terms of the vertical Hodge star $*=*_{H/Y}$ and the vertical $L^2$-inner product, namely 
\begin{equation*}
	Q(\eta, \omega) = \ang{\eta, *\omega}_{L^2}
\end{equation*}
so 
\begin{equation*}
	W^{\perp_Q} = * W^{\perp}.
\end{equation*}
$ $

\begin{proposition}\label{prop:DualMezzo}
Let $(\hat X,g, \bB)$ be a stratified pseudomanifold with a suitably scaled $\iie$-metric and Cheeger ideal boundary conditions $\bB$ corresponding to a Hodge mezzoperversity 
\begin{equation*}
	\cW = \{ W^j \lra Y^j \},
\end{equation*}
and for each $j,$ let $\sD W^j$ be the $Q_{H^j/Y^j}$-orthogonal complement of $W^j.$
Then 
\begin{equation*}
	\sD\cW = \{ \sD W^j \lra Y^j \}
\end{equation*}
is a Hodge mezzoperversity and if we denote the corresponding Cheeger ideal boundary conditions by $\sD\bB$ then the Hodge star $*$ defines bounded involutions
\begin{equation}\label{eq:HodgeStarInv}
	*: \cD_{\bB}(\eth_{\dR}) \lra \cD_{\sD \bB}(\eth_{\dR}), \quad
	*: \cD_{\bB}(d) \lra \cD_{\sD \bB}(d).
\end{equation}
\end{proposition}

\begin{proof}
The Hodge star of $g$ defines a unitary involution
\begin{equation*}
	*: L^2(X; \Lambda^*(\Iie T^*X)) \lra L^2(X; \Lambda^*(\Iie T^*X)) 
\end{equation*}
that for any $\omega \in \cD_{\max}(d) \cap \cD_{\max}(\delta)$ satisfies
\begin{equation*}
	d*\omega = \pm *\delta \omega, \quad
	\delta *\omega = \pm * d \omega.
\end{equation*}
For any subbundle $W$ of the vertical harmonic forms of $H \lra Y,$ there is an induced map
\begin{equation*}
	*: 
	\CI(Y; \Lambda^*Y \otimes W) \oplus dx \wedge \CI(Y; \Lambda^*Y \otimes W^{\perp}) 
	\lra
	\CI(Y; \Lambda^*Y \otimes *_{H/Y}W^{\perp}) \oplus dx \wedge \CI(Y; \Lambda^*Y \otimes *_{H/Y}W) 
\end{equation*}
by identifying these forms with forms on $H$ in the kernel of the vertical de Rham operator. Denoting $\sD W$ the $Q_{H/Y}$-orthogonal complement of $W,$ this map is
\begin{equation*}
	*: 
	\CI(Y; \Lambda^*Y \otimes W) \oplus dx \wedge \CI(Y; \Lambda^*Y \otimes W^{\perp}) 
	\lra
	\CI(Y; \Lambda^*Y \otimes \sD W) \oplus dx \wedge \CI(Y; \Lambda^*Y \otimes \sD W^{\perp}) 
\end{equation*}

Now if $\hat X$ has a single singular stratum, this is enough to prove the proposition.
Indeed, clearly $\sD \cW$ is a mezzoperversity and $u \in \cD_{\bB}^{\reg}(\eth_{\dR})$ if and only if  $*u \in \cD_{\sD\bB}^{\reg}(\eth_{\dR}),$
which suffices since $\cD_{\bB}(\eth_{\dR})$ and $\cD_{\bB}(d)$ are the graph closures of $\cD_{\bB}^{\reg}(\eth_{\dR}).$

If $\hat X$ has two singular strata and they are both non-Witt, then 
\begin{equation*}
	\sD W^1 \text{ is a flat subbundle of } \cH^{\mid}(H^1/Y^1) \lra Y^1.
\end{equation*}
Applying the proposition to the typical link at a point in $Y^2,$ $Z^2,$ a stratified space with a single singular stratum, we see that $\sD W^2$ is a flat subbundle of $\cH^{\mid}_{\sD W^1}(H^2/Y^2) \lra Y^2.$ Thus $\sD \cW$ is a mezzoperversity and the same argument as above shows that $u \in \cD_{\bB}^{\reg}(\eth_{\dR})$ if and only if  $*u \in \cD_{\sD\bB}^{\reg}(\eth_{\dR}),$ and hence \eqref{eq:HodgeStarInv} holds. The case of more strata is proven inductively in the same way.
\end{proof}

\begin{definition}
Let $(\hat X,g, \bB)$ be a stratified pseudomanifold with a suitably scaled $\iie$-metric and Cheeger ideal boundary conditions $\bB$ corresponding to a Hodge mezzoperversity $\cW.$ The mezzoperversity $\sD \cW$ of Proposition \ref{prop:DualMezzo} is known as the {\bf dual Hodge mezzoperversity} to $\cW.$ The de Rham mezzoperversity $[\sD \cW]$ is the {\bf dual de Rham mezzoperversity} of $[\cW].$
\end{definition}

Note that the dual de Rham mezzoperversity is independent of the choice of metric.

\begin{corollary}
Let $(\hat X,g, \bB)$ be an oriented stratified pseudomanifold with a suitably scaled $\iie$-metric and Cheeger ideal boundary conditions $\bB$ corresponding to a Hodge mezzoperversity $\cW$ with duals $\sD \bB$ and $\sD \cW.$
The intersection pairing
\begin{equation*}
	Q : \cD_{\bB}(d) \times \cD_{\sD \bB}(d) \lra \bbR, \quad Q(\eta, \omega) = \int_X \eta \wedge \omega
\end{equation*}
is non-degenerate and descends to a non-degenerate `generalized Poincar\'e duality' pairing
\begin{equation}\label{eq:PDPairing}
	Q: \tH_{\cW}^*(\hat X) \times \tH_{\sD\cW}^*(\hat X) \lra \bbR.
\end{equation}
\end{corollary}

\begin{proof}
At the level of differential forms the statement follows easily from $Q(\eta, \omega) = \ang{\eta, *\omega}_{L^2}$ with both $\eta$ and $*\omega$ in $\cD_{\bB}(d).$
To see that the pairing descends to cohomology, note that if both $\eta$ and $\omega$ are regular then by Stokes' theorem for manifolds with corners
\begin{equation*}
	\int_X d(\eta \wedge \omega) = \sum \int_{H_j} \alpha_j(\eta) \wedge \alpha_j(\omega)
\end{equation*}
and each of these summands vanishes since $\sD W^j = (W^j)^{\perp_Q}.$ By continuity in the graph norm, it follows that $\int_X d(\eta\wedge\omega)=0$ for all pairs in $\cD_{\bB}(d) \times \cD_{\sD \bB}(d),$ and hence $Q(\eta, \omega)$ vanishes if one of the forms is closed and the other exact.  
\end{proof}

\subsection{Cheeger spaces}

The generalized Poincar\'e duality of Goresky-MacPherson is most interesting when the dual perversities $\bar p$ and $\bar q$ coincide, i.e., on Witt spaces.
Similarly our generalized Poincar\'e duality is most interesting when the dual mezzoperversities $\cW$ and $\sD \cW$ coincide.

\begin{definition}
We say that a de Rham mezzoperversity is {\bf self-dual}  if it coincides with its dual mezzoperversity.
If $\hat X$ is a stratified pseudomanifold that carries a self-dual mezzoperversity we say that $\hat X$ is a {\bf Cheeger space.}
\end{definition}

\begin{remark}
Not all stratified pseudomanifolds are Cheeger spaces. 
Indeed, $W^j \lra Y^j$ is isomorphic to $\sD W^j \lra Y^j$ precisely when 
the vertical Hodge star sends $W^j$ to its orthogonal complement with respect to the induced vertical metrics.
This implies that the vertical family of signature operators on $H^j \lra Y^j$ has vanishing families index.
In particular the individual links $Z^j$ must have vanishing index.\\

If $\hat X$ has only conic singularities, then it is a Cheeger space if and only if the links all have vanishing signature.
\end{remark}

On a Cheeger space we can refine the de Rham operator to the signature operator by restricting to self-dual forms.
If $\hat X$ is even-dimensional, the Hodge star induces a natural involution on the differential forms on $X,$
\begin{equation*}
	\cI: \CIc(X; \Lambda^*T^*X) \lra \CIc(X;\Lambda^*T^*X), \quad \cI^2 =\Id
\end{equation*}
that extends to $L^2$ $\iie$ forms
\begin{equation*}
	\cI: L^2(X; \Lambda^*(\Iie T^*X)) \lra \CIc(X;\Lambda^*(\Iie T^*X)), \quad \cI^2 =\Id.
\end{equation*}
The $+1,$ $-1$ eigenspaces are known as the self-dual and anti-self-dual forms and denoted by
\begin{equation*}
	L^2(X; \Lambda_+^*(\Iie T^*X)), \quad
	\Mand
	L^2(X; \Lambda_-^*(\Iie T^*X)).
\end{equation*}
The signature operator $\eth_{\sign}^+$ is the de Rham operator $d+\delta$ as an unbounded operator between self-dual and anti-self-dual forms.
Its formal adjoint is $\eth_{\sign}^-,$ the de Rham operator $d+\delta$ as an unbounded operator between anti-self-dual and self-dual forms.
If $\cW$ is a self-dual mezzoperversity and $\bB$ are the associated Cheeger ideal boundary conditions, we define
\begin{equation*}
	\cD_{\bB}(\eth_{\sign}^{\pm}) = \cD_{\bB}(\eth_{\dR}) \cap L^2(X; \Lambda_\pm^*(\Iie T^*X)).
\end{equation*}

\begin{theorem}
Let $(\hat X, g, \bB)$ be a stratified space endowed with a suitably scaled $\iie$ metric and Cheeger ideal boundary conditions $\bB$ corresponding to a self-dual mezzoperversity $\cW.$
The signature operator 
\begin{equation*}
	\eth_{\sign}^+: \cD_{\bB}(\eth_{\sign}^+) \subseteq L^2(X; \Lambda_+^*(\Iie T^*X)) \lra L^2(X; \Lambda_-^*(\Iie T^*X))
\end{equation*}
is closed and Fredholm. Its adjoint is $(\eth_{\sign}^-, \cD_{\bB}(\eth_{\sign}^-))$ and 
its Fredholm index is the signature of the generalized Poincar\'e duality quadratic form $Q$ from \eqref{eq:PDPairing}.
\end{theorem}

\section{Local Poincar\'e Lemma}
In order to relate the cohomology of the complex $(d, \cD_{\bB}(d))$ to the hypercohomology of a complex of sheaves, as 
we do in the companion paper \cite{ABLMP}, we must compute the cohomology of a distinguished neighborhood 
with respect to ideal boundary conditions. So assume that $\cU_q \cong \mathbb B^h \times [0,1)_x \times Z_q = \mathbb B \times
C(Z_q)$ is a distinguished neighborhood, with trivialized vertical cohomology with boundary conditions, (consistent with the 
flat structure), $\cU_q \times \cH_{\bB(Z_q)}(Z_q)$, and flat subbundle $\cU_q \times W(Z_q),$; 
we compute $\Ht^*(d\rest{\cU_q}, \cD_{\bB}(d\rest{\cU_q})).$

A mezzoperversity for $\hat X$ induces a mezzoperversity on each link $Z_y,$ and hence ideal boundary conditions 
$(\bB_y, B_y^{k+1})$ on $\{ y \} \times C(Z)$ for every $y \in \bbB.$ The cohomology of the resulting complex is given 
by the expected formula:
\begin{proposition}
The cohomology 
$\Ht^k(d_{C(Z_q)}, \cD_{(\bB_y, B_y^{k+1})}(d_{C(Z_q)}))$ is independent of the point $y \in \bbB$ and is given by
\begin{equation*}
	\Ht^k(d_{C(Z_q)}, \cD_{(\bB_y,B^{k+1}_y)}(d_{C(Z_q)})) 
	=\begin{cases}
	\Ht^k_{\bB(Z_q)}(Z_q) & \Mif k < \tfrac12\dim Z \\
	W(Z_q) & \Mif k = \tfrac12\dim Z\\
	0 & \Mif k >\tfrac12 \dim Z
	\end{cases}
\end{equation*}
\end{proposition}

The proof follows the usual computation, as in \cite{Cheeger:Hodge} and \cite[\S 3.5]{Hunsicker-Mazzeo}. 
Note that Cheeger defines certain operators on the maximal domain of the exterior derivative; we point out that,
using notation as in \cite{Cheeger:Hodge}, 
\begin{equation*}
\begin{gathered}
	\pi^*: \cD_{\bB(Z)}(d^Z)_{[j]\leq f/2} \lra \cD_{\max, \bB(Z)}(d^{C(Z)})_{[j]\leq f/2} \\
	w\in \cD_{\max, \bB(Z)}(d^{C(Z)}) \implies Kw = \pi_*(w) \Mor \pi_*(\chi w) \in \cD_{\bB(Z)}(d^{C(Z)})
\end{gathered}
\end{equation*}
hence Cheeger's proof carries over directly. 

Finally, observe that one can reduce the cohomology $\Ht^*(d\rest{\cU_q}, \cD_{\bB}(d\rest{\cU_q}))$ to 
$\Ht^k(d_{C(Z)}, \cD_{\bB}(d_{C(Z)})).$ This requires a bit of care since, even though $\bbB$ is contractible, the 
boundary conditions vary with $y \in \bbB.$ However the natural filtration by vertical degree with respect to 
$\pi: \cU_q \lra \bbB$ shows that the cohomology $\Ht^*(d\rest{\cU_q}, \cD_{\bB}(d\rest{\cU_q}))$ can be computed 
by first passing to $\Omega^*(\bbB; \Ht^*(d_{C(Z)}, \cD_{(\bB,B^{k+1})}(d_{C(Z)})) )$ (indeed this is the $E_1^{*,*}$ term in 
the Serre spectral sequence). Since this is now a trivial local system, we can use the contractibility of $\bbB$ to 
see that this reduces to $\Ht^k(d_{C(Z)}, \cD_{\bB}(d_{C(Z)})).$  One can proceed alternately 
as in \cite[Corollary 2.15]{Bruning-Lesch}. In any case, this gives 
\begin{equation*}
	\Ht^*(d\rest{\cU_q}, \cD_{\bB}(d\rest{\cU_q})) = \Ht^*(\bbB) \otimes \Ht^*(d_{C(Z)}, \cD_{\bB}(d_{C(Z)})).
\end{equation*}

\end{document}